\documentclass[10pt]{article}
\pdfoutput=1

\usepackage{amsfonts,amsmath,amssymb}
\usepackage{mathtools}
\usepackage{stackrel}
\usepackage{leftidx}
\usepackage[amsmath,amsthm,thmmarks]{ntheorem}

\usepackage{authblk}

\usepackage{euscript}
\usepackage{pxfonts}
\usepackage{dsfont}
\usepackage{bbold}
\usepackage{yfonts}
\usepackage{scalerel}

\usepackage{tabularx}

\usepackage[shortlabels]{enumitem}

\usepackage{hyperref}
\usepackage{cleveref}

\usepackage{xcolor}
\usepackage{tcolorbox}
\usepackage{graphicx}
\usepackage{subfigure}

\usepackage[all,2cell,arrow,matrix]{xy} \UseAllTwocells \SilentMatrices
\usepackage{tikz-cd}
\usepackage{tikz}
\usetikzlibrary{arrows,arrows.meta,decorations.markings,decorations.pathreplacing,calc}

\usepackage{verbatim}
\usepackage{comment}

\usepackage[nottoc]{tocbibind} 
\usepackage{appendix}

\usepackage{geometry}
\newcommand\void[1]       {}

\theoremstyle{definition}

\newtheorem{thm}{Theorem}[section]
\newtheorem{prop}[thm]{Proposition}

\newtheorem{cor}[thm]{Corollary}
\newtheorem{lem}[thm]{Lemma}

\newtheorem{conv}[thm]{Convention}

{
\theoremsymbol{\mbox{${\blacksquare}$}}
\newtheorem{defn}[thm]{Definition}
}
{
\theoremsymbol{\mbox{$\vardiamondsuit$}}
\newtheorem{expl}[thm]{Example}
}
{
\theoremsymbol{\mbox{$\diamondsuit$}}
\newtheorem{rem}[thm]{Remark}
}
\qedsymbol{\mbox{$\square$}}

\numberwithin{equation}{section}

\newcommand\be            {\begin{equation}}
\newcommand\ee            {\end{equation}}
\newcommand\bea           {\begin{eqnarray}}
\newcommand\eea         {\end{eqnarray}}
\newcommand\bnu          {\begin{enumerate}}
\newcommand\enu          {\end{enumerate}}
\newcommand\bit          {\begin{itemize}}
\newcommand\eit          {\end{itemize}}

\newcommand{\pf}{\begin{proof}}
\newcommand{\epf}{\qed\end{proof}}

\makeatletter
\providecommand{\leftsquigarrow}{%
  \mathrel{\mathpalette\reflect@squig\relax}%
}
\newcommand{\reflect@squig}[2]{%
  \reflectbox{$\m@th#1\rightsquigarrow$}%
}
\makeatother

\DeclareMathAlphabet{\mathcal}{OMS}{cmsy}{m}{n}	
\DeclareMathAlphabet{\mathsf}{OT1}{cmss}{m}{n}	

\newcommand\Rb			{\mathbb{R}}
\newcommand\Zb			{\mathbb{Z}}
\newcommand\bk			{\mathbb{k}}

\newcommand\CA			{\EuScript{A}}
\newcommand\CB			{\EuScript{B}}
\newcommand\CC			{\EuScript{C}}
\newcommand\CD			{\EuScript{D}}
\newcommand\CE			{\EuScript{E}}

\newcommand\CL			{\EuScript{L}}
\newcommand\CM			{\EuScript{M}}
\newcommand\CN			{\EuScript{N}}

\newcommand\CP			{\EuScript{P}}

\newcommand{\FZ}			{\text{\usefont{U}{euf}{m}{n}Z}}

\newcommand\BC			{\mathbf{C}}
\newcommand\BD			{\mathbf{D}}

\newcommand\BM			{\mathbf{M}}

\newcommand\SX			{\mathsf{X}}

\DeclareMathOperator{\ev}{ev}
\DeclareMathOperator{\coev}{coev}

\DeclareMathOperator{\fun}{Fun}
\DeclareMathOperator{\Fun}{Fun}

\DeclareMathOperator{\Alg}{Alg}

\newcommand{\cat}{\mathbf{Cat}}
\newcommand{\fscat}{\mathbf{fsCat}}

\newcommand{\lmod}{\mathbf{LMod}}

\newcommand{\fslmod}{\mathbf{fsLMod}}
\newcommand{\ecat}{\mathbf{ECat}}
\newcommand{\fsecat}{\mathbf{fsECat}}

\newcommand{\op}			{\mathrm{op}}
\newcommand{\rev}			{\mathrm{rev}}

\newcommand{\one}			{\mathbb{1}}

\newcommand{\ob}          {\mathrm{ob}}

\newcommand{\oplax}{${}^{\mathrm{oplax}}$}
\newcommand{\lax}{${}^{\mathrm{lax}}$}
\newcommand{\hbcat}{\Gamma}
\DeclareMathOperator{\Lax}{lax}

\newcommand{\dtimes}{\mathbin{\dot{\times}}}
\newcommand{\hodot}{\mathbin{\hat{\odot}}}
\newcommand{\hotimes}{\mathbin{\hat{\otimes}}}

\newcommand\forget  {\text{\usefont{U}{euf}{m}{n}f}}

\newcommand\Set			{\EuScript{S}\mathrm{et}}

\newcommand\vect			{\mathrm{Vec}}

\newcommand\rep			{\mathrm{Rep}}

\newcommand{\Algn}[1][n]			{\mathrm{Alg}_{E_{#1}}}

\newcommand{\bscale}	{0.7}
\makeatletter
\newcommand{\ec}[2][]	{{\@ec{#1 |}{{#2}}}}
\newcommand{\bc}[2][]	{{\@ec{#1}{{#2}}}}
\newcommand{\@ec}[2]	{\mathchoice
  {\displaystyle \raise.9ex\hbox{$\scaleobj{\bscale}{#1}$} #2}%
  {\textstyle \raise.9ex\hbox{$\scaleobj{\bscale}{#1}$} #2}%
  {\scriptstyle \raise.55ex\hbox{$\scriptstyle \scaleobj{\bscale}{#1}$} #2}%
  {\scriptscriptstyle \raise.38ex\hbox{$\scriptscriptstyle \scaleobj{\bscale}{#1}$} #2}%
}
\makeatother

\newcommand{\cf}[1]{{#1}^{\mathbf{c}}}

\begin{document}

\title{Enriched monoidal categories I: centers}
\author[a,b]{Liang Kong \thanks{Email: \href{mailto:kongl@sustech.edu.cn}{\tt kongl@sustech.edu.cn}}}
\author[c,d]{Wei Yuan \thanks{Email: \href{mailto:wyuan@math.ac.cn}{\tt wyuan@math.ac.cn}}}
\author[e,a]{Zhi-Hao Zhang \thanks{Email: \href{mailto:zzh31416@mail.ustc.edu.cn}{\tt zzh31416@mail.ustc.edu.cn}}}
\author[f,g,h,a,b]{Hao Zheng \thanks{Email: \href{mailto:haozheng@mail.tsinghua.edu.cn}{\tt haozheng@mail.tsinghua.edu.cn}}}
\affil[a]{Shenzhen Institute for Quantum Science and Engineering, \authorcr
Southern University of Science and Technology, Shenzhen, 518055, China}
\affil[b]{Guangdong Provincial Key Laboratory of Quantum Science and Engineering, \authorcr
Southern University of Science and Technology, Shenzhen, 518055, China}
\affil[c]{Academy of Mathematics and Systems Science, \authorcr
Chinese Academy of Sciences, Beijing, 100190, China}
\affil[d]{School of Mathematical Sciences, \authorcr
University of Chinese Academy of Sciences, Beijing, 100049, China}
\affil[e]{Wu Wen-Tsun Key Laboratory of Mathematics of Chinese Academy of Sciences, \authorcr
School of Mathematical Sciences, \authorcr
University of Science and Technology of China, Hefei, 230026, China}
\affil[f]{Institute of Applied Mathematics, Tsinghua University, Beijing, 100084, China}
\affil[g]{Beijing Institute of Mathematical Sciences and Applications, Beijing, 101408, China}
\affil[h]{Department of Mathematics, Peking University, Beijing, 100871, China}
\date{\vspace{-5ex}}

\maketitle

\begin{abstract}
This work is the first one in a series, in which we develop a mathematical theory of enriched (braided) monoidal categories and their representations. In this work, we introduce the notion of the $E_0$-center ($E_1$-center or $E_2$-center) of an enriched (monoidal or braided monoidal) category, and compute the centers explicitly when the  enriched (braided monoidal or monoidal) categories are obtained from the canonical constructions. These centers have important applications in the mathematical theory of gapless and gapped boundaries of topological phases. They also play very important roles in the higher representation theory, which is the focus of the second work in the series.
\end{abstract}

\tableofcontents


\section{Introduction}

Enriched categories were introduced long ago \cite{EK66}, and have been studied intensively (see a classical review \cite{Kel82}).  A category enriched in a symmetric monoidal category $\CA$ is also called an $\CA$-enriched category or simply an $\CA$-category, and is denoted by $\ec[\CA]{\CL}$. The category $\CA$ is called the base category of the enriched category \cite{Kel82}. In this work, we choose to call $\CA$ the \textit{background category} of $\ec[\CA]{\CL}$ due to its physical meaning. The $\CL$ in $\ec[\CA]{\CL}$ denotes the underlying category of $\ec[\CA]{\CL}$. 

\smallskip
The early study of enriched categories focused mainly on enriched categories with a fixed symmetric monoidal category $\CA$, or equivalently, on the 2-category of $\CA$-categories, $\CA$-functors and $\CA$-natural transformations \cite{Kel82}. 
Although it is quite obvious to consider the change of background category by a symmetric monoidal functor $\CA \to \CA'$, it is natural and sufficient for many purposes to focus on $\CA$-categories as Kelly wrote in his classical book \cite{Kel82}: 
\begin{quote}
Closely connected to this is our decision not to discuss the ``change of base category'' given by a symmetric monoidal functor $\mathcal{V} \to \mathcal{V}'$, \ldots The general change of base, important though it is, is yet logically secondary to the basic $\mathcal{V}$-category theory it acts on. 
\end{quote}
The change of base categories do appear in several early works (see for example \cite{EK66,Gra80,May72,May80}). However, it was not clear from these early works if the theory of enriched categories with arbitrary background categories becomes significantly richer.

In recently years, however, there are strong motivations from both mathematics and physics to study the 2-category of enriched categories with different background categories. First, mathematicians start to explore enriched monoidal categories with the background categories being only braided (instead of being symmetric)  \cite{JS93,For04,BM12,MP17,KZ18a,MPP18,JMPP21}. This development was partially influenced by the modern developments on higher algebras \cite{BD98,Kon99,Lur17}. Secondly, in physics, recent progress in the study of gapless boundaries of 2+1D topological orders \cite{KZ18b,KZ20,KZ21}, topological phase transitions \cite{CJKYZ20}, boundary-bulk relation in 1+1D CFT \cite{KYZ21} and in the theory of quantum liquids \cite{KZ22b} demands us to consider enriched categories, enriched monoidal categories and enriched braided (or symmetric) monoidal categories such that the background categories vary from monoidal categories to braided monoidal categories and to symmetric monoidal categories, and to consider functors between enriched categories with different background categories, and to develop the representation theory of an $\CA$-enriched (braided) monoidal category $\ec[\CA]{\CL}$ in the 2-category of enriched (monoidal) categories with different background categories \cite{Zhe17,KZ20,KZ21,KYZ21}. In other words, physics demands us to develop a generalized theory within the 2-category of enriched categories (with arbitrary background categories), generalized enriched functors that can change the background categories (see Definition\,\ref{def:1_morphism}) and generalized enriched natural transformations (see Definition\,\ref{def:2_morphism}). After the appearance of this paper in arXiv, new applications of enriched categories in physics emerge: (1) enriched (monoidal) categories were shown to give a rather complete characterization of 1+1D gapped quantum liquids realized in 1+1D lattice models (including Ising chain and Kitaev chain) and their boundaries \cite{KWZ22,XZ22}; (2) they appear in the study of gapped boundaries of 3+1D Walker-Wang models \cite{HBJP23,GHKPPS24}; (3) they were further confirmed in the recent study of topological phase transitions \cite{CW23a,LY23}. 


\medskip
This work is the first one in a series to develop this generalized theory. More precisely, in this work, we study enriched (braided) monoidal categories, and study the $E_0$-center of an enriched category, the $E_1$-center (or the Drinfeld center or the monoidal center) of an enriched monoidal category and the $E_2$-center (or the M\"{u}ger center) of an enriched braided monoidal category via their universal properties. We introduce the canonical constructions of enriched (braided/symmetric monoidal) categories, and compute their centers (see Theorem \ref{thm:E_0_center_in_ecat}, \ref{thm:Drinfeld_center=E1_center}, \ref{thm:Mueger_center=E2_center} and Corollary \ref{cor:B=rigid-E0-center}, \ref{cor:E1_center=Z2Z2}, \ref{cor:E2_center_of_canonical_construction}). These results generalize the main results in \cite{KZ18a}, where two of the authors introduced the notion of the Drinfeld center of an enriched monoidal category without studying its universal property. These centers play important roles in the (higher) representation theory of enriched (braided) monoidal categories. We will develop the representation theory in the second work in the series. 

\medskip
The layout of this paper is given as follows. In Section\,\ref{sec:2-cat}, we review some basic notions and set our notations, and for $i=0,1,2$, we review the notion of the $E_i$-center of an $E_i$-algebra in a symmetric monoidal 2-category. In Section\,\ref{sec:enriched_categories}, we review the notion of an enriched category, and introduce the notions of an enriched functor and an enriched natural transformation. We also study the canonical construction of enriched categories as a locally isomorphic 2-functor. In Section\,\ref{sec:enriched_monoidal_categories}, we study enriched monoidal categories and the canonical construction, and compute the $E_0$-centers of enriched categories. In Section\,\ref{sec:enriched_braided_monoidal_categories}, we study enriched braided monoidal categories and the canonical construction, and compute the $E_1$-centers of enriched monoidal categories. In Section\,\ref{sec:enriched_symmetric_monoidal_categories}, we study enriched symmetric monoidal categories and the canonical construction, and compute the $E_2$-centers of enriched braided monoidal categories.

\medskip
Throughout the paper, we choose a ground field $k$ of characteristic zero when we mention a finite semisimple (or a multi-fusion) category.  We use $\bk$ or $\vect$ to denote the symmetric monoidal category of finite-dimensional vector spaces over $k$.

\bigskip
\noindent {\bf Acknowledgement}: 
LK and HZ are supported by Guangdong Provincial Key Laboratory (Grant No.~2019B121203002). LK and ZHZ are supported by NSFC under Grant No.~11971219 and by Guangdong Basic and Applied Basic Research Foundation under Grant No.~2020B1515120100. WY is supported by the NSFC under grant No.~11971463, 11871303, 11871127. ZHZ is supported by Wu Wen-Tsun Key Laboratory of Mathematics at USTC of Chinese Academy of Sciences. HZ is also supported by NSFC under Grant No.~11871078. We would like to thank referees for many suggestions for improvement.


\section{2-categories} \label{sec:2-cat}

In this section, we recall some basic notions in 2-categories and examples, and set the notations along the way. We refer the reader to \cite{JY21} for a detailed introduction to 2-categories and bicategories.

\subsection{Examples of 2-categories} \label{sec:expl-2-cat}
In this subsection, we recall some basic notions, such as an (op)lax-monoidal functor, an (op)lax-monoidal natural transformation, a few versions of the 2-category of (braided) monoidal categories, a left \oplax module over a monoidal category, a lax $\CA$-\oplax module functor and lax $\CA$-\oplax module natural transformation. We also set our notations for 2-categories that are frequently used in this work. 

\medskip
A bicategory is called a 2-category if the associators and the unitors are identities. We introduce simple notations for objects, 1-morphisms and 2-morphisms in a 2-category $\BC$ as $x,y\in\BC$, $(f,g \colon x\to y)\in \BC$ and $(\xi \colon f\Rightarrow g) \in \BC$, respectively. The compositions of 1-morphisms are denoted by juxtaposition. For example, for $f: x\to y$ and $h: y\to z$, their composition is denoted by $hf$. For 2-morphisms, there are two types of compositions: vertical composition and horizontal composition. By abusing notation, we denote both of them by juxtaposition. Their precise meanings should be clear from the context.

\begin{defn}
A 2-functor $F: \BC \to \BD$ is called \emph{locally equivalent} (or fully faithful) if $F_{x,y}: \BC(x,y) \to \BD(F(x),F(y))$ is an equivalence for all $x,y\in\BC$; it is called \emph{locally isomorphic} if $F_{x,y}$ is an isomorphism for all $x,y\in\BC$. 
\end{defn}

\begin{expl}
We give a few examples of 2-categories.  
\begin{itemize}
\item $\cat$: the 2-category $\cat$ of categories, functors and natural transformations. It is a symmetric monoidal 2-category with the tensor product given by the Cartesian product $\times$ and the tensor unit given by $\ast$, which consists of a single object and a single morphism.
\item $\Alg_{E_1}(\cat)$: the 2-category of monoidal categories, monoidal functors and monoidal natural transformations. We choose this notation because a monoidal category can be viewed as an $E_1$-algebra in $\cat$ (see for example \cite{Lur17}). It is a symmetric monoidal 2-category with the tensor product $\times$ and the tensor unit $\ast$.
\item $\Alg_{E_2}(\cat)$: the 2-category of braided monoidal categories, braided monoidal functors and monoidal natural transformations. It is a symmetric monoidal 2-category with the tensor product $\times$ and the tensor unit $\ast$.
\item $\Alg_{E_3}(\cat)$: the 2-category of symmetric monoidal categories, braided monoidal functors and monoidal natural transformations. It is a symmetric monoidal 2-category with the tensor product $\times$ and the tensor unit $\ast$.
\end{itemize}
It is well-known that the notion of a (symmetric or braided) monoidal category is equivalent to that of an ($E_3$- or $E_2$-) $E_1$-algebra in $\cat$ \cite{SW03, Lur17, Fre17}. This explains our notations. 
\end{expl}

In the rest of this subsection, we recall some basic notions and give a few more examples of 2-categories that are useful in this work. Let $\CA$ and $\CB$ be monoidal categories. 
\begin{defn} \label{def:lax-monoidal-functor}
A \textit{lax-monoidal functor} $F:\CA \to \CB$ is a functor equipped with a morphism $\one_\CB \to F(\one_\CA)$ and a natural transformation $F(x) \otimes F(y) \to F(x \otimes y)$ for $x,y\in\CA$ satisfying the following conditions:
\begin{enumerate}
    \item (lax associativity): for all $x, y, z \in \CA$, the following diagram commutes:
        \be\label{diag:lax_monoidal_asso}
        \begin{array}{c}
\xymatrix @R=0.2in @C=0.4in{
                F(x) \otimes F(y) \otimes F(z) \ar[r] \ar[d] & F(x) \otimes F(y \otimes z) \ar[d]\\
                F(x \otimes y) \otimes F(z) \ar[r] & F(x \otimes y \otimes z)
            }
\end{array}  
        \ee
    \item (lax unitality): for all $x \in \CA$, the following diagrams commute:
    \small
        \be \label{diag:lax_monoidal_unital}
        \begin{array}{c}
\xymatrix @R=0.2in@C=0.15in{
                \one_\CB \otimes F(x) \ar[r] \ar[d]& F(\one_\CA) \otimes F(x) \ar[d]\\
                F(x) & F(\one_\CA \otimes x) \ar[l]
            }
\end{array} \quad 
\begin{array}{c}
\xymatrix @R=0.2in@C=0.15in{
                F(x) \otimes \one_\CB \ar[r] \ar[d] & F(x) \otimes F(\one_\CA) \ar[d]\\
                F(x) & F(x \otimes \one_\CA) \ar[l] 
            }
\end{array} \ee \normalsize
\end{enumerate}
An \textit{oplax-monoidal functor} $G: \CA \to \CB$ is a lax-monoidal functor $\CA^{\op} \to \CB^{\op}$. Here $\CA^\op$ is denotes category obtained by reversing 1-morphisms in $\CA$.
\end{defn}

\begin{expl} \label{expl:lax_monoidal_functor}
We some standard examples and set our notation along the way. 
\bnu
\item Let $\CC$ be a monoidal category. A lax-monoidal functor $F: \ast \to \CC$ is simply an algebra $F(\ast)$ in $\CC$; an oplax-monoidal functor $G: \ast \to \CC$ is a co-algebra in $G(\ast)$ in $\CC$. 

\item Let $\Set$ be the category of (small) sets and $M$ a monoid. Then the functor $M \times -: \Set \to \Set$ defined by $S \mapsto M \times S$ is lax-monoidal. 

\item Let 
$A$ be a $k$-linear algebra. Then the functor $A\otimes_k -: \vect \to \vect$ defined by $V \mapsto A\otimes_k V$ is lax-monoidal. 
\enu
\end{expl}

\begin{defn}
A \emph{lax-monoidal natural transformation} between two lax-monoidal functor $F, G \colon \CA \to \CB$ is a natural transformation $\eta_a \colon F(a) \to G(a)$ such that the following diagrams commute:
\be \label{diag:unit_mul_cond_nat_lax_mon_nat_tran}
\begin{array}{c}
\xymatrix @R=0.2in{
        \one_\CB \ar[r] \ar[rd] & F(\one_\CA) \ar[d]^{\eta_{\one}}\\
        & G(\one_\CA)
    }
\end{array} \qquad
\begin{array}{c}
\xymatrix @R=0.2in{
        F(a) \otimes F(b) \ar[r] \ar[d]^{\eta_{a} \otimes \eta_{b}} & F(a \otimes b) \ar[d]^{\eta_{a \otimes b}}\\
        G(a) \otimes G(b) \ar[r]  & G(a \otimes b) 
    }
\end{array}
\ee
An oplax-monoidal natural transformation between two oplax-monoidal functors is defined similarly (by flipping non-vertical arrows in (\ref{diag:unit_mul_cond_nat_lax_mon_nat_tran})). 
\end{defn}

\begin{expl}
We give some examples of lax-monoidal natural transformations based on Example\,\ref{expl:lax_monoidal_functor}. 
\bnu
\item Let $F, F': \ast \to \CC$ be lax-monoidal functors. A lax-monoidal natural transformation $\phi: F \to F'$ is simply an algebra homomorphism from $F(\ast)$ to $F'(\ast)$. 

\item Let $M$ and $M$ be two monoids and $f: M \to M'$ a homomorphism between monoids. Then the family of maps $\{ M \times x \xrightarrow{f \times 1_x} M'\times x\}_{x\in \Set}$ defines a lax-monoidal natural transformation $M \times - \to M'\times -$.  

\item Let $A$ and $A'$ be two $k$-linear algebras and $g: A \to A'$ an algebra homomorphism. Then the family of linear maps $\{ A \otimes_k V \xrightarrow{f \otimes_k 1_V} M'\otimes_k V\}_{V\in \bk}$ defines a lax-monoidal natural transformation $A\otimes_k - \to A'\otimes_k -$. 

\enu
\end{expl}

Using above two notions, we obtain some new symmetric monoidal 2-categories (all with the tensor product $\times$ and the tensor unit $\ast$): 
\begin{itemize}
\item $\Alg_{E_1}^{\mathrm{lax}}(\cat)$: the 2-category of monoidal categories, lax-monoidal functors and lax-monoidal natural transformations.
\item $\Alg_{E_1}^{\mathrm{oplax}}(\cat)$: the 2-category of monoidal categories, oplax-monoidal functors and oplax-monoidal natural transformations.
\item $\Alg_{E_2}^{\mathrm{oplax}}(\cat)$: the 2-category of braided monoidal categories, braided oplax-monoidal functors and oplax-monoidal natural transformations.
\end{itemize}

\begin{defn} \label{def:oplax-left-module}
A \emph{left $\CA$-\oplax module} is a category $\CM$ equipped with an oplax-monoidal functor $\CA \to \Fun(\CM,\CM)$, or equivalently, an action functor $\odot \colon \CA \times \CM \to \CM$ equipped with
\bit
\item an \emph{oplax-associator}: a natural transformation $(- \otimes -) \odot - \Rightarrow - \odot (- \odot -)$ rendering the following diagram commutative.  
\be \label{diag:oplex_pen}
\begin{array}{c}
\xymatrix@C=0.2in @R=0.2in{
    ((a\otimes b) \otimes c) \odot x \ar[rr] \ar[d] & & (a\otimes b) \odot (c\odot x) \ar[d]  \\
(a \otimes (b\otimes c)) \odot x \ar[r] & a \odot ((b\otimes c) \odot x) \ar[r]  & a\odot (b \odot (c\odot x))
}
\end{array}
\ee
\item an \emph{oplax-unitor}: a natural transformation $\one \odot - \Rightarrow 1_\CM$ rendering the following two diagrams commutative.
\small
\be\label{diag:oplex_module_unit_left_right}
\begin{array}{c}
\xymatrix@C=0.1in@R=0.2in{
    (\one \otimes b) \odot x \ar[rr] \ar[rd] & &  \one \odot ( b\odot x) \ar[ld] \\
& b \odot x &
}
\end{array} \quad
\begin{array}{c}
\xymatrix@C=0.1in@R=0.2in{
    (b \otimes \one) \odot x \ar[rr] \ar[rd] & &  b \odot (\one \odot x) \ar[ld] \\
& b \odot x &
}
\end{array}
\ee \normalsize
\eit
If the oplax-associator (resp.~the oplax-unitor) is a natural isomorphism, then the left $\CA$-\oplax module is called \textit{strongly associative} (resp.~\textit{strongly unital}), and is called a \emph{left $\CA$-module} if it is both strongly associative and strongly unital. 

\smallskip
A left $\CA$-\lax module is defined by flipping the arrows of the oplax-associator and the oplax-unitor to give the lax-associator and the lax-unitor. 
\end{defn}

\begin{defn} \label{defn:lax_module_functor}
For two left $\CA$-\oplax modules $\CM$ and $\CN$, a \emph{lax $\CA$-\oplax module functor} from $\CM$ to $\CN$ is a functor equipped with a natural transformation
    \begin{align*}
        \alpha_{b,m} \colon b \odot F(m) \to F(b \odot m) 
    \end{align*}
    such that the following diagrams commute for all $a, b \in \CA$ and $m \in \CM$.
\small
\be \label{diag:lax_module_functor_diag}
\begin{array}{c}
\xymatrix@R=0.8em@C=1.5em{
a\odot (b\odot F(m)) \ar[r]  & a\odot F(b\odot m) \ar[dd] \\
(a\otimes b) \odot F(m) \ar[d] \ar[u] & \\
F((a\otimes b) \odot m) \ar[r] & F(a\odot (b\odot m))
}
\end{array} \quad
\begin{array}{c}
\xymatrix@C=1em @R=3em{
\one \odot F(m) \ar[rr] \ar[dr] & & F(\one \odot m) \ar[dl] \\
& F(m) &
}
\end{array}
\ee \normalsize
If $\alpha$ is a natural isomorphism, then $F$ is called an \emph{$\CA$-module functor}. 

An oplax $\CA$-\oplax module functor is defined by flipping the arrow $\alpha_{b,m}$. Similarly, a lax $\CA$-\lax module functor and an oplax $\CA$-\lax module functor can both be defined by flipping certain arrows. 
\end{defn}

\begin{defn}
For two lax $\CA$-\oplax module functors $F, G \colon \CM \to \CN$, a \emph{lax $\CA$-\oplax module natural transformation} $\xi \colon F \Rightarrow G$ is a natural transformation rendering the following diagram 
\[
        \xymatrix @R=0.2in @C=0.5in{
            a \odot F(m) \ar[r]^{1 \odot \xi_m} \ar[d] & a \odot G(m) \ar[d] \\
            F(a \odot m) \ar[r]^{\xi_{a \odot m}} & G(a \odot m)
        }
\]
commutative for all $a \in \CA$ and $m \in \CM$. An oplax $\CA$-\oplax module (or oplax $\CA$-\lax module, lax $\CA$-\lax module) natural transformation can be defined similarly. 
\end{defn}


\subsection{Centers in 2-categories}\label{sec:app_centers_in_2_cat}
In this subsection, we recall the notions of $E_i$-centers, for $i=0,1,2$, in a (symmetric) monoidal bicategory following Lurie's book \cite{Lur17}. 

\medskip
A contractible groupoid is a non-empty category that has a unique morphism between every two objects. An object $x$ in a bicategory $\BC$ is called a \textit{terminal object} if the hom category $\BC(y,x)$ is a contractible groupoid for every $y \in \BC$. 

In the rest of this section, we use $\BC$ to denote a monoidal bicategory (i.e., a tricategory with one object) equipped with a tensor product $\dtimes$ and a tensor unit $\ast$. By the coherence theorem of tricategories (see \cite{GPS95, Gur13}), without loss of generality, we further assume that $\BC$ is a semistrict monoidal 2-category, which is also called a Gray monoid (see \cite{KV94, BN96, JY21} for the definition of Gray monoid).

\begin{defn}\label{def:E_0_algebra}
An \emph{$E_0$-algebra} in $\BC$ is a pair $(A, u)$, where $A$ is an object in $\BC$ and $u \colon \ast \to A$ is a 1-morphism. 
\end{defn}

\begin{defn} \label{def:left_unital_action}
A \emph{left unital $(A, u)$-action} on an object $x \in \BC$ is a 1-morphism $g \colon A \dtimes x \to x$, together with an invertible 2-morphism $\alpha$ as depicted in the following diagram:
\[
\begin{array}{c}
\xymatrix @R=0.2in{
 & A \dtimes x \ar[dr]^{g} \\
    \ast \dtimes x = x \ar[ur]^{u \dtimes 1_x} \ar[rr]_-{1_x} \rrtwocell<\omit>{<-2>\alpha} & & x.
}
\end{array}
\]
\end{defn}

Let $x \in \BC$. We define the 2-category of left unital actions on $x$ as follows:
\begin{itemize}
    \item Objects are left unital actions on $x$.
    \item For $i=1,2$, let $((A_i,u_i), g_i, \alpha_i)$ be a left unital $(A_i, u_i)$-action on $x$. A 1-morphism from $((A_1,u_1), g_1, \alpha_1)$ to $((A_2,u_2), g_2, \alpha_2)$ is a triple $(p, \sigma, \rho)$, where $p \colon A_1 \to A_2$ is a 1-morphism in $\BC$, $\sigma \colon u_2 \Rightarrow p \circ u_1$ and $\rho \colon g_2 \circ (p \dtimes 1_x) \Rightarrow g_1$ are two invertible 2-morphism in $\BC$ such that the following identity of 2-morphisms holds. 
    \begin{align*}
        \begin{array}{c}
\xymatrix @C=0.6in @R=0.25in{
 & A_2 \dtimes x \ar@/^3ex/[ddr]^{g_2} \\
 \rtwocell<\omit>{\quad \sigma \dtimes 1}& A_1 \dtimes x \ar[dr]^{g_1} \ar[u]|{p \dtimes 1_x} \rtwocell<\omit>{\rho} & \\
\ast \dtimes x \ar[rr] \ar@/^3ex/[uur]^{u_2 \dtimes 1_x} \ar[ur]^{u_1 \dtimes 1_x} \rrtwocell<\omit>{<-2.5> \quad \alpha_1}  & & x
}
\end{array}
=
\begin{array}{c}
\xymatrix{
 & A_2 \dtimes x \ar[dr]^{g_2} \\
\ast \dtimes x \ar[ur]^{u_2 \dtimes 1_x} \ar[rr] \rrtwocell<\omit>{<-3> \quad\alpha_2} & & x
}
\end{array} .
    \end{align*}
\item Let $(p, \sigma, \rho) , (q, \beta, \varphi) \colon ((A_1,u_1), g_1, \alpha_1) \to ((A_2,u_2), g_2, \alpha_2)$ be 1-morphisms, a 2-morphism $(p, \sigma, \rho) \Rightarrow (q, \beta, \varphi)$ is a 2-morphism $\xi \colon p \Rightarrow q$ such that  
\begin{gather}
    \begin{array}{c}
        \xymatrix @R=0.3in @C=0.5in{
        \rtwocell<\omit>{<4.5>\sigma}& A_2\\
        \ast \ar@/^2ex/[ur]^-{u_2} \ar[r]_-{u_1} & A_1  \utwocell^{p}_{q}{\xi}
        }
   \end{array}=
    \begin{array}{c}
    \xymatrix @R=0.3in @C=0.5in{
        \rtwocell<\omit>{<4.5>\beta}& A_2\\
        \ast \ar@/^2ex/[ur]^-{u_2} \ar[r]_-{u_1} & A_1 \ar[u]_-{q}
    }
    \end{array}, \label{diag:comp-2-morphisms} \\ 
    \begin{array}{c}
       \xymatrix @R=0.3in @C=0.6in{
        A_2 \dtimes x \ar@/^2ex/[dr]^{g_2} \rtwocell<\omit>{<4.5>\varphi} & \\
        A_1 \dtimes x \utwocell<4.5>^{p \dtimes 1_x \quad}_{\quad q \dtimes 1_x}{\xi \dtimes 1} \ar[r]_-{g_1}& x
    }
    \end{array} =
    \begin{array}{c}
       \xymatrix @R=0.3in @C=0.6in{
        A_2 \dtimes x \ar@/^2ex/[dr]^{g_2} \rtwocell<\omit>{<4.5>\rho} & \\
        A_1 \dtimes x \ar[u]^-{p \dtimes 1_x} \ar[r]_-{g_1}& x
    }
    \end{array}. \nonumber
\end{gather}
\end{itemize}

\begin{rem}
    We add a remark on how to read the composition of 2-morphisms in diagrams for readers who are not familiar with 2-categories. The identity \eqref{diag:comp-2-morphisms} should be read as an identity between the composed 2-morphisms $(u_2 \xRightarrow{\sigma} p \circ u_1 \xrightarrow{\xi \circ 1_{u_1}} q \circ u_1)$ and $\beta$ (see also Chapter 3 in \cite{JY21}). 
\end{rem}

\begin{defn} \label{defn:E_0-center}
Let $x$ be an object in $\BC$. An $E_0$-center $\FZ_0(x)$ of $x$ is a terminal object (if exists) in the 2-category of left unital actions on $x$.
\end{defn}

\begin{rem}
A terminal object in a 2-category $\BC$ is unique in the sense that for any two terminal objects $x,x' \in \BC$ the hom category $\BC(x,x')$ is a contractible groupoid. Hence an $E_0$-center, if exists, is unique.
\end{rem}

\begin{rem}
In general, the notion of an $E_0$-center can be defined for an $E_0$-algebra. Indeed, a left unital action of an $E_0$-algebra on another $E_0$-algebra $(x,* \to x)$ is a left unital action on the object $x$ ``preserving'' the $E_0$-algebra structures, and the $E_0$-center of $(x,* \to x)$ is terminal among all left unital actions. One can show that the $E_0$-center is independent of the $E_0$-algebra structure $* \to x$, thus in practice we only use the notion of the $E_0$-center of an object.
\end{rem}

\begin{rem}
An $E_0$-center does not always exist. For example, if we view the natural numbers with multiplication as a monoidal 2-category with only identity 1-morphisms and identity 2-morphisms, then an $E_0$-center of $0$ does not exist. 
\end{rem}

\begin{defn}\label{defn:E_0_algebra}
    An $E_1$-algebra (or a pseudomonoid, see Section 3 in \cite{DS97}) in $\BC$ is a sextuple $(A, u, m,\alpha, \lambda, \rho)$, where $A \in \ob(\BC)$, $u \colon \ast \to A$ and $m \colon A \dtimes A \to A$ are 1-morphisms, and $\alpha, \lambda, \rho$ are invertible 2-morphisms (i.e., the associator, the left unitor and the right unitor, respectively) as depicted in the following diagrams
    \begin{align*}
        \xymatrix @R=0.2in @C=0.5in{
            A \dtimes A \dtimes A \ar[r]^-{m \dtimes 1} \ar[d]^-{1 \dtimes m} & A \dtimes A \ar[d]^-{m}\\
             \rtwocell<\omit>{<-3>\alpha}  A \dtimes A \ar[r]^-{m} & A 
        }\quad
        \xymatrix @R=0.2in{
            \ast \dtimes A \ar[r]^-{u \dtimes 1} \ar@/^{-2ex}/[rd]_-{1} \drtwocell<\omit>{<-0.7>\lambda} & A \dtimes A \ar[d]^-{m} & A \dtimes \ast \ar[l]_-{1 \dtimes u} \ar@/^{2ex}/[ld]^-{1} \\
            & A \urtwocell<\omit>{<-0.7>\rho}&
        } 
    \end{align*}
such that the following equations of pasting diagrams hold:
\small
\begin{align}
\begin{array}{c}
    \xymatrix @R=0.17in @C=0.1in{
        A^{\dtimes 4}  \ar[rr]^-{1 \dtimes 1 \dtimes m} \ar[rd]|-{1 \dtimes m\dtimes 1} \ar[dd]_{m \dtimes 1 \dtimes 1} & & A^{\dtimes 3} \ar[rd]^-{1 \dtimes m}\\
        & A^{\dtimes 3} \utwocell<\omit>{<3> 1 \dtimes \alpha} \ar[dd]|-{m \dtimes 1} \ar[rr]_-{1 \dtimes m} && A^{\dtimes 2} \ar[dd]^-{m}\\
        A^{\dtimes 3} \ar[rd]_{m \dtimes 1} \utwocell<\omit>{<3>\alpha \dtimes 1}&&\\
        & A^{\dtimes 2} \ar[rr]_-{m} \uutwocell<\omit>{<6>\alpha} && A
    } 
\end{array}= 
\begin{array}{c}
    \xymatrix @R=0.17in @C=0.1in{
        A^{\dtimes 4} \ar[dd]_-{m \dtimes 1 \dtimes 1} \ar[rr]^-{1\dtimes 1 \dtimes m} && A^{\dtimes 3} \ar[dd]|-{m \dtimes 1} \ar[rd]^-{1 \dtimes m} & \\
        &&& A^{\dtimes 2} \ar[dd]^-{m} \\
        A^{\dtimes 3}  \uutwocell<\omit>{<7> \Sigma_{m,m}}\ar[rr]^-{1 \dtimes m} \ar[rd]_-{m \dtimes 1} && A^{\dtimes 2} \ar[rd]|-{m} \utwocell<\omit>{<3> \alpha} \\
        & A^{\dtimes 2} \utwocell<\omit>{<3> \alpha} \ar[rr]_-{m}  && A
    } 
\end{array} \label{eq:algebra_2-cat_cond_1} \\
        \begin{array}{c}
            \xymatrix @R=0.15in @C=0.1in{
                & A \dtimes \ast \dtimes A \ar[ld]_-{1 \dtimes u \dtimes 1} \ar[rd]^-{1}&\\
            A \dtimes A \dtimes A \utwocell<\omit>{<10> \lambda} \ar[rr]^-{1 \dtimes m} \ar[d]_-{m \dtimes 1} && A \dtimes A \ar[d]^-{m}\\
            A \dtimes A \utwocell<\omit>{<10> \alpha} \ar[rr]_-{m} && A \\
            }
        \end{array} =
        \begin{array}{c}
            \xymatrix @R=0.15in @C=0.1in{
                & A \dtimes \ast \dtimes A \ar[ld]_-{1 \dtimes u \dtimes 1} \ar[rd]^-{1} \ar@/^1.3pc/[ldd]^-{1}&\\
            A \dtimes A \dtimes A  \ar[d]_-{m \dtimes 1} && A \dtimes A \ar[d]^-{m}\\
            A \dtimes A  \ar[rr]_-{m} \uutwocell<\omit>{<6> \rho} && A \\
            }
        \end{array} \label{eq:algebra_2-cat_cond_2}
\end{align} \normalsize
where $A^{\dtimes 4}$ is an abbreviation for $A \dtimes A \dtimes A \dtimes A$, and $\Sigma_{m,m}$ is the interchanger of the Gray-monoid $\BC$ (see, for example, definition 1 in \cite{DS97} for the data and axioms for Gray monoids). We often abbreviate an $E_1$-algebra to $(A, u, m)$ or $A$.
\end{defn}


The following result is a special case of \cite[Corollary\ 5.3.1.15]{Lur17}. For the sake of completeness, we sketch a proof.

\begin{prop}\label{prop:E_0_center_is_E_1_algebra}
    The $E_0$-center $\FZ_0(x)$ of $x \in \BC$, if exists, admits an $E_1$-algebra structure. 
\end{prop}

\begin{proof}
First, we show that a terminal object $A$ in a Gray monoid $(\BD,\dtimes,\ast)$ admits an $E_1$-algebra structure. Since $A$ is terminal, we have:
\bit
\item The hom categories $\BD(A \dtimes A,A)$ and $\BD(\ast,A)$ are contractible groupoids. So we can choose two 1-morphisms $m : A \dtimes A \to A$ and $u : \ast \to A$.
\item The hom categories $\BD(A^{\dtimes 3},A)$, $\BD(\ast \dtimes A,A)$ and $\BD(A \dtimes \ast,A)$ are contractible groupoids. So there are unique choices for the 2-morphisms $\alpha : m \circ (m \dtimes 1) \Rightarrow m \circ (1 \dtimes m)$, $\lambda : m \circ u \dtimes 1 \Rightarrow 1$ and $\rho : m \circ 1 \dtimes u \Rightarrow 1$.
\item The hom categories $\BD(A^{\dtimes 4},A)$ and $\BD(A \dtimes \ast \dtimes A,A)$ are contractible groupoids, so the pasting diagrams \eqref{eq:algebra_2-cat_cond_1} and \eqref{eq:algebra_2-cat_cond_2} automatically hold. Therefore, $(A,m,u,\alpha,\lambda,\rho)$ is an $E_1$-algebra in $\BD$.
\eit
Moreover, it is easy to see that different $E_1$-algebra structures on $A$ (i.e., different choices of $m$ and $u$) are equivalent.

Then it suffices to show that the Gray monoid structure of $\BC$ induces a Gray monoid structure on the 2-category of left unital actions on $x \in \BC$. Note that $(\ast, 1_{\ast})$ is an $E_0$-algebra. The unit of 2-category of left unital actions on $x$ is $((\ast, 1_\ast), 1_x, 1_{1_x})$. The product $((A_1, u_1), g_1, \alpha_1) \dtimes ((A_2, u_2), g_2, \alpha_2)$ of two left unital actions $((A_1, u_1), g_1, \alpha_1)$ and $((A_2, u_2), g_2, \alpha_2)$ is the left unital action defined by the following diagram.
\begin{align*}
 \xymatrix @R=0.15in @C=0.1in{   
    & & A_1 \dtimes A_2 \dtimes x \ar[rd]^{1 \dtimes g_2} & & \\
    & A_2 \dtimes x \ar[ru]^{u_1 \dtimes 1} \ar[rd]^{g_2} \rrtwocell<\omit>{<-1> \qquad \Sigma_{u_1, g_2}} &  & A_1 \dtimes x \ar[rd]^{g_1} & \\
    \ast \dtimes x \ar[ru]^{u_2 \dtimes 1} \ar[rr]_{1} \rrtwocell<\omit>{<-2> \alpha_2} & & \ast \dtimes x \ar[ru]^{u_1 \dtimes 1} \ar[rr]_{1} \rrtwocell<\omit>{<-2> \alpha_1}& & x ,
    }
\end{align*}
    where $\Sigma_{u_1, g_2}$ is the interchanger. By the fact that $\Sigma_{1, f}$ and $\Sigma_{f, 1}$ are identity 2-morphisms, it is not hard to check that $((\ast, 1_\ast), 1_x, 1_{1_x}) \dtimes ((A_1, u_1), g_1, \alpha_1) = ((A_1, u_1), g_1, \alpha_1) = ((A_1, u_1), g_1, \alpha_1) \dtimes ((\ast, 1_\ast), 1_x, 1_{1_x})$.

Let $(p: A_1 \to A_2, \sigma: u_2 \Rightarrow p \circ u_1, \rho: g_2 \circ (p \dtimes 1_x) \Rightarrow g_1)$ be a $1$-morphism from $((A_1, u_1), g_1, \alpha_1)$ to $((A_2, u_2), g_2, \alpha_2)$. For every left action $((B, v), h, \gamma)$, let 
\small
\begin{multline*}
(p, \sigma, \rho) \dtimes ((B, v), h, \gamma) \coloneqq \\
\bigg(p \dtimes 1_B, 
\begin{array}{c}
\xymatrix @R=0.15in @C=0.2in{
        \ast \ar[r]^{v} & B \ar[rd]_{u_1 \dtimes 1} \ar[rr]^{u_2 \dtimes 1} \rrtwocell<\omit>{<2> \quad \sigma \dtimes 1} & & A_2 \dtimes B \\
        & & A_1 \dtimes B \ar[ru]_{p \dtimes 1} & 
    }
\end{array} ,
\begin{array}{c}
\xymatrix @R=0.15in @C=0.4in{
        A_2 \dtimes B \dtimes x \ar[r]^{1 \dtimes h} \rtwocell<\omit>{<2> \quad \Sigma_{p,h}}  & A_2 \dtimes x \ar[r]^{g_2} \rtwocell<\omit>{<2> \rho} & x\\
        A_1 \dtimes B \dtimes x \ar[r]_{1 \dtimes h} \ar[u]^{p \dtimes 1} & A_1 \dtimes x \ar[u]^{p \dtimes 1} \ar @/_0.5pc/[ru]_{g_1} &
    }
\end{array} \biggr) ,
\end{multline*} \vspace{-1cm}
\begin{multline*}
((B, v), h, \gamma) \dtimes (p, \sigma, \rho) \coloneqq \\
\biggl( 1_B \dtimes p, 
\begin{array}{c}
\xymatrix @R=0.15in @C=0.4in{
    \ast \rtwocell<\omit>{<2> \sigma} \ar[r]^{u_2} \ar @/_0.5pc/[rd]_{u_1} & A_2 \ar[r]^{v \dtimes 1} \rtwocell<\omit>{<2> \quad \Sigma_{v, p}^{-1}}  & B \dtimes A_2\\
     & A_1 \ar[r]_{v \dtimes 1} \ar[u]_{p} & B \dtimes A_1 \ar[u]_{1 \dtimes p}
}
\end{array} ,
\begin{array}{c}
\xymatrix @R=0.15in @C=0.5in{
    B \dtimes A_2 \dtimes x \ar[r]^{1 \dtimes g_2} \rtwocell<\omit>{<2> \quad 1 \dtimes \rho} & B \dtimes x \ar[r]^{h} & x\\
    B \dtimes A_1 \dtimes x \ar @/_0.8pc/[ru]_{1 \dtimes g_1} \ar[u]^{1 \dtimes p \dtimes 1} & &
    }
\end{array} \biggr).
\end{multline*} \normalsize
Using the axioms satisfied by the interchangers, it is routine to check that $(p, \sigma, \rho) \dtimes ((B, v), h, \gamma)$ and $((B, v), h, \gamma) \dtimes (p, \sigma, \rho)$ are $1$-morphisms in the $2$-category of left unital actions. Similarly, it is not hard to check that for every $2$-morphism $\xi: p_1 \Rightarrow p_2$ between 1-morphisms $(p_1, \sigma_1, \rho_1) , (p_2, \sigma_2, \rho_2) \colon ((A_1,u_1), g_1, \alpha_1) \to ((A_2,u_2), g_2, \alpha_2)$,  
\begin{gather*}
\xi \dtimes ((B, v), h, \gamma) \coloneqq (\xi \dtimes 1_B: p_1 \dtimes 1_B \Rightarrow p_2 \dtimes 1_B), \\
((B, v), h, \gamma) \dtimes \xi \coloneqq (1_B \dtimes \xi: 1_B \dtimes p_1 \Rightarrow 1_B \dtimes p_2). 
\end{gather*}
are $2$-morphisms in the $2$-category of left unital actions, and the assignments on morphisms $- \dtimes ((B, v), h, \gamma)$ and $((B, v), h, \gamma) \dtimes -$ are 2-functors. 

    Let $(q: B_1 \to B_2, \beta: v_2 \Rightarrow q \circ v_1, \varphi: h_2 \circ (q \dtimes 1_x) \Rightarrow h_1)$ be a $1$-morphism from $((B_1, v_1), h_1, \gamma_1)$ to $((B_2, v_2), h_2, \gamma_2)$. Then the interchanger 
\begin{multline*}
[((A_2, u_2), g_2, \alpha_2) \dtimes (q, \beta, \varphi)] [(p, \sigma, \rho) \dtimes ((B_1, v_1), h_1, \gamma_1)] \\
\Rightarrow [(p, \sigma, \rho) \dtimes ((B_2, v_2), h_2, \gamma_2)] [((A_1, u_1), g_1, \alpha_1) \dtimes (q, \beta, \varphi)]
\end{multline*}
is defined by the interchanger $\Sigma_{p, q}:(p \dtimes 1_{B_1}) (1_{A_2} \dtimes q) \Rightarrow (p \dtimes 1_{B_2})(1_{A_1} \dtimes q)$ of $\BC$. The remaining axioms that the $2$-category of left unital actions on $x$ needs to satisfy as a Gray monoid are not difficult to verify, and we leave it as an exercise.
\end{proof}


According to the theory of higher algebras by Lurie \cite{Lur17}, if $\BC$ is a symmetric monoidal higher category, the $E_n$-center of an $E_n$-algebra $A$ in $\BC$ can be defined as the $E_0$-center of $A$ in the symmetric monoidal higher category $\Algn(\BC)$ of $E_n$-algebras in $\BC$, and we have $\Algn[1](\Algn(\BC)) \simeq \Algn[n+1](\BC)$ (see \cite[Theorem\ 5.1.2.2]{Lur17}). Then an $E_n$-center is automatically an $E_{n+1}$-algebra. The $E_1$-center of an $E_1$-algebra in a braided monoidal bicategory is also studied in \cite{Str04}.

\begin{expl}\label{expl:centers_in_cat}
Consider the symmetric monoidal 2-category $(\cat, \times, \ast)$ of categories, functors and natural transformations. $E_1$-algebras in $\cat$ are monoidal categories, and the $E_2$-algebras in $\cat$ are braided monoidal categories \cite{SW03,Fre17}\cite[Example\,5.1.2.4]{Lur17}, and $E_n$-algebras in $\cat$ for $n \geq 3$ are symmetric monoidal categories \cite[Example\,5.1.2.3]{Lur17}. 
   \begin{enumerate}
       \item The following diagram represents
           \begin{align}\label{diag:left_unital_in_cat}
 \xymatrix @R=0.2in{
 & \CA \times \CL \ar[dr]^{\odot} \\
{*} \times \CL \ar[ur]^{\one_\CA \times 1} \ar[rr] \rrtwocell<\omit>{<-2>\;\alpha} & & \CL
}   
\end{align}
a left unital $(\CA, \one_\CA)$-action on $\CL$ in $\cat$, where $\odot: \CA \times \CL \to \CL$ is a functor and $\alpha$ is a natural isomorphism. The functor from $\CA$ to the category $\fun(\CL, \CL)$ of endofunctors defined by $a \mapsto a \odot -$ is an $E_0$-algebra morphism, i.e., maps $\one_\CA \in \CA$ to $1_\CL \in \fun(\CL,\CL)$. We see that $\FZ_0(\CL) = \fun(\CL, \CL)$.

       \item For $\CA, \CL \in \Alg_{E_1}(\cat)$, let the diagram \eqref{diag:left_unital_in_cat} depict a left unital $(\CA, \one_\CA)$-action on $\CL$ in $\Alg_{E_1}(\cat)$, i.e., $\odot$ is a monoidal functor and $\alpha$ is a monoidal natural isomorphism. Then $\varphi \coloneqq - \odot \one_\CL \colon \CA \to \FZ_1(\CL)$ is a well-defined monoidal functor, and the half-braiding of $\varphi(a)$ is defined by:   
\begin{multline*}
    x \otimes \varphi(a) \xrightarrow{\alpha^{-1} \otimes 1} (\one_\CA \odot x) \otimes \varphi(a) \xrightarrow{\sim} (\one_\CA \otimes a) \odot (x \otimes \one_\CL) \\
    \xrightarrow{\sim} (a \otimes \one_\CA) \odot (\one_\CL \otimes x) \xrightarrow{\sim} \varphi(a) \otimes (\one_\CA \odot x) \xrightarrow{1 \otimes \alpha} \varphi(a) \otimes x. 
\end{multline*}
The $E_0$-center of $\CL$ in $\Alg_{E_1}(\cat)$ is given by the Drinfeld center $\FZ_1(\CL)$ of $\CL$ \cite{Maj91,JS91}\cite[Example\ 5.3.1.18]{Lur17}.  
Recall that the objects of the Drinfeld center $\FZ_1(\CL)$ are pairs $(x, \gamma_{-,x})$, where $\gamma_{-,x} \colon - \otimes x \to x \otimes -$ is a half braiding satisfying certain compatibility conditions (see for example section 7.13 in \cite{EGNO15}).


    \item For $\CA,\CL\in\Alg_{E_2}(\cat)$, the diagram \eqref{diag:left_unital_in_cat} depicts a left unital $(\CA, \one_\CA)$-action on $\CL$ in $\Alg_{E_2}(\cat)$. Since $\one_\CA \odot x \simeq x$ and $\odot$ is a braided monoidal functor, $\varphi(-) \coloneqq - \odot \one_\CL$ defines a braided monoidal functor from $\CA$ to the M\"{u}ger center \cite{Mueg03a} $\FZ_2(\CL)$ of $\CL$. Then it is straightforward to check that the $E_0$-center of $\CL$ in $\Alg_{E_2}(\cat)$ is given by the M\"{u}ger center $\FZ_2(\CL)$ of $\CL$ \cite[Example\ 5.3.1.18]{Lur17}.

   \item For $n \geq 3$, let $\Algn[n](\cat) \coloneqq \Alg_{E_1}(\Alg_{E_{n-1}}(\cat))$ be the symmetric monoidal 2-category of $E_n$-algebras in $\cat$. It is known that $\Alg_{E_n}(\cat)$ is biequivalent to the symmetric monoidal 2-category of symmetric monoidal categories, symmetric monoidal functors and monoidal natural transformations \cite[Example\,5.1.2.3]{Lur17}. The $E_0$-center of a symmetric monoidal category $\CL$ in $\Algn[n](\cat)$ is given by $\CL$ itself.
   \end{enumerate} 
\end{expl}

\begin{expl}
Let $\CC$ be a monoidal category. We can view $\CC$ as a monoidal 2-category with only identity 2-morphisms and denote it by $\BC$.  Then the The $E_0$-center of $x$ in the monoidal category $\CC$ can be defined as the $E_0$-center of x in $\BC$ (recall Definition \ref{defn:E_0-center}). In this sense, the theory of centers in 2-categories can be viewed as a direct generalization of that in 1-categories \cite{Ost03,Dav10,Str12,KYZ21}.
\end{expl}

\begin{rem}
Suppose $\BC$ is a monoidal 2-category and $\BM$ is a $\BC$-module. We can similarly define the $E_0$-center of an object $x \in \BM$ in $\BC$. More generally, if $\BC$ is an $E_{n+1}$-monoidal higher category and $\BM$ is an $E_n$-module over $\BC$, then $\Algn(\BM)$ is an $\Algn(\BC)$-module, and the $E_n$-center of an $E_n$-algebra $A \in \Algn(\BM)$ in $\BC$ can be defined as the $E_0$-center of $A$ in $\Algn(\BC)$. 
\end{rem}


\section{Enriched categories} \label{sec:enriched_categories}

In this section, we introduce the basic notions of an enriched functor and an enriched natural transformation, thus obtain a few versions of the categories of enriched categories. In the end, we recall and further study the canonical construction of enriched categories. 

\subsection{Enriched categories, functors and natural transformations}

Let $(\CA, \otimes, \one_\CA)$ be a monoidal category. In the following, the associators and unitors of $\CA$ are suppressed for simplicity. We recall the notion of an $\CA$-enriched category (see for example \cite{Kel82}).
\begin{defn}
An \emph{$\CA$-enriched category} $\ec[\CA]{\CL}$,\footnote{See Remark\,\ref{rem:notation-enriched} for an explanation of the notation $\ec[\CA]{\CL}$.} or a category {\em enriched in $\CA$}, consists of the following data:
\bnu
\item A set of objects $\ob (\ec[\CA]{\CL})$. If $x \in \ob(\ec[\CA]{\CL})$, we also write $x \in \ec[\CA]{\CL}$.  
\item An object $\ec[\CA]{\CL}(x,y)$ in $\CA$ for every pair $x,y\in \ec[\CA]{\CL}$. 
\item A distinguished morphism $1_x \colon \one_\CA \to \ec[\CA]{\CL}(x,x)$ in $\CA$ for every $x\in \ec[\CA]{\CL}$. 
\item A composition morphism $\ec[\CA]{\CL}(y,z)\otimes \ec[\CA]{\CL}(x,y)\xrightarrow{\circ} \ec[\CA]{\CL}(x,z)$ in $\CA$ for every triple $x,y,z\in \ec[\CA]{\CL}$. 
\enu
They are required to make the following diagrams commutative for $x,y,z,w\in \ec[\CA]{\CL}$.  
\begin{align}
\begin{array}{c}
\xymatrix @C=0.5in @R=0.2in{
    \ec[\CA]{\CL}(y,z) \otimes \ec[\CA]{\CL}(x,y)\otimes \ec[\CA]{\CL}(w,x) \ar[r]^-{1\otimes\circ} \ar[d]_{\circ\otimes 1} & \ec[\CA]{\CL}(y,z)\otimes \ec[\CA]{\CL}(w,y) \ar[d]^\circ \\
  \ec[\CA]{\CL}(x,z)\otimes \ec[\CA]{\CL}(w,x) \ar[r]^-\circ & \ec[\CA]{\CL}(w,z)
}
\end{array} \label{diag:asso-circ} \\
\begin{array}{c}
\xymatrix @C=0.3in @R=0.2in{
    \one_\CA \otimes \ec[\CA]{\CL}(x,y) \ar[r]^\simeq \ar[d]_{1_y \otimes 1} 
    & \ec[\CA]{\CL}(x,y) 
    & \ec[\CA]{\CL}(x,y) \otimes \one_\CA \ar[l]_\simeq \ar[d]^{1 \otimes 1_x}\\
    \ec[\CA]{\CL}(y,y) \otimes \ec[\CA]{\CL}(x,y) \ar[ru]^-{\circ} &  & \ec[\CA]{\CL}(x,y) \otimes \ec[\CA]{\CL}(x,x) \ar[lu]_-{\circ}
}
\end{array} \label{diag:right-left-unit}
\end{align}
\end{defn}

In the rest of this work, the monoidal category $\CA$ is referred to as the \textit{background category}\footnote{The category $\CA$ is often called the base category or ambient category of $\ec[\CA]{\CL}$ in mathematical literature \cite{Kel82}. Our choice was influenced by its physical meaning \cite{KZ18b}, which was further clarified and extended in recent physical works under the name of ``categorical symmetry'' \cite{JW20,KLWZZ20a,KWZ22,KZ22,CW23,LZ24}.} of $\ec[\CA]{\CL}$. The \textit{underlying category} of $\ec[\CA]{\CL}$, denoted by $\CL$, is defined by 
\[
\ob(\CL) \coloneqq \ob(\ec[\CA]{\CL}) \quad \mbox{and} \quad \CL(x,y) \coloneqq \CA(\one_\CA, \ec[\CA]{\CL}(x,y)), \quad \forall x,y\in\ec[\CA]{\CL};
\] 
and the composition $g\circ f$ of morphisms $\one_\CA \xrightarrow{g} \ec[\CA]{\CL}(y,z)$ and $\one_\CA \xrightarrow{f} \ec[\CA]{\CL}(x,y)$ is defined by
\begin{align}\label{equ:def_underlying_cat_1}
    \one_\CA \to \one_\CA \otimes \one_\CA \xrightarrow{g \otimes f} \ec[\CA]{\CL}(y,z) \otimes \ec[\CA]{\CL}(x,y) \xrightarrow{\circ} \ec[\CA]{\CL}(x,z);  
\end{align}
the identity morphism of $x$ in $\CL$ is precisely $1_x \colon \one_\CA \to \ec[\CA]{\CL}(x,x)$.

\begin{rem} \label{rem:hom_functor}
Let $\ec[\CA]{\CL}$ be an enriched category. A morphism $f \in \CL(x,y)$ in the underlying category $\CL$ induces a morphism $\ec[\CA]{\CL}(w,f) \colon \ec[\CA]{\CL}(w,x) \to \ec[\CA]{\CL}(w,y)$ in $\CA$ for $w \in \ec[\CA]{\CL}$, defined by
\be
\ec[\CA]{\CL}(w,x) \simeq \one_\CA \otimes \ec[\CA]{\CL}(w,x) \xrightarrow{f \otimes 1} \ec[\CA]{\CL}(x,y) \otimes \ec[\CA]{\CL}(w,x) \xrightarrow{\circ} \ec[\CA]{\CL}(w,y) .
\ee
Similarly $f$ also induces a morphism $\ec[\CA]{\CL}(f,w) \colon \ec[\CA]{\CL}(y,w) \to \ec[\CA]{\CL}(x,w)$ for $w \in \ec[\CA]{\CL}$. Moreover, if $g \in \CL(y,z)$, one can verify that $\ec[\CA]{\CL}(w,g) \circ \ec[\CA]{\CL}(w,f) = \ec[\CA]{\CL}(w,g \circ f)$. Thus $\ec[\CA]{\CL}(w,-)$ defines an ordinary functor $\CL \to \CA$. Similarly, $\ec[\CA]{\CL}(-,w)$ is an ordinary functor $\CL^\op \to \CA$. Combining them together, we obtain a well-defined functor $\ec[\CA]{\CL}(-,-) \colon \CL^\op \times \CL \to \CA$ \cite{Kel82}. 
\end{rem}

\begin{expl}
A usual category is a category enriched in $\Set$; a usual finite $k$-linear category or a multi-fusion category \cite{EGNO15} is enriched in $\vect$. 
\end{expl}

\begin{expl}[\cite{Law73}] \label{expl:metric_space}
Consider the poset $\Rb_+ \coloneqq ([0,+\infty),\geq)$ of non-negative real numbers as a category. It can be equipped with a monoidal structure with the tensor product given by the addition $+$ and the tensor unit given by $0$. Then a metric space $X = (X,d)$ can be viewed as an $\Rb_+$-enriched category $\ec[\Rb_+]{X}$:
\bit
\item the set of objects in $\ec[\Rb_+]{X}$ is the underlying set of $X$;
\item the hom objects are given by $\ec[\Rb_+]{X}(x,y) \coloneqq d(x,y) \in \Rb_+$;
\item the compositions are given by the triangle inequality $d(y,z) + d(x,y) \geq d(x,z)$;
\item the identities are given by $0 \geq d(x,x)$ (indeed, $d(x,x) = 0$).
\eit
\end{expl}

\begin{expl}
For an $\CA$-enriched category $\ec[\CA]{\CL}$, we define an $\CA^\rev$-enriched category $\ec[\CA^\rev]{\CL^\op}$ as follows:
\bit
\item $\ob(\ec[\CA^\rev]{\CL^\op}) \coloneqq \ob(\ec[\CA]{\CL})$;
\item For $x,y \in \ec[\CA^\rev]{\CL^\op}$, $\ec[\CA^\rev]{\CL^\op}(x,y) \coloneqq \ec[\CA]{\CL}(y,x)$;
\item The identity is given by $1_x \colon \one_\CA \to \ec[\CA]{\CL}(x,x) = \ec[\CA^\rev]{\CL^\op}(x,x)$;
\item the composition is defined by
\[
\ec[\CA^\rev]{\CL^\op}(y,z) \otimes^\rev \ec[\CA^\rev]{\CL^\op}(x,y) = \ec[\CA]{\CL}(y,x) \otimes \ec[\CA]{\CL}(z,y) \xrightarrow{\circ} \ec[\CA]{\CL}(z,x) = \ec[\CA^\rev]{\CL^\op}(x,z) ,
\]
where $\otimes^\rev$ is the reversed tensor product of $\CA$.
\eit
We call $\ec[\CA^\rev]{\CL^\op}$ the opposite category of $\ec[\CA]{\CL}$, i.e., $(\ec[\CA]{\CL})^\op \coloneqq \ec[\CA^\rev]{\CL^\op}$. 
We have
\[
\CA^\rev(\one_\CA,\ec[\CA^\rev]{\CL^\op}(x,y)) = \CA(\one_\CA,\ec[\CA]{\CL}(y,x)) = \CL(y,x) = \CL^\op(x,y) .
\]
It means that the underlying category of $\ec[\CA^\rev]{\CL^\op}$ is $\CL^\op$. This explains our notation. 
\end{expl}


\begin{defn}\label{def:1_morphism}
    An \emph{enriched functor} $\ec{F} \colon \ec[\CA]{\CL} \to \ec[\CB]{\CM}$ between two enriched categories consists of the following data:
\begin{enumerate}
    \item A lax-monoidal functor $\hat{F} \colon \CA \to \CB$, which is called the \textit{background changing functor}; 
    \item a map $F \colon \ob(\CL) \to \ob(\CM)$; 
    \item a family of morphisms $\ec{F}_{x,y} \colon \hat{F}(\ec[\CA]{\CL}(x,y)) \to \ec[\CB]{\CM}(F(x), F(y))$ in $\CB$ for all objects $x,y\in \CL$; 
\end{enumerate}    
satisfying the following two axioms:     
\begin{enumerate}
    \item $1_{F(x)} = (\one_\CB \to \hat{F}(\one_\CA) \xrightarrow{\hat{F}(1_x)} \hat{F}(\ec[\CA]{\CL}(x,x)) \xrightarrow{\ec{F}_{x,x}} \ec[\CB]{\CM}(F(x), F(x)))$;
    \item The diagram 
\small
\be \label{diag:1_morphism_diag}
\begin{array}{c}
\xymatrix{
\hat{F}(\ec[\CA]{\CL}(y,z))\otimes \hat{F}(\ec[\CA]{\CL}(x,y)) \ar[r] \ar[d]_{\ec{F}_{y,z}\otimes \ec{F}_{x,y}} & \ec[\CB]{\CM}(F(y),F(z))\otimes \ec[\CB]{\CM}(F(x),F(y)) \ar[dd]^-\circ \\
\hat{F}(\ec[\CA]{\CL}(y,z)\otimes \ec[\CA]{\CL}(x,y)) \ar[d]^-{\hat{F}(\circ)} \\
\hat{F}(\ec[\CA]{\CL}(x,z)) \ar[r]^{\ec{F}_{x,z}} & \ec[\CB]{\CM}(F(x),F(z))
}
\end{array}
\ee \normalsize
in $\CB$ is commutative.
\end{enumerate}
\end{defn}

\begin{expl}
We want to emphasize that when $\CA=\Set$, the notion of an enriched functor does not coincide with an ordinary functor unless we require $\hat{F}=1_\CA$ (see also Definition\,\ref{def:A-functor}). For example, let $M$ be a non-trivial monoid and $t: M \to \ast$ be the unique map from $M$ to the terminal object $\ast$ in $\Set$. Then $\hat{F}=M\times -: \Set \to \Set$ (recall Example,\,\ref{expl:lax_monoidal_functor}) and $\ec{F}_{x,y}: M\times \mathrm{Map}(x,y) \xrightarrow{t\times 1} \ast \times \mathrm{Map}(x,y) = \mathrm{Map}(x,y)$ for $x,y\in \Set$ defines an enriched functor $\Set \to \Set$, which is not a functor in the usual sense. In other words, even in the case $\CA=\Set$, this work provides a non-trivial generalization of the usual category theory. 
\end{expl}

\begin{expl}
Recall Example\,\ref{expl:metric_space}. A metric space $X = (X,d)$ can be viewed as an $\Rb_+$-enriched category $\ec[\Rb_+]{X}$. It is not hard to see that a lax-monoidal functor $\hat F : \Rb_+ \to \Rb_+$ is a monotone map $\hat F : \Rb_+ \to \Rb_+$ such that $\hat F(x) + \hat F(y) \geq \hat F(x+y)$. Therefore, for any real number $K \geq 0$ the map $\hat F_K \coloneqq K \cdot -$ of multiplying by $K$ is a lax-monoidal functor from $\Rb_+$ to itself (indeed, $\hat F_K$ is a monoidal functor and any monoidal functor $\Rb_+ \to \Rb_+$ has this form). Given two metric spaces $(X,d_X),(Y,d_Y)$, an enriched functor $\ec{F} : \ec[\Rb_+]{X} \to \ec[\Rb_+]{Y}$ with the background changing functor given by $\hat F = \hat F_K$ is simply a Lipschitz map $F : X \to Y$ with a Lipschitz constant $K$, i.e., a map $F : X \to Y$ satisfying
\[
K \cdot d_X(x_1,x_2) \geq d_Y(F(x_1),F(x_2)) .
\]
\end{expl}

For every $f: \one_\CA \to \ec[\CA]{\CL}(x,y)$, we define 
\be \label{equ:def_underlying_cat_2}
        F(f) \coloneqq \big( \one_\CB \to \hat{F}(\one_\CA) \xrightarrow{\hat{F}(f)} \hat{F}(\ec[\CA]{\CL}(x,y)) \xrightarrow{\ec{F}_{x,y}} \ec[\CB]{\CM}(F(x), F(y)) \big). 
\ee
The commutativity of the diagram \eqref{diag:1_morphism_diag} implies that $F(gf) = F(g)F(f)$ for every $f \in \CL(x,y)$ and $g \in \CL(y,z)$. As a consequence, we obtain a functor $F: \CL \to \CM$, which is called the \textit{underlying functor} of $\ec{F}$. When $\hat{F}=1_\CA$, we recover the classical notion of a underlying functor (see \cite{Kel82}).

\begin{rem}
Given an enriched functor $\ec{F} : \ec[\CA]{\CL} \to \ec[\CB]{\CM}$, for $f \in \CL(x,y)$ in the underlying category $\CL$, the commutativity of the diagram \eqref{diag:1_morphism_diag} implies that of the following diagram. 
\[
\xymatrix{
\hat F(\ec[\CA]{\CL}(w,x)) \ar[d]_{\hat F(\ec[\CA]{\CL}(w,f))} \ar[rr]^{\ec{F}_{w,x}} & & \ec[\CB]{\CM}(F(w),F(x)) \ar[d]^{\ec[\CB]{\CM}(F(w),F(f))} \\
 \hat F(\ec[\CA]{\CL}(w,y)) \ar[rr]^{\ec{F}_{w,y}} & & \ec[\CB]{\CM}(F(w),F(y))
}
\]
In other words, $\ec{F}_{w,-} : \hat F(\ec[\CA]{\CL}(w,-)) \Rightarrow \ec[\CB]{\CM}(F(w),F(-))$ is a natural transformation. Similarly, $\ec{F}_{-,w}: \hat F(\ec[\CA]{\CL}(-,w)) \Rightarrow \ec[\CB]{\CM}(F(-),F(w))$ is also a natural transformation. Combining them together, we obtain a natural transformation $\ec{F}_{-,-}: \hat F(\ec[\CA]{\CL}(-,-)) \Rightarrow \ec[\CB]{\CM}(F(-),F(-))$, generalizing the classical result for $\hat{F}=1_\CA$ \cite{Kel82}. 
\end{rem}

\begin{defn}\label{def:2_morphism}
    An \emph{enriched natural transformation} $\ec{\xi}: \ec{F} \Rightarrow \ec{G}$ between two enriched functors $\ec{F}, \ec{G}: \ec[\CA]{\CL} \to \ec[\CB]{\CM}$ consists of a lax-monoidal natural transformation $\hat{\xi}: \hat{F} \to \hat{G}$, which is called the \emph{background changing natural transformation} of $\ec{\xi}$, and a family of morphisms 
\begin{align*}
    \xi_x: \one_\CB \to \ec[\CB]{\CM}(F(x), G(x)), \quad  \forall x \in \ec[\CA]{\CL},
\end{align*}
rendering the following diagram commutative. 
\small
\begin{equation} \label{diag:2_morphism_diag}
\begin{array}{c}
\xymatrix @R=0.2in @C=0.1in{
\hat{F}(\ec[\CA]{\CL}(x,y)) \ar[r] \ar[d]^{\hat{\xi}} & \one_\CB \otimes \hat{F}(\ec[\CA]{\CL}(x,y)) \ar[d]^-{\xi_y \otimes \ec{F}_{x,y}} & \\
\hat{G}(\ec[\CA]{\CL}(x,y)) \otimes \one_\CB \ar[d]^-{\ec{G}_{x,y} \otimes \xi_{x}} & \ec[\CB]{\CM}(F(y), G(y)) \otimes \ec[\CB]{\CM}(F(x), F(y)) \ar[d]^{\circ} \\
\ec[\CB]{\CM}(G(x), G(y)) \otimes \ec[\CB]{\CM}(F(x), G(x)) \ar[r]^-{\circ} & \ec[\CB]{\CM}(F(x), G(y))
    }
\end{array}
\end{equation} \normalsize
The family of morphisms $\{ \xi_x \}$ automatically defines a natural transformation $\xi: F \Rightarrow G$, which is called the \emph{underlying natural transformation} of $\ec{\xi}$. 
\end{defn}

\begin{rem}
    Two enriched natural transformations $\ec{\xi}$ and $\ec{\eta}$ are equal in $\ec[\Lax]{\ecat}$ (see Proposition \ref{prop:2_cat_enriched_def} for the definition of $\ec[\Lax]{\ecat}$) if and only if $\hat{\xi}=\hat{\eta}$ and $\xi=\eta$. 
\end{rem}

\begin{rem}\label{rem:2_morphism}
The diagram \eqref{diag:2_morphism_diag}( can be rewritten as follows: 
\be \label{diag:2_morphism_diag_2}
\begin{array}{c}
\xymatrix @R=0.2in @C=0.25in{
        \hat{G}(\ec[\CA]{\CL}(x,y)) \ar[d]_{\ec{G}_{x,y}} & \hat{F}(\ec[\CA]{\CL}(x,y)) \ar[r]^-{\ec{F}_{x,y}} \ar[l]_{\hat{\xi}} & \ec[\CB]{\CM}(F(x),F(y)) \ar[d]^{\ec[\CB]{\CM}(F(x),\xi_y)} \\
        \ec[\CB]{\CM}(G(x),G(y)) \ar[rr]^{\ec[\CB]{\CM}(\xi_x,G(y))} & & \ec[\CB]{\CM}(F(x),G(y))
    }
\end{array}
\ee  
The commutativity of \eqref{diag:2_morphism_diag_2} guarantees that $\{ \xi_x \}$ gives a well-defined natural transformation $\xi: F\Rightarrow G$.
\end{rem}

\begin{defn}[\cite{Kel82}] \label{def:A-functor}
An enriched functor $\ec{F} \colon \ec[\CA]{\CL} \to \ec[\CA]{\CM}$ is called an \textit{$\CA$-functor} if $\hat{F}=1_\CA$. 
An enriched natural transformation $\ec{\xi} \colon \ec{F} \Rightarrow \ec{G}$ between two $\CA$-functors is called an \textit{$\CA$-natural transformation} if $\hat{\xi}$ is the identity natural transformation of $1_\CA$, i.e., $\hat{\xi}=1_{1_\CA}$.
\end{defn}
 

\begin{expl}\label{expl:enrich_cat_with_one_obj}
An $\CA$-enriched category with only one object $\ast$ is determined by an algebra $(A, m_A:A \otimes A \to A, \iota_A: \one \to A)$ in $\CA$ with the identity morphism $1_\ast:\one \to A$ defined by $\iota_A$ and the composition morphism $\circ: A \otimes A \to A$ defined by $m$. We use $\ast^A$ to denote this enriched category. If $\CA$ is the monoidal category with only one object and one morphism, we simply denote $\ast^{\one}$ by $\ast$.

    Let $B$ be an algebra in a monoidal category $\CB$. Then an enriched functor $\ec{F}: \ast^A \to \ast^B$ is defined by a lax-monoidal functor $\hat{F}: \CA \to \CB$ and an algebra homomorphism $\ec{F}_{*,*}: \hat{F}(A) \to B$. 
    
    Given another enriched functor $\ec{G}: \ast^A \to \ast^B$, an enriched natural transformations $\ec{\xi}: \ec{F} \Rightarrow \ec{G}$ is defined by a lax-monoidal natural transformation $\hat{\xi}: \hat{F} \to \hat{G}$ and a morphism $\xi: \one_\CB \to B$ rendering the following diagram commutative.
    \begin{align}\label{diag:enrich_cat_with_one_obj_1}
\begin{array}{c}
\xymatrix @R=0.2in @C=0.4in{
            \hat{F}(A) \ar[d] \ar[r] & \one_\CB \otimes \hat{F}(A) \ar[r]^-{\xi \otimes \ec{F}_{*,*}} & B \otimes B \ar[d]^-{m_B} \\
            \hat{G}(A) \otimes \one_\CB \ar[r]^-{\ec{G}_{*,*} \otimes \xi} & B \otimes B \ar[r]^-{m_B} & B
        }
\end{array}
    \end{align}
For example, when $\xi=\iota_B$ and $\hat{\xi}$ satisfies the condition $\ec{G}_{*,*}\circ\hat{\xi}_A=\ec{F}_{*,*}$, they define an enriched natural transformation. 
\end{expl}

\begin{expl}\label{conv:morphism_to_x}
    Let $\ec[\CA]{\CL}$ be an $\CA$-enriched category. For every $x \in \ec[\CA]{\CL}$, we use $\ec{x}$ to denote the enriched functor $\ast \to \ec[\CA]{\CL}$ defined by $\hat{x}=\one_\CA$, $\ast \mapsto x$ and $\ec{x}_{*,*}=1_x$. The underlying functor of $\ec{x}$, i.e., $\ast \mapsto x$, is denoted by $x$. 
\end{expl}

\subsection{2-categories of enriched categories}


\begin{prop}\label{prop:2_cat_enriched_def}
    Enriched categories (as objects), enriched functors (as 1-morphisms) and enriched natural transformations (as 2-morphisms) form a 2-category $\ec[\Lax]{\ecat}$. 
    \begin{enumerate}
        \item The composition $\ec{G} \circ \ec{F}$ of 1-morphisms $\ec{F}: \ec[\CA]{\CL} \to \ec[\CB]{\CM}$ and $\ec{G}:\ec[\CB]{\CM} \to \ec[\CC]{\CN}$ is the enriched functor defined by the the composition of lax-monoidal functors $\hat{G}\hat{F}$ and the underlying functor $GF$ induced by the family of morphisms   
\small
\[
    \hat{G}\hat{F}(\ec[\CA]{\CL}(x,y)) \xrightarrow{\hat{G}(\ec{F}_{x,y})} \hat{G}(\ec[\CB]{\CM}(F(x), F(y)) \xrightarrow{\ec{G}_{F(x), F(y)}} \ec[\CC]{\CN}(GF(x), GF(y)) 
\] \normalsize
for all $x,y \in \ec[\CA]{\CL}$.
        \item The horizontal composition $\ec{\eta} \circ \ec{\xi}: \ec{H}\circ \ec{F} \Rightarrow \ec{L} \circ \ec{G}$ of 2-morphisms $\ec{\xi}: \ec{F} \Rightarrow \ec{G}$ and $\ec{\eta}:\ec{H} \Rightarrow \ec{L}$, where $\ec{F}, \ec{G}: \ec[\CA]{\CL} \to \ec[\CB]{\CM}$ and $\ec{H}, \ec{L}: \ec[\CB]{\CM} \to \ec[\CC]{\CN}$ are 1-morphisms, is the enriched natural transformation defined by the horizontal compositions of natural transformations $\hat{\eta} \circ \hat{\xi}: \hat{H} \circ \hat{F} \to \hat{L} \circ \hat{G}$ and $\eta \circ \xi: H \circ F \to L \circ G$. 
        \item  The vertical composition of 2-morphisms $\ec{\xi}: \ec{F} \Rightarrow \ec{G}$ and $\ec{\beta}: \ec{G} \Rightarrow \ec{K}$, where $\ec{F}, \ec{G}, \ec{K}: \ec[\CA]{\CL} \to \ec[\CB]{\CM}$ are 1-morphisms, is defined by the vertical composition of natural transformations $\hat{\beta} \hat{\xi}: \hat{F} \to \hat{K}$ and $\beta \xi: F \to K$, 
    \end{enumerate}
\end{prop}

We denote the sub-2-category of $\ec[\Lax]{\ecat}$ consisting of enriched functors $\ec{F}$ such that $\hat{F}$ is monoidal by $\ecat$. Fix a monoidal category $\CA$, all $\CA$-enriched categories, $\CA$-functors and $\CA$-natural transformations form a sub-2-category of $\ec[\Lax]{\ecat}$, denoted by $\ec[\CA]{\ecat}$.

\begin{rem}
An enriched natural transformation $\ec{\alpha}$ is invertible in $\ec[\Lax]{\ecat}$ (i.e., an enriched natural isomorphism) if and only if both $\hat \alpha$ and $\alpha$ are natural isomorphisms.
\end{rem}


\begin{defn}
Let $\ec[\CC]{\CM}$ and $\ec[\CD]{\CN}$ be enriched categories. We define their Cartesian product $\ec[\CC]{\CM} \times \ec[\CD]{\CN}$ as a $(\CC \times \CD)$-enriched category as follows:
\bit
\item The objects of $\ec[\CC]{\CM} \times \ec[\CD]{\CN}$ are the same as $\CM \times \CN$.
\item For any $m_1,m_2 \in \CM$ and $n_1,n_2 \in \CN$, 
\[
(\ec[\CC]{\CM} \times \ec[\CD]{\CN})((m_1,n_1),(m_2,n_2)) \coloneqq (\ec[\CC]{\CM}(m_1,m_2),\ec[\CD]{\CN}(n_1,n_2)) \in \CC \times \CD .
\]
\item The composition is defined by
\begin{align*}
& \hspace{-1cm}\phantom{=} (\ec[\CC]{\CM} \times \ec[\CD]{\CN})((m_2,n_2),(m_3,n_3)) \otimes (\ec[\CC]{\CM} \times \ec[\CD]{\CN})((m_1,n_1),(m_2,n_2)) \\
& = (\ec[\CC]{\CM}(m_2,m_3),\ec[\CD]{\CN}(n_2,n_3)) \otimes (\ec[\CC]{\CM}(m_1,m_2),\ec[\CD]{\CN}(n_1,n_2)) \\
& = (\ec[\CC]{\CM}(m_2,m_3) \otimes \ec[\CC]{\CM}(m_1,m_2) , \ec[\CD]{\CN}(n_2,n_3) \otimes \ec[\CD]{\CN}(n_1,n_2)) \\
& \xrightarrow{(\circ, \circ)} (\ec[\CC]{\CM}(m_1,m_3),\ec[\CD]{\CN}(n_1,n_3)) = (\ec[\CC]{\CM} \times \ec[\CD]{\CN})((m_1,n_1),(m_3,n_3)) .
\end{align*}
\item The identity morphism $1_{(m,n)}$ defined by $(1_m, 1_n)$: 
\small
\[
1_{(m,n)} \coloneqq \bigl( (\one_\CC,\one_\CD) \xrightarrow{(1_m, 1_n)} (\ec[\CC]{\CM}(m,m),\ec[\CD]{\CN}(n,n)) = (\ec[\CC]{\CM} \times \ec[\CD]{\CN})((m,n),(m,n)) \bigr).
\] \normalsize
\eit
(Note that the underlying category of $\ec[\CC]{\CM} \times \ec[\CD]{\CN}$ is $\CM \times \CN$.)
\end{defn}

For enriched functors $\ec{F}: \ec[\CC_1]{\CM_1} \to \ec[\CC_2]{\CM_2}$, $\ec{G}: \ec[\CD_1]{\CN_1} \to \ec[\CD_2]{\CN_2}$, the lax-monoidal functor $\hat{F} \times \hat{G}: \CC_1 \times \CD_1 \to \CC_2 \times \CD_2$, the map $F \times G: \ob(\CM_1) \times \ob(\CN_1) \to \ob(\CM_2) \times \ob(\CN_2)$, and the family of morphisms
\small
\[
    \bigl(\hat{F}(\ec[\CC_1]{\CM_1}(a, b)), \hat{G}(\ec[\CD_1]{\CN_1}(x, y)) \bigr) \xrightarrow{(\ec{F}_{a,b}, \ec{G}_{x,y})} \bigl(\ec[\CC_2]{\CM_2}(F(a), F(b)), \ec[\CD_2]{\CN_2}(G(x), G(y) \bigr) 
\] \normalsize
for all $(a, x), (b, y) \in \ob(\CM_1) \times \ob(\CN_1)$
define an enriched functor from $\ec[\CC_1]{\CM_1} \times \ec[\CD_1]{\CN_1}$ to $\ec[\CC_2]{\CM_2} \times \ec[\CD_2]{\CN_2}$. In the following, we use $\ec{F} \times \ec{G}$ to denote this enriched functor. Similarly, for enriched natural transformations $\ec{\xi}: \ec{F}_1 \Rightarrow \ec{F}_2$, $\ec{\eta}: \ec{G}_1 \Rightarrow \ec{G}_2$, where $\ec{F}_1, \ec{F}_2: \ec[\CC_1]{\CM_1} \to \ec[\CC_2]{\CM_2}$, $\ec{G}_1, \ec{G}_2: \ec[\CD_1]{\CN_1} \to \ec[\CD_2]{\CN_2}$, we use $\ec{\xi} \times \ec{\eta}$ to denote the enriched natural transformation $\ec{F}_1 \times \ec{G}_1 \Rightarrow \ec{F}_2 \times \ec{G}_2$ with background changing natural transformation $\hat{\xi} \times \hat{\eta}$, and underlying natural transformation $\xi \times \eta$.  


It can be shown that $\ec[\CC]{\CM} \times \ec[\CD]{\CN}$ is the binary product of $\ec[\CC]{\CM}$ and $\ec[\CD]{\CN}$ (in the bilimit sense) and $\ast$ is the terminal object of $\ec[\Lax]{\ecat}$. Then the following result is an easy corollary of Theorem 2.15 in \cite{CKWW08}. For the sake of completeness, we sketch a proof. 

\begin{prop}
$\ec[\Lax]{\ecat}$ is a symmetric monoidal 2-category with the tensor product given by the Cartesian product $\times$ and the tensor unit given by $\ast$. 
\end{prop}

\begin{proof}
    The tensor product of $\ec[\Lax]{\ecat}$ is defined by the 2-functor which maps $(\ec[\CC]{\CM}, \ec[\CD]{\CN})$ to $\ec[\CC]{\CM} \times \ec[\CD]{\CN}$, $(\ec{F}, \ec{G})$ to $\ec{F} \times \ec{G}$, and $(\ec{\xi}, \ec{\eta})$ to $\ec{\xi} \times \ec{\eta}$. And the tensor unit of $\ec[\Lax]{\ecat}$ is the enriched category $\ast$ whose background category and underlying category are both the trivial category with only one object and only its identity morphism. The associator and left/right unitor of $\ec[\Lax]{\ecat}$ are derived from those of symmetric monoidal 2-categories $\cat$ and $\Alg_{E_1}^{\mathrm{lax}}(\cat)$. In particular, the associator and left/right unitor are 2-natural isomorphisms, and the pentagonator and 2-unitors are all given by the identity 2-morphisms. The braiding $\ec{\Sigma} \colon \ec[\CC]{\CM} \times \ec[\CD]{\CN} \to \ec[\CD]{\CN} \times \ec[\CC]{\CM}$ is defined by the enriched functor which switches the two arguments of the background categories and underlying categories. The left/right hexagonator and syllepsis are identity modifications. All the pasting diagram axioms in the definition of symmetric monoidal 2-categories hold for $\ec[\Lax]{\ecat}$ since every $2$-morphism in the diagrams is an identity 2-morphism (see Chapter 12 in \cite{JY21} for a detailed description of pasting diagram axioms).  
\end{proof}

\begin{rem}
Note that $\ec[\CA]{\ecat}$ is not a monoidal sub-2-category of $\ec[\Lax]{\ecat}$ because the Cartesian product $\times$ is not defined in $\ec[\CA]{\ecat}$. We show in Example\,\ref{expl:pushforward} that if $\CA$ is braided, then $\ec[\CA]{\ecat}$ has a natural monoidal structure with a non-trivial tensor product. 
\end{rem}

\begin{rem}\label{rem:background_underlying_2-functors}
    There is a symmetric monoidal 2-functor from $\ec[\Lax]{\ecat}$ to $\cat$ defined by $\ec[\CA]{\CL}\mapsto \CL$, $\ec{F}\mapsto F$ and $\ec{\xi}\mapsto \xi$. There is a symmetric monoidal 2-functor from $\ec[\Lax]{\ecat}$ (resp.~$\ecat$) to $\Alg_{E_1}^{\mathrm{lax}}(\cat)$ (resp.~$\Alg_{E_1}(\cat)$) defined by $\ec[\CA]{\CL}\mapsto \CA$, $\ec{F}\to\hat{F}$ and $\ec{\xi}\mapsto \hat{\xi}$ (see \cite{Gur11, Gur13, GO13} for the definition of symmetric monoidal bifunctor). 
\end{rem}

\subsection{Pushforward 2-functors}
In this subsection, we review a push-forward 2-functor and use it to reformulate the notions of an enriched functor and an enriched natural transformation.

Let $\CA$ and $\CB$ be monoidal categories and $R: \CA \to \CB$ a lax-monoidal functor. Given an enriched category $\ec[\CA]{\CL}$, there is a $\CB$-enriched category $R_\ast(\ec[\CA]{\CL})$ \cite{EK66,Kel82,MP17} defined as follows: 
\bnu
\item $\ob(R_\ast(\ec[\CA]{\CL})) \coloneqq \ob(\ec[\CA]\CL)$;  
\item $R_\ast(\ec[\CA]{\CL})(x,y) \coloneqq R(\ec[\CA]{\CL}(x,y))$; 
\item the identity morphism is
\[
1_x \coloneqq \bigl( \one_\CB \to R(\one_\CA) \xrightarrow{R(1_x)} R(\ec[\CA]{\CL}(x,x))= R_\ast(\ec[\CA]{\CL})(x,x) \bigr) ;
\] 
\item the composition of morphisms is defined by the following composed morphisms in $\CB$: 
\[
R(\ec[\CA]{\CL}(y,z)) \otimes R(\ec[\CA]{\CL}(x,y)) \to R(\ec[\CA]{\CL}(y,z)\otimes \ec[\CA]{\CL}(x,y)) \xrightarrow{R(\circ)} R(\ec[\CA]{\CL}(x,z)). 
\]
\enu
The two defining conditions (\ref{diag:asso-circ}) and (\ref{diag:right-left-unit}) hold due to the lax-monoidality of $R$.  


\begin{rem} \label{rem:R_enriched-functor}
There is a canonical enriched functor $\ec{R}: \ec[\CA]{\CL} \to R_\ast(\ec[\CA]{\CL})$ defined by $\hat{R}:=R$, $\ob(\ec[\CA]{\CL}) \xrightarrow{1} \ob(\ec[\CA]{\CL})$ and $\ec{R}_{x,y}=1_{R(\ec[\CA]{\CL}(x,y))}$ for $x,y\in \ob(\ec[\CA]{\CL})$.
\end{rem}

Given an $\CA$-functor $F: \ec[\CA]{\CL} \to \ec[\CA]{\CM}$, i.e., a map $F: \ob(\CL) \to \ob(\CM)$ together with $F_{x,y}: \ec[\CA]{\CL}(x,y) \to \ec[\CA]{\CM}(F(x),F(y))$ for $x,y\in\CL$. Then the same map $F: \ob(\CL) \to \ob(\CM)$ together with 
\[
R(F_{x,y}): R(\ec[\CA]{\CL}(x,y)) \to R(\ec[\CA]{\CM}(F(x),F(y)))
\] 
for $x,y\in\CL$ defines a $\CB$-functor $R_\ast(F): R_\ast(\ec[\CA]{\CL}) \to R_\ast(\ec[\CA]{\CM})$. 

Given an $\CA$-natural transformation $\xi: F \Rightarrow G$ between two $\CA$-functors $F,G: 
\ec[\CA]{\CL} \to \ec[\CA]{\CM}$. For $x\in\CL$, we define a family of morphisms $R_\ast(\xi_x)$ in $\CB$: 
\[
R_\ast(\xi_x) \coloneqq \bigl( \one_\CB \to R(\one_\CA) \xrightarrow{R(\xi_x)} R(\ec[\CA]{\CM}(F(x),G(x))) \bigr).
\] 
It is straightforward to check that $R_\ast(\xi_x)$ defines an $\CA$-natural transformation $R_\ast(\xi)$. 

\medskip
The following result was essentially in \cite{EK66,Kel82} except we drop the assumption on the symmetric monoidality. 
\begin{thm}
Given a lax-monoidal functor $R: \CA \to \CB$, there is a well-defined 2-functor $R_\ast: \ec[\CA]{\ecat} \to \ec[\CB]{\ecat}$ (called the \emph{pushforward of $R$}) defined by $\ec[\CA]\CL \mapsto R_\ast(\ec[\CA]\CL)$, $F \mapsto R_\ast(F)$, $\xi \mapsto R_\ast(\xi)$. 
\end{thm} 

\begin{expl} \label{expl:pushforward}
We give a few examples of the pushforward 2-functor $R_\ast$. 
\bnu[(1)]
\item Given a monoidal category $\CA$, the functor $\CA(\one_\CA,-)$ has a lax-monoidal structure defined by $\CA(\one_\CA,x) \times \CA(\one_\CA,y) \to \CA(\one_\CA \otimes\one_\CA, x\otimes y)\simeq \CA(\one_\CA, x\otimes y)$. Then the pushforward 2-functor $\CA(\one_\CA,-)_\ast$ maps an $\CA$-enriched category to its underlying category, an $\CA$-functor to its underlying functor, and an $\CA$-natural transformation to its underlying natural transformation. 

\item Let $\CA$ be a braided monoidal category. The tensor product $\otimes: \CA \times \CA \to \CA$ is monoidal. Then $\otimes_\ast: \ec[\CA\times \CA]{\ecat} \to \ec[\CA]{\ecat}$ is a well-defined 2-functor. The Cartesian product $\times$ defines a functor $\times: \ec[\CA]{\ecat} \times \ec[\CA]{\ecat} \to \ec[\CA\times \CA]{\ecat}$. Then we obtain a composed functor
\[
\dtimes: \ec[\CA]{\ecat} \times \ec[\CA]{\ecat} \xrightarrow{\times} \ec[\CA\times \CA]{\ecat} \xrightarrow{\otimes_\ast} \ec[\CA]{\ecat}. 
\]
This functor $\dtimes$, together with the tensor unit $\ast$, endows a monoidal structure on $\ec[\CA]{\ecat}$ \cite{For04}. If $\CA$ is a symmetric monoidal category, then $\ec[\CA]{\ecat}$ is a symmetric monoidal 2-category \cite{Kel82}. 

\item Given a braided monoidal category $\CA$, let $\CB$ be a left monoidal $\CA$-module \cite{Lur17,AFT16,KZ18,HPT16,BZBJ18a}, i.e., $\CB$ is equipped with a braided monoidal functor $\phi \colon \CA \to \FZ_1(\CB)$. Let $\odot$ be the composed functor $\CA \times \CB \xrightarrow{\phi\times 1} \FZ_1(\CB) \times \CB \to \CB$. It defines a left unital $\CA$-action on $\CB$. Moreover, $\odot$ is monoidal, thus the pushforward $\odot_\ast$ is well-defined. Therefore, we obtain a composed 2-functor 
\be \label{arrow:A-action-on-B}
\ec[\CA]{\ecat} \times \ec[\CB]{\ecat} \xrightarrow{\times} \ec[\CA\times \CB]{\ecat} \xrightarrow{\odot_\ast} \ec[\CB]{\ecat},
\ee
which endows $\ec[\CB]{\ecat}$ with a structure of a left $\ec[\CA]{\ecat}$-module. This clarifies 
Remark 3.21 in \cite{KZ21}. Notice that this $\ec[\CA]{\ecat}$-action on $\ec[\CB]{\ecat}$ factors through a canonical action of $\ec[\FZ_1(\CB)]{\ecat}$ on $\ec[\CB]{\ecat}$. If, in addition, $\CA$ is symmetric, $\CB$ is braided and equipped with a braided functor $\phi: \CA \to \FZ_2(\CB)$, then the functor $\odot$ is braided monoidal. Then the $\ec[\CA]{\ecat}$-action on $\ec[\CB]{\ecat}$ factors through the canonical action of $\ec[\FZ_2(\CB)]{\ecat}$ on $\ec[\CB]{\ecat}$. 
\enu
\end{expl}

We obtain a new characterization of an enriched functor. 
\begin{prop} \label{prop:e-functor-split}
An enriched functor $\ec{F}: \ec[\CA]{\CL} \to \ec[\CB]{\CM}$ is precisely a pair $(\hat{F}, \check{F})$, where $\hat{F}: \CA\to\CB$ is a lax-monoidal functor and $\check{F}: \hat{F}_\ast(\ec[\CA]{\CL}) \to \ec[\CB]{\CM}$ is a $\CB$-functor. Moreover, in the light of Remark\,\ref{rem:R_enriched-functor}, an enriched functor $\ec{F}$ always splits into a composition, i.e., $\ec{F}= \check{F} \circ \ec{\hat{F}}$. We call it the \emph{splitting property} of an enriched functor. 
\end{prop}
\pf
This is just a reformulation of Definition\,\ref{def:1_morphism} in terms of a $\CB$-functor (recall Definition\,\ref{def:A-functor}). 
\end{proof}

\begin{rem}
Recall Example\,\ref{expl:pushforward} (3). Let $A$ be an $E_1$-algebra in $\ec[\CA]{\ecat}$ and $M$ an $A$-module in $\ec[\CB]{\ecat}$. By the splitting property of an enriched functor, the $A$-action on $M$ factors through an $\phi_\ast(A)$-action on $M$. As a consequence, for an $E_0$-algebra in $\ec[\CB]{\ecat}$ (i.e., a $\CB$-enriched category together with a distinguished object), its $E_0$-center in $\ecat$, if exists, necessarily lives in $\ec[\FZ_1(\CB)]{\ecat}$. If, in addition, $\CB$ is braided, for an $E_1$-algebra in $\ec[\CB]{\ecat}$, its $E_1$-center in $\ecat$, if exists, lives in $\ec[\FZ_2(\CB)]{\ecat}$. Similarly, if $\CB$ is symmetric, for an $E_2$-algebra in $\ec[\CB]{\ecat}$, its $E_2$-center in $\ecat$ lives in $\ec[\CB]{\ecat}$. 
\end{rem}

Let $R,R': \CA \to \CB$ be two lax-monoidal functors and $\phi: R \Rightarrow R'$ be a lax-monoidal natural transformation. This $\phi$ induces an enriched functor $\phi_\ast: R_\ast(\ec[\CA]{\CL}) \to R_\ast'(\ec[\CA]{\CL})$, called the pushforward of $\phi$. More precisely, $\phi_\ast$ is a $\CB$-functor defined as follows: 
\bnu
\item $\phi_\ast$ is the identity map on $\ob(\CL)$; 
\item on morphisms: $(\phi_\ast)_{x,y}: R_\ast(\ec[\CA]{\CL})(x,y) \to R_\ast'(\ec[\CA]{\CL})(x,y)$ is defined by 
\[
(\phi_\ast)_{x,y} \coloneqq \bigl( R_\ast(\ec[\CA]{\CL}(x,y)) \xrightarrow{\phi_{\ec[\CA]{\CL}(x,y)}} R_\ast'(\ec[\CA]{\CL}(x,y)) \bigr).
\] 
\enu
The lax-monoidality and the naturalness of $\phi$ imply that such defined $\phi_\ast$ preserves the identities and the compositions. Therefore, $\phi_\ast$ is a well-defined $\CB$-functor. We obtain a new characterization of an enriched natural transformation. 
\begin{prop}
Let $\ec{F}, \ec{G}: \ec[\CA]{\CL} \to \ec[\CB]{\CM}$ be two enriched functors. An enriched natural transformation $\ec{\xi}:\ec{F} \to \ec{G}$ is precisely a pair $(\hat{\xi}, \check{\xi})$, where $\hat{\xi}: \hat{F} \Rightarrow \hat{G}$ is a lax-monoidal natural transformation and $\check{\xi}: \check{F} \Rightarrow \check{G}\circ \hat{\xi}_\ast$ is a $\CB$-natural transformation. 
\end{prop}

\subsection{Canonical construction} \label{subsec:can_con}
In this subsection, we recall the canonical construction of enriched categories, and further study it in the new 2-category of enriched categories. The main results are Theorem\,\ref{thm:2-functor_from_A_modules_to_cat_A} and \ref{thm:fsLMod=lax_fsECat}. 
Throughout this subsection, $\CA$ and $\CB$ are monoidal categories, $\CL$ is a left $\CA$-\oplax module (recall Definition\,\ref{def:oplax-left-module}), and $\CM$ is a left $\CB$-\oplax module. 

\medskip
Recall that $\CL$ is called \textit{enriched} in $\CA$ if the internal hom $[x,y]_\CA$ exists in $\CA$ for all $x,y\in\CL$. 
We sometimes abbreviate $[x, y]_{\CA}$ to $[x,y]$ for simplicity. We use $\ev_x$ and $\coev_x$ to denote the counit $[x,y] \odot x \to y$ and unit $a \to [x, a\odot x]$ of the adjunction 
\begin{align*}
    \CL(a \odot x, y) \simeq \CA(a, [x,y]). 
\end{align*}
For any pair $(a,\phi)$ where $a \in \CA$ and $\phi : a \odot x \to y$ is a morphism in $\CL$, there exists a unique morphism $\bar \phi : a \to [x,y]$ such that the following diagram commutes:
\[
\xymatrix{
a \odot x \ar@{-->}[r]^-{\bar \phi \odot 1_x} \ar[dr]_{\phi} & [x,y] \odot x \ar[d]^{\ev_x} \\
 & y
}
\]
It follows that for any morphisms $f: x' \to x, g: y \to y'$, there exists a unique morphism $[f,g]: [x,y] \to [x', y']$ rendering the diagram commutative. 
\begin{align*}
    \xymatrix @R=0.1in{
        [x,y] \odot x' \ar[rr]^-{[f,g] \odot 1} \ar[d]^-{1 \odot f} && [x', y']\odot x' \ar[d]^-{\ev_{x'}}\\
        [x,y] \odot x \ar[r]^-{\ev_x} & y \ar[r]^-{g} & y'
    }
\end{align*}
As a consequence, the internal homs define a bifunctor $[-,-]:\CL^{\op} \times \CL \to \CA$. Indeed, $[f,-]$ is the mate of $(- \odot f)$ under the adjunction $(- \odot x) \dashv [x,-]$.

\begin{rem}
Let $(L: \CB \to \CA, R: \CA \to \CB, \eta: 1_\CB \Rightarrow RL, \varepsilon: LR \Rightarrow 1_\CA)$ be an adjunction such that $R$ is lax-monoidal. The functor $L$ is automatically oplax-monoidal \cite{Kel74}. The left $\CA$-\oplax module $\CL$ pulls back to a left $\CB$-\oplax module with the $\CB$-action defined by $L(-)\odot -: \CB \times \CL \to \CL$. If $\CL$ is enriched in $\CA$, then $\CL$ is enriched in $\CB$ with $[x,y]_\CB = R([x,y]_\CA)$. The evaluation $[x,y]_\CB\odot x \to y$ is defined by the composed morphism: $LR([x,y]_\CA)\odot x \xrightarrow{\varepsilon} [x,y]_\CA \odot x \xrightarrow{\ev_x} y$. 
\end{rem}

\begin{defn}[\cite{Lin81}] 
If $\CL$ is enriched in $\CA$, then $\CL$ can be promoted to an $\CA$-enriched category, denoted by $\bc[\CA]{\CL}$. More precisely, the enriched category $\bc[\CA]{\CL}$ has the same objects as $\CL$ and $\bc[\CA]{\CL}(x,y)=[x,y]$. The identity $1_x: \one\to [x,x]$ and composition $\circ: [y,z]\otimes[x,y]\to[x,z]$ are the unique morphisms rendering the following diagrams commutative:
\small
\[
\xymatrix @R=0.2in @C=0.25in{
    \one_\CA \odot x \ar[r]^{1_x \odot 1} \ar[rd]_{u_x} & [x,x]\odot x \ar[d]^{\ev_x}\\
    & x 
}\quad
    \xymatrix @R=0.2in @C=0.25in{
        ([y,z] \otimes [x,y]) \odot x \ar[r] \ar[d]^{\circ \odot 1} & [y,z] \odot ([x,y] \odot x) \ar[r]^-{1 \odot \ev_x} & [y,z] \odot y \ar[d]^{\ev_y}\\
        [x,z] \odot x \ar[rr]^{\ev_x} & & z 
    }
\]
\normalsize
This construction of $\bc[\CA]{\CL}$ is called the \textit{canonical construction}. 
\end{defn}

In the case of canonical construction $\bc[\CA]{\CL}$, if $\CL$ is a strongly unital, i.e., $u_x: \one_\CA \odot x \to x$ is invertible for every $x \in \CL$, then $\CL$ can be canonically identified with the underlying category of $\bc[\CA]{\CL}$ 
via an isomorphism of categories defined by the identity map on objects and the map $(f: y \to y') \mapsto (\underline{f}: \one_\CA \to [y,y'])$ on morphisms, where $\underline{f}$ is defined by the following commutative diagram. 
\begin{equation}\label{diag:can_con_identify_with_origin_rem_diag_I}
\begin{array}{c}
\xymatrix @R=0.15in @C=0.8in{
        \one_\CA \odot y \ar[d]_{u_y} \ar[r]^-{\underline{f} \odot 1} & [y,y'] \odot y \ar[d]^{\ev_y}\\
        y \ar[r]^{f} & y'
    }
\end{array}
\end{equation} 

\begin{rem} \label{rem:notation-enriched}
A brief remark to our notations $\ec[\CA]{\CL}$ and $\bc[\CA]{\CL}$. We use ``$|$'' in a generic enriched category $\ec[\CA]{\CL}$ to indicate that $\CA$ might not directly act on $\CL$. Once we remove the block ``$|$'', it suggests an $\CA$-action on $\CL$ and the canonical construction. 
\end{rem}


\medskip
An enriched functor between two enriched categories from the canonical construction has a nice characterization (see Proposition~\ref{prop:1_morphism_des}). We explain that now. In the rest of this subsection, $\CL$ and $\CM$ are assumed to be 
enriched in $\CA$ and $\CB$, respectively. 

\medskip
Let $L: \CB \to \CA$ be an oplax-monoidal functor. Then the left $\CA$-\oplax module $\CL$ pulls back along $L$ to a left $\CB$-\oplax module with the $\CB$-action defined by $L(-) \odot -:  \CB \times \CL \to \CL$. For a functor $F: \CL\to\CM$, a lax $\CB$-\oplax module functor structure on $F$ is defined by a natural transformation $\alpha = \{ \alpha_{b,x}: b\odot F(x) \to F(L(b)\odot x)\}_{a\in\CA,x\in\CL}$ satisfying some natural axioms (recall Definition~\ref{defn:lax_module_functor}). We also call this natural transformation $\alpha$ as an \emph{$L$-oplax structure} of $F$. We introduce a new structure on $F$ that is in some sense dual to the $L$-oplax structure of $F$.

\begin{defn} \label{def:R_lax_module_functor}
Let $R: \CA \to \CB$ be a lax-monoidal functor and $F: \CL \to \CM$ be a functor. An \emph{$R$-lax structure} on $F$ is a natural transformation $\beta = \{\beta_{a, x} \colon R(a) \odot F(x) \to F(a \odot x)\}_{a\in\CA, x\in\CL}$ rendering the following diagrams commutative.
\be \label{diag:R_lax_module_functor_con}
\begin{gathered}
\begin{array}{c}
\xymatrix @R=0.2in @C=0.3in{
(R(a) \otimes R(b)) \odot F(x) \ar[r] \ar[d] & R(a \otimes b) \odot F(x) \ar[r]^-{\beta_{a \otimes b,x}} & F((a \otimes b) \odot x) \ar[d] \\
R(a) \odot (R(b) \odot F(x)) \ar[r]^-{1 \odot \beta_{b,x}} & R(a) \odot F(b \odot x) \ar[r]^-{\beta_{a,b \odot x}} & F(a \odot (b \odot x))
}
\end{array} \\
\begin{array}{c}
\xymatrix @R=0.2in @C=0.2in{
\one_\CB \odot F(x) \ar[r] \ar[d] & R(\one_\CA) \odot F(x) \ar[d]^-{\beta_{\one_\CA,x}} \\
F(x) & F(\one_\CA \odot x) \ar[l]
}
\end{array}
\end{gathered}
\ee 
The functor $F$ equipped with an $R$-lax structure is called an \emph{$R$-lax functor}. 
\end{defn}

\begin{rem}
When $\CA=\CB, R=1_\CA$, an $R$-lax functor becomes a lax $\CB$-\oplax module functor. In general, an $R$-lax functor $F$ is not a ``module functor'' in any sense because $\CA \to \CB \to \fun(\CM,\CM)$ (as the composition of a lax-monoidal functor and an oplax-monoidal functor) is neither lax-monoidal nor oplax-monoidal. 
\end{rem}

\begin{rem}\label{rem:mot_def_R_lax_functor}
We explain the duality between an $L$-oplax structure and an $R$-lax structure. 
Let $(L: \CB \to \CA, R: \CA \to \CB, \eta: 1_\CB \Rightarrow RL, \varepsilon: LR \Rightarrow 1_\CA)$ be an adjunction such that $R$ is lax-monoidal. Then $L$ is automatically oplax-monoidal \cite{Kel74} with the oplax structure maps defined by 
\begin{gather*}
L(\one_\CB) \to LR(\one_\CA) \xrightarrow{\varepsilon} \one_\CA, \\
L(a \otimes b) \xrightarrow{L(\eta \otimes \eta)} L(RL(a) \otimes RL(b)) \to LR(L(a) \otimes L(b)) \xrightarrow{\varepsilon} L(a) \otimes L(b).
\end{gather*}
Let $\{\beta_{a, x}: R(a) \odot F(x) \to F(a \odot x)\}_{a\in\CA,x\in\CL}$ be an $R$-lax structure on $F$, then   
\begin{align*}
    \alpha_{b, x} \coloneqq \bigl(b \odot F(x) \xrightarrow{\eta \odot 1} RL(b) \odot F(x) \xrightarrow{\beta_{L(b),x}} F(L(b) \odot x) \bigr), \quad \forall b\in\CB, x\in\CL
\end{align*}   
define an $L$-oplax structure on $F$.

Conversely, let $\{\alpha_{b,x}: b \odot F(x) \to F(L(b) \odot x)\}_{b\in\CB,x\in\CL}$ be an $L$-oplax structure on $F$, then   
\[
    \beta_{a,x} \coloneqq \bigl(R(a) \odot F(x) \xrightarrow{\alpha_{R(a),x}} F(LR(a) \odot x) \xrightarrow{F(\varepsilon_a \odot 1)} F(a \odot x) \bigr),  \;\; \forall a\in\CA, x\in\CL
\]   
define an $R$-lax structure on $F$. These two constructions are mutually inverse. In other words, there is a bijection between the set of $R$-lax structures on $F$ and that of $L$-oplax structures on $F$. 
\end{rem}

The $R$-lax structure gives a characterization of an enriched functor between two enriched categories from the canonical construction as explained in the following proposition.

\begin{prop} \label{prop:1_morphism_des}
Suppose $\CL$ and $\CM$ as left \oplax modules are strongly unital. Let $\ec{F} : \bc[\CA]{\CL} \to \bc[\CB]{\CM}$ be an enriched functor. The underlying functor $F : \CL \to \CM$ with the natural transformation
\begin{equation} \label{eq:beta_ax}
\beta_{a, x} \coloneqq \bigl( \hat{F}(a) \odot F(x) \xrightarrow{\hat{F}(\coev_x)} \hat{F}([x, a \odot x]) \odot F(x) \xrightarrow{\ec{F}_{x, a\odot x}} [F(x), F(a \odot x)] \odot F(x) \xrightarrow{\ev_{F(x)}} F(a \odot x) \bigr)
\end{equation}
is an $\hat F$-lax functor. Conversely, given a lax-monoidal functor $\hat F : \CA \to \CB$ and an $\hat F$-lax functor $F : \CL \to \CM$ with the $\hat F$-lax structure $\beta_{a,x} : \hat F(a) \odot F(x) \to F(a \odot x)$, the morphisms
\begin{multline} \label{eq:1_morphism_des_2}
\ec{F}_{x,y} \coloneqq \bigl( \hat F([x,y]) \xrightarrow{\coev_{F(x)}} [F(x),\hat F([x,y]) \odot F(x)] \\
\xrightarrow{[F(x),\beta_{[x,y],x}]} [F(x),F([x,y] \odot x)] \xrightarrow{[F(x),F(\ev_x)]} [F(x),F(y)] \bigr),
\end{multline}
together with $\hat F$ and $F$, define an enriched functor $\ec{F} : \bc[\CA]{\CL} \to \bc[\CB]{\CM}$. Moreover, these two constructions are mutually inverse.
\end{prop}

\pf
Given an enriched functor $\ec{F} : \bc[\CA]{\CL} \to \bc[\CB]{\CM}$, we need to show the natural transformation $\beta_{a,x}$ defined in (\ref{eq:beta_ax}) satisfies two diagrams in \eqref{diag:R_lax_module_functor_con}. The first diagram is the outer diagram of the following commutative diagram
\scriptsize
\[
\xymatrix@C=1.8em{
\hat F(a) \hat F(b) F(x) \ar[r]^{\hat F^2 1} \ar[d]_{1 \hat F(\coev_x) 1} & \hat F(ab) F(x) \ar[r]^{\hat F(\coev_x) 1} \ar[dr]^-{\hat F(\coev_{bx} \coev_x)} & \hat F([x,abx]) F(x) \ar@/^12ex/[dd]^(0.25){\ec{F}_{x,abx} 1} \\
\hat F(a) \hat F([x,bx]) F(x) \ar[r]^-{\hat F(\coev_{bx}) 1 1} \ar[d]_{1 \ec{F}_{x,bx} 1} & \hat F([bx,abx]) \hat F([x,bx]) F(x) \ar[r]^{\hat F^2 1} \ar[d]^{\ec{F}_{bx,abx} \ec{F}_{x,bx} 1} & \hat F([bx,abx] [x,bx]) F(x) \ar[u]_{\hat F(\circ) 1} \\
\hat F(a) [F(x),F(bx)] F(x) \ar[d]_{1 \ev_{F(x)}} & [F(bx),F(abx)] [F(x),F(bx)] F(x) \ar[r]^-{\circ 1} \ar[dr]^{1 \ev_{F(x)}} & [F(x),F(abx)] F(x) \ar@/^12ex/[dd]^(0.3){\ev_{F(x)}} \ar@{}[ul]|{\bigstar} \\
\hat F(a) F(bx) \ar[r]_-{\hat F(\coev_{bx}) 1} & \hat F([bx,abx]) F(bx) \ar[r]_-{\ec{F}_{bx,abx} 1} & [F(bx),F(abx)] F(bx) \ar[d]_{\ev_{bx}} \\
 & & F(abx) 
}
\] \normalsize
where the pentagon $\star$ commutes because $\ec{F}$ is an enriched functor. The second diagram is the outer diagram of the following commutative diagram
\[
\xymatrix{
\one F(x) \ar[r]^{\hat F^0 1} \ar[ddd] \ar[ddr]_{1_{F(x)} 1} \ar@{}[dr]|{\bigstar} & \hat F(\one) F(x) \ar[d]^{\hat F(1_x) 1} \ar[dr]^{\hat F(\coev_x) 1} \\
 & \hat F([x,x]) F(x) \ar[d]^{\ec{F}_{x,x} 1} & \hat F([x,\one x]) F(x) \ar[d]^{\ec{F}_{x,\one x} 1} \ar[l] \\
 & [F(x),F(x)] F(x) \ar[dl]_{\ev_{F(x)}} & [F(x),F(\one x)] F(x) \ar[d]^{\ev_{F(x)}} \ar[l] \\
F(x) & & F(\one x) \ar[ll]
}
\]
where the subdiagram $\star$ commutes because $\ec{F}$ is an enriched functor.

Conversely, suppose $\beta_{a,x}$ is an $\hat F$-lax structure of $F$, we need to show that the morphisms $\ec{F}_{x,y}$, together with $\hat F$ and $F$, define an enriched functor $\ec{F} : \bc[\CA]{\CL} \to \bc[\CB]{\CM}$. That $\ec{F}$ preserves the identity morphisms follows from the commutativity of the outer diagram of the following commutative diagram
\[
\xymatrix@C=2em{
\one \ar[r]^{\hat F^0} \ar[d]^{\coev_{F(x)}} \ar@/_8ex/[dd]_{1_{F(x)}} & \hat F(\one) \ar[r]^{\hat F(1_x)} \ar[d]^{\coev_{F(x)}} & \hat F([x,x]) \ar[d]^{\coev_{F(x)}} \\
[F(x),\one \odot F(x)] \ar[r] \ar[d] & [F(x),\hat F(\one) \odot F(x)] \ar[r] & [F(x),\hat F([x,x]) \odot F(x)] \ar[d]^{[1,\beta_{[x,x],x}]} \\
[F(x),F(x)] & & [F(x),F([x,x] \odot x)]\, , \ar[ll] \ar@{}[ull]|{\bigstar}
}
\]
where $\star$ commutes by the definition of $\hat F$-lax structure. Using the adjunction $(- \odot x) \dashv [x,-]$, the condition \eqref{diag:1_morphism_diag} is equivalent to the commutativity of the outer square of the following diagram:
\[
\xymatrix{
\hat F([y,z]) \hat F([x,y]) F(x) \ar[r]^{\hat F^2 1} \ar[d]^{1 \beta_{[x,y],x}} \ar@{}[dr]|{\bigstar} & \hat F([y,z] [x,y]) F(x) \ar[r]^{\hat F(\circ) 1} \ar[d]^{\beta_{[y,z][x,y],x}} & \hat F([x,z]) F(x) \ar[d]^{\beta_{[x,z],x}} \\
\hat F([y,z]) F([x,y] x) \ar[r]^{\beta_{[y,z],[x,y] x}} \ar[d]^{1 F(\ev_x)} & F([y,z] [x,y] x) \ar[r]^{F(\circ 1)} \ar[d]^{F(1 \ev_x)} & F([x,z] x) \ar[d]^{F(\ev_x)} \\
\hat F([y,z]) F(y) \ar[r]^{\beta_{[y,z],y}} & F([y,z] y) \ar[r]^{F(\ev_y)} & F(z),
}
\]
where the square $\star$ commutes because $\beta$ is an $\hat F$-lax structure of $F$.

It is easy to verify that these two constructions are mutually inverse.
\end{proof}

\begin{rem}
The special case of Proposition \ref{prop:1_morphism_des} when $\CB = \CA$ and $\hat F = 1_\CA$ has been proved in \cite{Lin81,MPP18}.
\end{rem}

\begin{defn}\label{def:morphism_bet_R_lax_functors}
For $i=1,2$, let $R_i:\CA \to \CB$ be a lax-monoidal functor and $F_i: \CL \to \CM$ an $R_i$-lax functor. Given a lax-monoidal natural transformation $\hat{\xi}: R_1 \Rightarrow R_2$, a natural transformation $\xi:F_1 \Rightarrow F_2$ is \emph{$\hat{\xi}$-lax} or called a \emph{$\hat{\xi}$-lax natural transformation} if the following diagram commutes
\be \label{diag:morphism_bet_R_lax_functors_con}
\begin{array}{c}
\xymatrix @R=0.2in @C=0.4in{
R_1(a) \odot F_1(x) \ar[r] \ar[d]^{\hat{\xi}_a \odot \xi_x} & F_1(a \odot x) \ar[d]^{\xi_{a \odot x}} \\
R_2(a) \odot F_2(x) \ar[r] & F_2(a \odot x)\ ,
}
\end{array}
\ee
where the unlabeled arrows are given by the $R_i$-lax structure on $F_i$.
\end{defn}

\begin{rem}
 The definition \ref{def:morphism_bet_R_lax_functors} is motivated by the following fact. Let $(L: \CB \to \CA, R: \CA \to \CB, \eta: 1_\CB \Rightarrow RL, \varepsilon: LR \Rightarrow 1_\CA)$ be an adjunction such that $R$ is lax-monoidal and $F_1, F_2: \CL \to \CM$ be $R$-lax functors. By Remark \ref{rem:mot_def_R_lax_functor}, $F_i$ is equipped with an $L$-oplax structure induced by the $R$-lax structure on $F_i$. Then a natural transformation $\xi: F_1 \Rightarrow F_2$ is $1_R$-lax if and only if $\xi$ is a $\CB$-module natural transformation.  
\end{rem}

%

\begin{prop}\label{prop:2_morphism_des}
    For $i = 1,2$, let $\hat{F}_i:\CA \to \CB$ be a lax-monoidal functor and $F_i : \CL \to \CM$ an $\hat F_i$-lax functor. Then we have two enriched functors $\ec{F_1}, \ec{F_2} : \bc[\CA]{\CL} \to \bc[\CB]{\CM}$. Given a lax-monoidal natural transformation $\hat{\xi}:\hat{F}_1 \to \hat{F}_2$, a natural transformation $\xi:F_1 \to F_2$ is $\hat{\xi}$-lax if and only if $(\hat{\xi}, \xi)$ defines an enriched natural transformation $\ec{\xi}: \ec{F}_1 \to \ec{F}_2$.
\end{prop}

\begin{proof}
Given an enriched natural transformation $\ec{\xi} : \ec{F_1} \Rightarrow \ec{F_2}$, we need to show that $(\hat \xi,\xi)$ satisfies the diagram \eqref{diag:morphism_bet_R_lax_functors_con}. This is the outer diagram of the following commutative diagram
\scriptsize
\[
\xymatrix@C=3em{
\hat F_1(a) F_1(x) \ar[r]^-{\hat F_1(\coev_x) 1} \ar[dd]_{\hat \xi_a \xi_x} & \hat F_1([x,ax]) F_1(x) \ar[r]^-{(\ec{F_1})_{x,ax} 1} \ar[d]^{\hat \xi_{[x,ax]} 1} & [F_1(x),F_1(ax)] F_1(x) \ar[r]^-{\ev_{F_1(x)}} \ar[dr]^{[1,\xi_{ax}] 1} \ar@{}[d]|{\bigstar} & F_1(ax) \ar@/^11ex/[dd]^(0.3){\xi_{ax}} \\
 & \hat F_2([x,ax]) F_1(x) \ar[r]^-{(\ec{F_2})_{x,ax} 1} \ar[d]^{1 \xi_x} & [F_2(x),F_2(ax)] F_1(x) \ar[r]^{[\xi_x,1] 1} \ar[d]^{1 \xi_x} & [F_1(x),F_2(ax)] F_1(x) \ar[d]_{\ev_{F_1(x)}} \\
\hat F_2(a) F_2(x) \ar[r]^-{\hat F_2(\coev_x) 1} & \hat F_2([x,ax]) F_2(x) \ar[r]^-{(\ec{F_2})_{x,ax} 1} & [F_2(x),F_2(ax)] F_2(x) \ar[r]^-{\ev_{F_2(x)}} & F_2(ax) & 
}
\] \normalsize
where the pentagon $\star$ commutes because $\ec{\xi}$ is an enriched natural transformation.

Conversely, suppose the natural transformation $\xi: F_1 \Rightarrow F_2$ is $\hat{\xi}$-lax. Then the naturality of $\ec{\xi}$ is the outer diagram of the following commutative diagram
\scriptsize
\[
\xymatrix@C=1.8em{
\hat F_1([x,y]) \ar[r]^-{\coev_{F_1(x)}} \ar[d]^{\hat \xi_{[x,y]}} & [F_1(x),\hat F_1([x,y]) F_1(x)] \ar[r] \ar[d]^{[1,\hat \xi_{[x,y]} \xi_x]} \ar@{}[dr]|{\bigstar} & [F_1(x),F_1([x,y] x)] \ar[d]^{[1,\xi_{[x,y] x}]} \ar[dr]^{[1,\ev_x]} \\
\hat F_2([x,y]) \ar[d]^{\coev_{F_2(x)}} & [F_1(x),\hat F_2([x,y]) F_2(x)] \ar[r] & [F_1(x),F_2([x,y] x)] \ar[dr]^{[1,F_2(\ev_x)]} & [F_1(x),F_1(y)] \ar[d]^{[1,\xi_y]} \\
[F_2(x),\hat F_2([x,y]) F_2(x)] \ar[r] \ar[ur]^{[\xi_x,1]} & [F_2(x),F_2([x,y] x)] \ar[r]^{[1,F_2(\ev_x)]} \ar[ur]^{[\xi_x,1]} & [F_2(x),F_2(y)] \ar[r]^{[\xi_x,1]} & [F_1(x),F_2(y)]
}
\] \normalsize
where the square $\star$ commutes due to the diagram \eqref{diag:morphism_bet_R_lax_functors_con}.
\end{proof}

Let $\mathbf{LMod}$ be the 2-category defined as follows. 
\bit
    \item The objects are pairs $(\CA, \CL)$, where $\CA$ is a monoidal category and $\CL$ is a strongly unital left $\CA$-\oplax module that is enriched in $\CA$.

    \item A 1-morphism $(\CA, \CL) \to (\CB, \CM)$ is a pair $(\hat{F}, F)$, where $\hat{F}: \CA \to \CB$ is a lax-monoidal functor and $F:\CL \to \CM$ is an $\hat{F}$-lax functor. 

    \item A 2-morphism $(\hat{F}, F) \Rightarrow (\hat{G}, G)$ is a pair $(\hat{\xi},\xi)$, where
    $\hat{\xi}: \hat{F} \Rightarrow \hat{G}$ is a lax-monoidal natural transformation and $\xi:F\Rightarrow G$ is a $\hat{\xi}$-lax natural transformation. 
\eit
The horizontal/vertical composition is induced by the horizontal/vertical composition of functors and natural transformations. 
The 2-category $\mathbf{LMod}$ is symmetric monoidal with the tensor product defined by the Cartesian product and the tensor unit given by $(\ast, \ast)$. And the braiding $(\CA, \CL) \times (\CB, \CM) = (\CA \times \CB, \CL \times \CM) \to (\CB \times \CA, \CM \times \CL) = (\CB, \CM) \times (\CA, \CL)$ is defined by the functors $(a, b) \mapsto (b, a): \CA \times \CB \to \CB \times \CA$ and $(l, m) \mapsto (m, l): \CL \times \CM \to \CM \times \CL$. \\

By Proposition~\ref{prop:1_morphism_des} and Proposition~\ref{prop:2_morphism_des}, we immediately obtain the following result.


\begin{thm}\label{thm:2-functor_from_A_modules_to_cat_A}
The canonical construction can be promoted to a symmetric monoidal locally isomorphic 2-functor from $\mathbf{LMod}$ to $\ec[\Lax]{\ecat}$ defined as follows. 
    \bit
        \item The image of $(\CA, \CL)$ is the $\CA$-enriched category $\bc[\CA]{\CL}$ defined by the canonical construction. 

        \item The image of a 1-morphism $(\hat{F}, F):(\CA, \CL) \to (\CB, \CM)$ is the enriched functor $\bc[\CA]{\CL} \to \bc[\CB]{\CM}$ defined by Proposition \ref{prop:1_morphism_des}.

        \item The image of a 2-morphism $(\hat{\xi}, \xi)$ is the enriched natural transformation $\ec{\xi}$ defined by the background changing natural transformation $\hat \xi$ and the underlying natural transformation $\xi$.  
    \eit
\end{thm}

\begin{proof} 
    For every 1-morphism $(\hat{F}, F): (\CA, \CL) \to (\CB, \CM)$, we use $\ec{F}$ to denote the enriched functor from $\bc[\CA]{\CL} \to \bc[\CB]{\CM}$ defined by Proposition \ref{prop:1_morphism_des}. Note that the assignments $(\hat{F}, F) \mapsto \ec{F}$ and $(\hat{\xi}, \xi) \mapsto \ec{\xi}$ map the identity 1-morphisms and identity 2-morphisms to identity 1-morphisms and identity 2-morphisms, respectively. Let $(\hat{F}_1, F_1)$, $(\hat{F}_2, F_2)$, $(\hat{F}_3, F_3)$ be 1-morphisms from $(\CA, \CL)$ to $(\CB, \CM)$. For 2-morphisms $(\hat{\xi}_1, \xi_1): (\hat{F}_1, F_1) \Rightarrow (\hat{F}_2, F_2)$ and $(\hat{\xi}_2, \xi_2): (\hat{F}_2, F_2) \Rightarrow (\hat{F}_3, F_3)$, the 2-morphism $(\hat{\xi}_2 \hat{\xi}_1, \xi_2 \xi_1)$ is mapped to $\ec{\xi}_2\ec{\xi}_1$ since the vertical composition of $\ec{\xi}_1$ and $\ec{\xi}_2$ is defined by $\hat{\xi}_2 \hat{\xi}_1$ and $\xi_2 \xi_1$.

    Let $(\hat{G}, G): (\CB, \CM) \to (\CC, \CN)$ be a 1-morphism. By the definition of $\ec{F}_{x,y}$ and $\ec{G}_{F(x), F(y)}$ (see equation \eqref{eq:1_morphism_des_2}), we have
    \small
    \begin{align*}
        & \biggl( \hat{G}\hat{F}([x, y]) \odot GF(x) \xrightarrow{\left [ \ec{G}_{F(x), F(y)} \hat{G}(\ec{F}_{x,y}) \right] \odot 1} [GF(x), GF(y)] \odot GF(x) \xrightarrow{\ev_{GF(x)}} GF(y) \biggr) \\
        =& \biggl(\hat{G}\hat{F}([x, y]) \odot GF(x) \to G(\hat{F}([x, y]) \odot F(x)) \to GF([x,y]\odot x) \xrightarrow{GF(\ev_x)} GF(y) \biggr),  
    \end{align*} \normalsize
    where the unlabled arrows are induced by the $\hat{F}$-lax structure of $F$ and $\hat{G}$-lax structure of $G$. This implies that the image of $(\hat{G}\hat{F}, GF)$ equals $\ec{G} \ec{F}$. Recall that the horizontal compostions of 2-morphisms in $\mathbf{LMod}$ and $\ec[\Lax]{\ecat}$ are defined by the horozontal compositions of natural transformations. The assignments $(\CA, \CL) \mapsto \bc[\CA]{\CL}$, $(\hat{F}, F) \mapsto \ec{F}$ and $(\hat{\xi}, \xi) \mapsto \ec{\xi}$ define a 2-functor. By Proposition~\ref{prop:1_morphism_des} and Proposition~\ref{prop:2_morphism_des}, every local functor of the 2-functor is an isomorphism.
 
Next, we describe the braided monoidal structure of the canonical construction $2$-functor (we refer the reader to Definition 4.10 in \cite{Gur13} and Section 2.4 in \cite{Gur11} for the definition of braided monoidal bifunctors). Let $(\CA, \CL)$, $(\CB, \CM) \in \mathbf{LMod}$. Note that $\bc[\CA]{\CL} \times \bc[\CB]{\CM}$ and $\bc[(\CA \times \CB)]{(\CL \times \CM)}$ have the same background category and underlying category. For every $x_1, x_2 \in \ob(\CL)$ and $y_1, y_2 \in \ob(\CM)$, the following two hom objects 
\begin{align*}
    (\bc[\CA]{\CL} \times \bc[\CB]{\CM})((x_1, y_1), (x_2, y_2)) & = ([x_1, x_2], [y_1, y_2]), \\
    \bc[(\CA \times \CB)]{(\CL \times \CM)}((x_1, y_1), (x_2, y_2)) & = ([x_1, x_2]', [y_1, y_2]')
\end{align*}
may not be identical; nonetheless, as stipulated by the canonical construction's definition, they are both internal hom of the strongly unital left $(\CA \times \CB)$-\oplax module $\CL \times \CM$. Then the identity functor of $\CA \times \CB$, the identity map of $\ob(\CL) \times \ob(\CM)$ and the canonical isomorphisms $[x_1, x_2] \xrightarrow{\sim} [x_1, x_2]'$ and $[y_1, y_2] \xrightarrow{\sim} [y_1, y_2]'$ rendering the following diagrams commutative
\[
    \xymatrix @C=0.2in @R=0.1in{
        [x_1, x_2] \odot x_1 \ar[rr]^{\sim} \ar[rd]_{\ev_{x_1}} & & [x_1, x_2]' \odot x_1 \ar[ld]^{\ev'_{x_1}}\\
        & x_2 &
    } \quad 
    \xymatrix @C=0.2in @R=0.1in{
        [y_1, y_2] \odot y_1 \ar[rr]^{\sim} \ar[rd]_{\ev_{y_1}} & & [y_1, y_2]' \odot y_1 \ar[ld]^{\ev'_{y_1}}\\
        & y_2 &
    }
\]
define an enriched isomorphism, denote by $\ec{\chi_{\bc[\CA]{\CL}, \bc[\CB]{\CM}}}$, from $\bc[\CA]{\CL} \times \bc[\CB]{\CM}$ to $\bc[(\CA \times \CB)]{(\CL \times \CM)}$. Note that the background changing functor and the underlying functor of $\ec{\chi_{\bc[\CA]{\CL}, \bc[\CB]{\CM}}}$ are both identity functors. For $(\hat{F}, F): (\CA_1, \CL_1) \to (\CA_2, \CL_2)$ and $(\hat{G}, G): (\CB_1, \CM_1) \to (\CB_2, \CM_2)$, we claim that 
\begin{align*}
    \ec{\chi_{\bc[\CA_2]{\CL_2}, \bc[\CB_2]{\CM_2}}} \circ (\ec{F} \times \ec{G}) = \ec{(F \times G)} \circ \ec{\chi_{\bc[\CA_1]{\CL_1}, \bc[\CB_1]{\CM_1}}}.
\end{align*}
    Indeed, the background changing functor and underlying functor of the above two functors are $\hat{F} \times \hat{G}$ and $F \times G$, respectively. And the morphisms
\scriptsize
\[
    (\hat{F} \times \hat{G})\left[ (\bc[\CA_1]{\CL_1} \times \bc[\CB_1]{\CM_1})((x_1, y_1), (x_2, y_2))\right ] \to \bc[(\CA_2 \times \CB_2)]{(\CL_2 \times \CM_2)}\left[ (F(x_1), G(y_1)), (F(x_2), G(y_2)) \right] 
\] \normalsize
for both enriched functors are induced by the morphisms
\begin{gather*}
\hat{F}([x_1, x_2]) \odot F(x_1) \to F([x_1, x_2]\odot x_1) \xrightarrow{F(\ev_{x_1})} F(x_2), \\
\hat{G}([y_1, y_2]) \odot G(y_1) \to G([y_1, y_2]\odot y_1) \xrightarrow{G(\ev_{y_1})} G(y_2),
\end{gather*}
where the unlabeled arrows are induced by the $\hat{F}$-lax structure of $F$ and the $\hat{G}$-lax structure of $G$. For $(\hat{\xi}, \xi): (\hat{F}_1, F_1) \Rightarrow (\hat{F}_2, F_2)$ and $(\hat{\eta}, \eta): (\hat{G}_1, G_1) \Rightarrow (\hat{G}_2, G_2)$, we have
\begin{align*}
    &\xymatrix @C=0.3in @R=0.1in{
        \bc[\CA_1]{\CL_1} \times \bc[\CB_1]{\CM_1} \ar @/^1pc/[rr]^{\ec{F_1} \times \ec{G_1}} \ar @/_1pc/[rr]_{\ec{F_2} \times \ec{G_2}} \rrtwocell<\omit>{\quad \ec{\xi} \times \ec{\eta}} && \bc[\CA_2]{\CL_2} \times \bc[\CB_2]{\CM_2} \ar[r]^-{\ec{\chi}} & \bc[(\CA_2 \times \CB_2)]{(\CL_2 \times \CM_2)}
    } \\
    =& 
    \xymatrix @C=0.3in @R=0.1in{
        \bc[\CA_1]{\CL_1} \times \bc[\CB_1]{\CM_1} \ar[r]^-{\ec{\chi}} & \bc[(\CA_1 \times \CB_1)]{(\CL_1 \times \CM_1)} \ar @/^1pc/[rr]^{\ec{(F_1 \times G_1)}} \ar @/_1pc/[rr]_{\ec{(F_2 \times G_2}} \rrtwocell<\omit>{\quad \ec{(\xi \times \eta)}} && \bc[(\CA_2 \times \CB_2)]{(\CL_2 \times \CM_2)},
    }
\end{align*}
    since the background changing natural transformation and the underlying natural transformation of both enriched natural transformations are $\hat{\xi} \times \hat{\eta}$ and $\xi \times \eta$, respectively. By \cite[Proposition 4.2.11]{JY21}, the enriched isomorphisms $\ec{\chi_{\bc[\CA_2]{\CL_2}, \bc[\CB_2]{\CM_2}}}$ define a 2-natural isomorphism
\begin{align*}
    \xymatrix @C=0.5in @R=0.2in{
        \mathbf{LMod} \times \mathbf{LMod} \ar[r] \ar[d]^{\times} \drtwocell<\omit>{<-0.5>\ec{\chi}} & \ec[\Lax]{\ecat} \times \ec[\Lax]{\ecat} \ar[d]^{\times} \\
        \mathbf{LMod} \ar[r] & \ec[\Lax]{\ecat}.
    }
\end{align*}
Note that the image of $(\ast, \ast)$ is the tensor unit $\ast$ of $\ec[\Lax]{\ecat}$. It is routine to check that the following diagrams
\small
\begin{gather*}
    \xymatrix @C=0.6in @R=0.15in{
        (\bc[\CA]{\CL} \times \bc[\CB]{\CM}) \times \bc[\CC]{\CN} \ar[d] \ar[r]^-{\ec{\chi} \times 1} & \bc[(\CA \times \CB)]{(\CL \times \CM)} \times \bc[\CC]{\CN} \ar[r]^-{\ec{\chi}} & \bc[((\CA \times \CB) \times \CC)]{((\CL \times \CM) \times \CN)} \ar[d]\\
        \bc[\CA]{\CL} \times (\bc[\CB]{\CM} \times \bc[\CC]{\CN}) \ar[r]^-{1 \times \ec{\chi}} & \bc[\CA]{\CL} \times \bc[(\CB \times \CC)]{(\CM \times \CN)} \ar[r]^-{\ec{\chi}} & \bc[(\CA \times (\CB \times \CC))]{(\CL \times (\CM \times \CN))}
    } \\
   \xymatrix @C=0.25in @R=0.15in{
       \ast \times \bc[\CA]{\CL} \ar[rd] \ar[r]^-{\ec{\chi}} & \bc[(\ast \times \CA)]{(\ast \times \CL)}\ar[d]\\
       & \bc[\CA]{\CL}
   } \quad 
   \xymatrix @C=0.25in @R=0.15in{
       \bc[\CA]{\CL} \times \ast \ar[rd] \ar[r]^-{\ec{\chi}} & \bc[(\CA \times \ast)]{(\CL \times \ast)}\ar[d]\\
       & \bc[\CA]{\CL}
   }\quad
    \xymatrix @C=0.25in @R=0.15in{
        \bc[\CA]{\CL} \times \bc[\CB]{\CM} \ar[d] \ar[r]^-{\ec{\chi}} &  \bc[(\CA \times \CB)]{(\CL \times \CM)} \ar[d]\\
        \bc[\CB]{\CM} \times \bc[\CA]{\CL} \ar[r]^-{\ec{\chi}} & \bc[(\CB \times \CA)]{(\CM \times \CL)}
    }
\end{gather*} \normalsize
commute, where the unlabeled arrows are induced by the symmetric monoidal structures of $\mathbf{LMod}$ and $\ec[\Lax]{\ecat}$. Therefore, the invertible modifications required by the definition of a braided monoidal bifunctor can all be chosen as identity modifications. And all the pasting diagram axioms in the definition of symmetric monoidal bifunctors (see \cite[Definition 1.5]{GO13}) are fulfilled, as every 2-morphism in the diagrams is an identity 2-morphism. 
\end{proof}




In physical applications, $\CA$ and $\CL$ are often finite semisimple. In this case, we obtain a stronger result of the canonical construction. A finite category over a ground field $k$ is a $k$-linear category $\CC$ that is equivalent to the category of finite-dimensional modules over a finite-dimensional $k$-algebra $A$ (see \cite[Definition 1.8.6]{EGNO15} for an intrinsic definition). We say that $\CC$ is semisimple if $A$ is semisimple. 
We use $\fscat$ to denote the 2-category of finite semisimple categories, $k$-linear functors and natural transformations. 

\begin{rem}
  Let $\ec[\CA]{\CL}$ be an enriched category. If the background category $\CA$ is a $k$-linear category such that the tensor product $\otimes: \CA \times \CA \to \CA$ is $k$-bilinear, then underlying category $\CL$ is also a $k$-linear category.  
\end{rem}

\begin{defn}
An enriched category $\ec[\CA]{\CL}$ is called finite (semisimple) if $\CA$ and $\CL$ are both finite (semisimple) categories and the tensor product $\otimes: \CA \times \CA \to \CA$ is $k$-bilinear.  
\end{defn}

\begin{expl} \label{expl:cc-fusion-cat}
Let $\CA$ be a multi-fusion category and $\CL$ a finite (semisimple) left $\CA$-module. Then $\CL$ is enriched in $\CA$ and the canonical construction $\bc[\CA]{\CL}$ is a finite (semisimple) enriched category. 
\end{expl}

\begin{defn} \label{def:multi_linear_functor}
Let $\ec[\CB]{\CM}$ and $\ec[\CA_i]{\CL_i}$ be finite enriched categories, $i = 1, \ldots, n$. An enriched functor $\ec{F}: \ec[\CA_1]{\CL_1} \times \cdots \times \ec[\CA_n]{\CL_n} \to \ec[\CB]{\CM}$ is called \textit{multi-$k$-linear} (or \emph{$k$-linear} if $n=1$) if the background changing functor $\hat F: \CA_1 \times \cdots \times \CA_n \to \CB$ is multi-$k$-linear. 
\end{defn}

\begin{rem} 
Let $\ec{F}: \ec[\CA_1]{\CL_1} \times \cdots \times \ec[\CA_n]{\CL_n} \to \ec[\CB]{\CM}$ be a multi-$k$-linear enriched functor. It is easy to see from (\ref{equ:def_underlying_cat_2}) that the underlying functor $F: \CL_1 \times \cdots \times \CL_n \to \CM$ is also multi-$k$-linear.
\end{rem}

We denote the 2-category consisting of finite semisimple enriched categories, $k$-linear enriched functors and enriched natural transformations by $\ec[\Lax]{\fsecat}$. 

\medskip
Let $\fslmod$ be the 2-category defined as follows.
\bit
    \item An object is a pair $(\CA, \CL)$, where $\CA$ is a finite semisimple monoidal category and $\CL$ is a strongly unital finite semisimple left $\CA$-\oplax module that is enriched in $\CA$, such that the $\CA$-action functor $\odot \colon \CA \times \CL \to \CL$ is $k$-bilinear.

    \item A 1-morphism $(\CA, \CL) \to (\CB, \CM)$ is a pair $(\hat{F}, F)$, where $\hat{F} \colon \CA \to \CB$ is a lax-monoidal $k$-linear functor and $F \colon \CL \to \CM$ is an $\hat{F}$-lax $k$-linear functor.

    \item A 2-morphism $(\hat{F}, F) \Rightarrow (\hat{G}, G)$ is a pair $(\hat{\xi},\xi)$, where $\hat{\xi} \colon \hat{F} \Rightarrow \hat{G}$ is a lax-monoidal natural transformation and $\xi \colon F\Rightarrow G$ is a $\hat{\xi}$-lax natural transformation. 
     
\eit

\begin{thm} \label{thm:fsLMod=lax_fsECat}
The canonical construction defines a 2-equivalence from $\fslmod$ to $\ec[\Lax]{\fsecat}$. Moreover, this 2-equivalence is locally isomorphic. 
\end{thm}

\begin{proof}
    Let $\ec[\CA]{\CL} \in \ec[\Lax]{\fsecat}$. Since $k$-linear functor between two finite semisimple categories is exact, the $k$-linear functor $\ec[\CA]{\CL}(x, -): \CL \to \CA$ admits a left adjoint. By \cite[Theorem\ 1.4]{MPP18} (see also \cite{Lin81}), an enriched category $\ec[\CA]{\CL}$ is equivalent to an enriched category obtained by the canonical construction of a strongly unital oplax module if and only if every functor $\ec[\CA]{\CL}(x, -): \CL \to \CA$ admits a left adjoint. Therefore, the 2-functor induced by canonical construction is essentially surjective on objects. By Proposition \ref{prop:1_morphism_des}, Proposition \ref{prop:2_morphism_des}, it is also fully faithful on both 1-morphisms and 2-morphisms. Then the Whitehead theorem for 2-categories (see Theorem 7.5.8 in \cite{JY21}) implies that this 2-functor is a 2-equivalence. 
\end{proof}


\begin{rem} 
We use $\fsecat$ to denote the sub-2-category of $\ec[\Lax]{\fsecat}$ consisting of those $k$-linear enriched functors $F$ such that $\hat{F}$ is monoidal. We will show in the second work in our series that $\fsecat$ has a monoidal structure with the tensor product given by an analogue  of Deligne's tensor product.  
\end{rem}

%
%


\section{Enriched monoidal categories} \label{sec:enriched_monoidal_categories}

An enriched monoidal category can be defined as an algebra object in either $\ecat$ or $\ec[\Lax]{\ecat}$. In this section, we restrict ourselves to the $\ecat$ case only. Namely, the background changing functor is always monoidal in this section. For a study of the $\ec[\Lax]{\ecat}$ case, see \cite{JS93,For04,BM12}.

\subsection{Definitions and examples}

An enriched monoidal category should be an algebra (a pseudomonoid in \cite{DS97}) in $\ecat$. 
An algebra in $\ecat$ is a collection $(\ec[\CA]{\CL},\ec{\otimes},\ec{\one},\ec{\alpha},\ec{\lambda},\ec{\rho})$ such that (recall Remark\,\ref{rem:background_underlying_2-functors})
\bnu
\item the background category $(\CA,\hat \otimes,\hat \one,\hat \alpha,\hat \lambda,\hat \rho)$ is an algebra in the symmetric monoidal 2-category $\Algn[1](\cat)$ of monoidal categories;
\item the underlying category $(\CL,\otimes,\one,\alpha,\lambda,\rho)$ is an algebra in the symmetric monoidal 2-category $\cat$ of categories, i.e., a monoidal category.
\enu
However, in this work, we propose a slightly different definition of an enriched monoidal category for simplicity and convenience (see Remark\,\ref{rem:equivalence_of_2_definitions}). It is equivalent to \cite[Definition\ 2.3]{KZ18a} (see also \cite{MP17} for a strict version). Note that the notion of an algebra in $\Algn[1](\cat)$ is equivalent to that of a braided monoidal category. 

\begin{defn} \label{def:emc}
Let $\CA$ be a braided monoidal category. An enriched monoidal category consists of the following data:
\bit
\item an enriched category $\ec[\CA]{\CL}$;

\item a tensor product enriched functor $\ec{\otimes} : \ec[\CA]{\CL} \times \ec[\CA]{\CL} \to \ec[\CA]{\CL}$ in $\ecat$ such that $\hotimes=\otimes_\CA$, where the monoidal structure on $\otimes_\CA$ is given by Convention \ref{conv:monoidal_stru_tensor};

\item a distinguished object $\one_\CL \in \ec[\CA]{\CL}$ called the unit object, or equivalently, an enriched functor $\ec{\one_\CL} : * \to \ec[\CA]{\CL}$ in $\ecat$ (see Example \ref{conv:morphism_to_x});

\item an \emph{associator}: an enriched natural isomorphism $\ec{\alpha} : \ec{\otimes} \circ (\ec{\otimes} \times 1) \Rightarrow \ec{\otimes} \circ (1 \times \ec{\otimes})$ with $\hat \alpha$ given by the associator of the monoidal category $\CA$ (i.e., $\hat{\alpha}=\alpha_\CA$);

\item two \emph{unitors}: two enriched natural isomorphisms $\ec{\lambda} : \ec{\otimes} \circ (\ec{\one_\CL} \times 1) \Rightarrow 1_{\ec[\CA]{\CL}}$ and $\ec{\rho} : \ec{\otimes} \circ (1 \times \ec{\one_\CL}) \Rightarrow 1_{\ec[\CA]{\CL}}$ such that $\hat \lambda$ and $\hat \rho$ are given by the left and right unitors of the monoidal category $\CA$, respectively (i.e., $\hat{\lambda}=\lambda_\CA,\hat{\rho}=\rho_\CA$);
\eit
such that $(\CL,\otimes,\one_\CL,\alpha,\lambda,\rho)$ is a monoidal category, which is called the \textit{underlying monoidal category} of $\ec[\CA]{\CL} = (\ec[\CA]{\CL},\ec{\otimes},\ec{\one_\CL},\ec{\alpha},\ec{\lambda},\ec{\rho})$. The enriched monoidal category $\ec[\CA]{\CL}$ is called \emph{strict} if the underlying monoidal category $\CL$ is strict.
\end{defn}

\begin{conv}\label{conv:monoidal_stru_tensor}
    Let $\CA$ a braided monoidal category with the braiding $c_{a,b}: a \otimes b \to b \otimes a$. In the following, the monoidal structure on the tensor product functor $(a, b) \mapsto a \otimes b: \CA \times \CA \to \CA$ is always understood as the lax-monoidal structure induced by the braiding:
\small \[
\otimes(a_1,b_1) \otimes \otimes(a_2,b_2) = a_1 \otimes b_1 \otimes a_2 \otimes b_2 \xrightarrow{1 \otimes c_{b_1, a_2} \otimes 1} a_1 \otimes a_2 \otimes b_1 \otimes b_2 = \otimes((a_1,b_1) \otimes (a_2,b_2)), 
\] \normalsize
or equivalently, as the oplax-monoidal structure of $\otimes$ is defined by the anti-braiding.  
\end{conv}

\begin{rem} \label{rem:equivalence_of_2_definitions}
An enriched monoidal category can be canonically identified with an $E_1$-algebra in $\ecat$. In general, however, an $E_1$-algebra is not an enriched monoidal category but only isomorphic to an enriched monoidal category as $E_1$-algebras. Indeed, let $(\ec[\CA]{\CL},\ec{\otimes},\ec{\one},\ec{\alpha},\ec{\lambda},\ec{\rho})$ be an $E_1$-algebra in $\ecat$. We use $\one'_\CA$ and $\one_\CL$ to denote $\hat{\one}(\ast)$ and $\one(*)$, respectively. Unwinding the definition of $E_1$-algebra in $\ecat$, we have $(\CA, \hotimes, \one'_\CA, \hat{\alpha}, \hat{\lambda}, \hat{\rho})$ is a monoidal category. And the two monoidal category $(\CA, \otimes, \one_\CA)$ and $(\CA, \hotimes, \one'_\CA)$ together with the isomorphisms, i.e., the monoidal structures of $\hat{\one}$ and $\hotimes$,
    \begin{align*}
        (a_1 \hotimes b_1) \otimes (a_2 \hotimes b_2) \simeq (a_1 \otimes a_2) \hotimes (b_1 \otimes b_2), \quad \one_{\CA} \simeq \one'_\CA 
    \end{align*}
is a strong 2-monoidal category (see \cite{AM10} for the definition of a 2-monoidal category or a duoidal category). By the coherence result for 2-monoidal categories, the identity functor of $\CA$ together with the isomorphism $\one_\CA \simeq \one'_\CA$ and the natural isomorphisms
\begin{align*}
        a \hotimes b \xrightarrow{\sim} (a \otimes \one_\CA) \hotimes (\one_\CA \otimes b) \xrightarrow{\sim} (a \otimes \one'_\CA) \hotimes (\one'_\CA \otimes b) \xrightarrow{\sim} (a \hotimes \one'_\CA) \otimes (\one'_\CA \otimes b) \xrightarrow{\sim} a \otimes b
    \end{align*}
is an isomorphism from $(\CA, \otimes, \one_\CA)$ to $(\CA, \hotimes, \one'_\CA)$, and $(\CA, \otimes, \one_\CA)$ is braided with the braiding given by
\begin{multline*}
a \otimes b \xrightarrow{\sim} (\one'_\CA \hotimes a) \otimes (b \hotimes \one'_\CA) \xrightarrow{\sim} (\one'_\CA \otimes b) \hotimes (a \otimes \one'_\CA) \\
\xrightarrow{\sim} (b \otimes \one'_\CA) \hotimes (\one'_\CA \otimes a) \xrightarrow{\sim} (b \hotimes \one'_\CA) \otimes (\one'_\CA \hotimes a) \xrightarrow{\sim} b \otimes a 
\end{multline*}
    (see \cite[Chapter 6]{AM10} for more details). And $\ec[\CA]{\CL}$ can be promoted to an enriched monoidal category such that the underlying monoidal category is $(\CL, \otimes, \one_\CL, \alpha, \lambda, \rho)$, and the tensor product enriched functor $\ec[\CA]{\CL} \times \ec[\CA]{\CL} \to \ec[\CA]{\CL}$ is defined by $\otimes_\CA: \CA \times \CA \to \CA$ (see Convention \ref{conv:monoidal_stru_tensor}), the map $\otimes: \ob(\CL) \times \ob(\CL) \to \ob(\CL)$, and morphisms
\begin{align*}
    \ec[\CA]{\CL}(x_1, y_1) \otimes \ec[\CA]{\CL}(x_2, y_2) \xrightarrow{\sim} \ec[\CA]{\CL}(x_1, y_1) \hotimes \ec[\CA]{\CL}(x_2, y_2) \to \ec[\CA]{\CL}(x_1 \otimes x_2, y_1 \otimes x_2), 
\end{align*}
where the second arrow is given by the morphism in the definition of the enriched functor $\ec{\otimes}$. Moreover, as algebras in $\ecat$, the enriched monoidal category $\ec[\CA]{\CL}$ and $(\ec[\CA]{\CL},\ec{\otimes},\ec{\one},\ec{\alpha},\ec{\lambda},\ec{\rho})$ are isomorphic. More explicitly, the identity enriched functor on $\ec[\CA]{\CL}$ can be promoted to an algebra isomorphism between these two algebras. 

Based on the above discussion, we require $\hotimes=\otimes_\CA$ and $\hat{\alpha}=\alpha_\CA,\hat{\lambda}=\lambda_\CA,\hat{\rho}=\rho_\CA$ in Definition\,\ref{def:emc}. It turns out that these additional requirements also endow Definition\,\ref{def:emc} with a new interpretation as an algebra in the monoidal category $\ec[\CA]{\ecat}$ with the tensor product $\dtimes$ (recall Example\,\ref{expl:pushforward}). This new interpretation is how the notion was defined in \cite[Definition 2.3]{KZ18a}. 
\end{rem}


\begin{conv}
We use $\overline{\CA}$ to denote the same monoidal category $\CA$ equipped with the anti-braiding. By Convention \ref{conv:monoidal_stru_tensor}, the lax/oplax-monoidal structure of the functor $\otimes: \overline{\CA} \times \overline{\CA} \to \overline{\CA}$ is induced by the anti-braiding/braiding of $\CA$. 
\end{conv}

\begin{expl}
Let $\ec[\CA]{\CL}$ be an enriched monoidal category. We use $\CL^\rev$ to denote the monoidal category obtained from $\CL$ by reversing the tensor product. One can canonically construct an $\overline{\CA}$-enriched monoidal category $\ec[\overline{\CA}]{\CL^\rev}$, whose underlying monoidal category is $\CL^\rev$, as follows. As an enriched category $\ec[\overline{\CA}]{\CL^\rev} = \ec[\CA]{\CL}$. The tensor product enriched functor $\ec{\otimes^\rev}: \ec[\overline{\CA}]{\CL^\rev} \times \ec[\overline{\CA}]{\CL^\rev} \to \ec[\overline{\CA}]{\CL^\rev}$ is defined by the tensor product functor $\overline{\CA} \times \overline{\CA} \to \overline{\CA}$ and the family of morphisms
\[
\ec[\CA]{\CL}(x_1, x_2) \otimes \ec[\CA]{\CL}(y_1, y_2) \xrightarrow{c_{\ec[\CA]{\CL}(x_1, x_2), \ec[\CA]{\CL}(y_1, y_2)}} \ec[\CA]{\CL}(y_1, y_2) \otimes \ec[\CA]{\CL}(x_1, x_2) \xrightarrow{\ec{\otimes}_{(y_1,x_1),(y_2,x_2)}} \ec[\CA]{\CL}(x_1 \otimes^\rev y_1 , x_2 \otimes^\rev y_2).
\]
The associator of $\ec[\overline{\CA}]{\CL^\rev}$ is defined by the associator of $\overline{\CA}$ and the inverse of the associator of $\CL$. The left/right unitor of $\ec[\overline{\CA}]{\CL^\rev}$ is defined by the left/right unitor of $\overline{\CA}$ and the right/left unitor of $\CL$. We refer to $\ec[\overline{\CA}]{\CL^\rev}$ as the reversed category of $\ec[\CA]{\CL}$, i.e., 
$(\ec[\CA]{\CL})^\rev \coloneqq \ec[\overline{\CA}]{\CL^\rev}$. 
\end{expl}

\begin{rem}
In the definition of the enriched functor $\ec{\otimes^\rev}$, we have used the braiding instead of the anti-braiding because $\ec{\otimes^\rev}$ is not well-defined if we use the anti-braiding. 
\end{rem}

\begin{defn} \label{def:enriched_monoidal_functor}
Let $\ec[\CA]{\CL}$ and $\ec[\CB]{\CM}$ be enriched monoidal categories. An enriched monoidal functor $\ec{F} : \ec[\CA]{\CL} \to \ec[\CB]{\CM}$ consists of the following data:
\bit
\item an enriched functor $\ec{F} : \ec[\CA]{\CL} \to \ec[\CB]{\CM}$ in $\ecat$;
\item an enriched natural isomorphism 
\[
\ec{F^2} : \ec{\otimes} \circ (\ec{F} \times \ec{F}) \Longrightarrow \ec{F} \circ \ec{\otimes}
\] 
with the background changing natural transformation $\hat F^2$ given by the monoidal structure of $\hat F$;

\item an enriched natural isomorphism 
\[
\ec{F^0} : \ec{\one_\CM} \Longrightarrow \ec{F} \circ \ec{\one_\CL}
\] 
with the background changing natural transformation $\hat F^0$ given by the monoidal structure of $\hat F$; 
\eit
satisfying the following conditions:
\bnu
\item the background changing functor $\hat F = (\hat F,\hat F^2,\hat F^0) : \CA \to \CB$ is a braided monoidal functor; 
\item the underlying functor $F : \CL \to \CM$, together with the underlying natural transformations $F^2$ and $F^0$, defines a monoidal functor.
\enu
\end{defn}

\begin{defn} \label{def:enriched_mon_nat_tran}
    Let $\ec[\CA]{\CL},\ec[\CB]{\CM}$ be enriched monoidal categories and $\ec{F},\ec{G} : \ec[\CA]{\CL} \to \ec[\CB]{\CM}$ enriched monoidal functors. An \textit{enriched monoidal natural transformation} is an enriched natural transformation $\ec{\xi}: \ec{F} \Rightarrow \ec{G}$ such that the underlying natural transformation $\xi: F \Rightarrow G$ is a monoidal natural transformation. 
\end{defn}


\begin{rem}
    Recall that braided monoidal categories are algebras in $\Algn[1](\cat)$, and the algebra homomorphisms between braided monoidal categories in $\Algn[1](\cat)$ are braided monoidal functors. Then it is routine to check that the algebra homomorphisms between enriched monoidal categories in $\ecat$ are enriched monoidal functors and $2$-morphisms between algebra homomorphisms in $\ecat$ are enriched monoidal natural transformations (see Remark \ref{rem:background_underlying_2-functors}). By Remark \ref{rem:equivalence_of_2_definitions} and the Whitehead theorem for 2-categories, the symmetric monoidal 2-category of enriched monoidal categories, enriched monoidal functors and enriched monoidal natural transformations is a sub-2-category of the 2-category $\Algn[1](\ecat)$ of algebras in $\ecat$, and the inclusion is a 2-equivalence.
\end{rem}

\begin{expl} \label{expl:comm_alg_B_en_monoidal_cat}
Let $(A, m_A, \iota_A)$ be a commutative algebra in a braided monoidal category $\CA$. The $\CA$-enriched category $\ast^{A}$ introduced in Example \ref{expl:enrich_cat_with_one_obj} is a strict enriched monoidal category with the tensor product functor induced by $m_A$. Conversely, every strict $\CA$-enriched monoidal category with one object arises in this way (see also \cite[Example\ 3.5]{KZ18a}).

Let $(B, m_B, \iota_B)$ be a commutative algebra in a braided monoidal category $\CB$ and $\ec{F}: \ast^A \to \ast^B$ be an enriched monoidal functor. The enriched natural isomorphisms $\ec{F^2},\ec{F^0}$ are determined by the underlying natural isomorphisms $F^2,F^0$, and any morphism $(F^2)_*,(F^0)_*$ together with $\hat F^2,\hat F^0$ form an enriched natural transformation respectively since $B$ is commutative. The condition that $(F,F^2,F^0)$ is a monoidal functor is equivalent to say that $(F^2)_*$ is the inverse of $(F^0)_*$ in the underlying category of $\ast^B$. Hence an enriched monoidal functor $\ec{F} : \ast^A \to \ast^B$ is defined by a braided monoidal functor $\hat F : \CA \to \CB$, an algebra homomorphism $\ec{F}_{*,*} : \hat{F}(A) \to B$ (see Example \ref{expl:enrich_cat_with_one_obj}) and an isomorphism $(F^0)_*$ in the underlying category of $\ast^B$.

Let $\ec{G}: \ast^A \to \ast^B$ be another enriched monoidal functor defined by a braided monoidal functor $\hat{G}$, an algebra homomorphism $\ec{G}_{*,*} : \hat{G}(A) \to B$ and an isomorphism $(G^0)_*$. Then an enriched natural transformation $\ec{\xi}: \ec{F} \Rightarrow \ec{G}$ is an enriched monoidal natural transformation if and only if $(G^0)_* = \xi_* \circ (F^0)_*$ in the underlying category of $\ast^B$, which further implies that $\xi_*$ is invertible. If $\xi_*$ is invertible, the commutativity of $B$ implies that the diagram \ref{diag:enrich_cat_with_one_obj_1} is equivalent to
\[
\bigl( \hat{F}(A) \xrightarrow{\ec{F}_{*,*}} B \bigr) = \bigl( \hat{F}(A) \xrightarrow{\hat{\xi}_A} \hat{G}(A) \xrightarrow{\ec{G}_{*,*}} B \bigr) .
\]
As a conclusion, an enriched monoidal natural transformation $\ec{\xi} : \ec{F} \Rightarrow \ec{G}$ is defined by a monoidal natural transformation $\hat \xi$ such that $\ec{F}_{*,*} = \ec{G}_{*,*} \circ \hat \xi_A$ and the underlying natural transformation $\xi_* = (G^0)_* \circ (F^0)_*^{-1}$. 
\end{expl}

\subsection{Canonical construction}

Let $\CA$ be a braided monoidal category viewed as an algebra in $\Alg_{E_1}^{\mathrm{oplax}}(\cat)$ by Convention \ref{conv:monoidal_stru_tensor}.

\begin{defn} \label{def:oplax_monoidal_module}
A \emph{monoidal left $\CA$-\oplax module} is a left $\CA$-\oplax module in $\Alg_{E_1}^{\mathrm{oplax}}(\cat)$. A \emph{monoidal left $\CA$-module} is a left $\CA$-module in $\Algn[1](\cat)$.
\end{defn}

\begin{rem}
More explicitly, a monoidal left $\CA$-\oplax module $\CL$ consists of the following data:
\bit
    \item a monoidal category $\CL$;
    \item a left $\CA$-action given by an oplax-monoidal functor $\odot: \CA \times \CL \to \CL$ (i.e., the functor $\odot$ is equipped with a natural transformation 
\be \label{eq:abxy}
(a \otimes b) \odot (x \otimes y) \to (a \odot x) \otimes (b \odot y)\quad\quad \forall a,b\in\CA,x,y\in\CL
\ee 
and a morphism $\one_\CA \odot \one_\CL \to \one_\CL$ rendering the following diagrams commutative); 
\scriptsize
\begin{gather}
\begin{array}{c}
\xymatrix @R=0.15in @C=0.5in{
                [(a \otimes b) \otimes c] \odot [(x \otimes y)\otimes z] \ar[d] \ar[r] 
                & [a \otimes (b \otimes c)] \odot [x \otimes (y\otimes z)] \ar[d]\\
                [(a \otimes b) \odot (x \otimes y)] \otimes (c \odot z) \ar[d] 
                & (a \odot x) \otimes [(b \otimes c) \odot (y \otimes z)] \ar[d]\\ 
                [(a \odot x) \otimes (b \odot y)] \otimes (c \odot z) \ar[r] 
                & (a \odot x) \otimes [(b \odot y) \otimes (c \odot z)]
            }
\end{array} \label{diag:oplax_monoidal_module_1} \\
\begin{array}{c}
\xymatrix @R=0.15in @C=0.1in{
                a \odot x & (\one_\CA \otimes a) \odot (\one_\CL \otimes x) \ar[l] \ar[d] \\
                \one_\CL \otimes(a \odot x) \ar[u] & (\one_\CA \odot \one_\CL) \otimes(a \odot x) \ar[l]
            }
\end{array} \;
\begin{array}{c}
\xymatrix @R=0.15in @C=0.1in{
                a \odot x & (a \otimes \one_\CA) \odot (x\otimes \one_\CL) \ar[l] \ar[d] \\
                 (a \odot x) \otimes \one_\CL \ar[u] &  (a \odot x) \otimes (\one_\CA \odot \one_\CL) \ar[l]
            }
\end{array} \label{diag:oplax_monoidal_module_2}
\end{gather} \normalsize
        
    \item two oplax-monoidal natural transformation: the associator $(a \otimes b) \odot x \to a \odot (b \odot x)$ and the unitor $\one_\CA \odot x \to x$ (i.e., the following diagrams are commutative\footnote{The diagram \eqref{diag:oplax_monoidal_module_4} and the first diagram in \eqref{diag:1A-1L} are the defining properties of the associator as an oplax-monoidal natural transformation. The other two are those of the unitors.}, 
\scriptsize
\begin{gather}
\begin{array}{c}
\xymatrix @R=0.15in @C=0.2in{
                [(a_1 \otimes a_2) \otimes (b_1 \otimes b_2)] \odot (x \otimes y) \ar[r] \ar[d]
                & (a_1 \otimes a_2) \odot [(b_1 \otimes b_2) \odot (x \otimes y)] \ar[d]\\ 
                [(a_1 \otimes b_1) \otimes (a_2 \otimes b_2)] \odot (x \otimes y) \ar[d]
                & (a_1 \otimes a_2) \odot [(b_1 \odot x) \odot (b_2 \odot y)] \ar[d]\\ 
                [(a_1 \otimes b_1) \odot x] \otimes [(a_2 \otimes b_2) \odot y] \ar[r]
                & [a_1 \odot (b_1 \odot x)] \otimes [a_2 \odot (b_2 \odot y)]
}
\end{array} \label{diag:oplax_monoidal_module_4} \\
\begin{array}{c}
\xymatrix @R=0.3in @C=0.2in{
                \one_\CA \odot (x \otimes y) \ar[r] \ar[d] & (\one_\CA \otimes \one_\CA) \odot (x \otimes y) \ar[d]\\
            x \otimes y & (\one_\CA \odot x) \otimes (\one_\CA \odot y) \ar[l]
            }
\end{array} \label{diag:oplax_monoidal_module_4b} \\
\begin{array}{c}
\xymatrix @R=0.3in @C=0.2in{
        (\one_\CA \otimes \one_\CA) \odot \one_\CL \ar[rr] \ar[d] & & \one_\CA \odot (\one_\CA \odot \one_\CL) \ar[d]^{\text{oplax}} \\
        \one_\CA \odot \one_\CL \ar[r]^-{\text{oplax}} & \one_\CL & \one_\CA \odot \one_\CL \ar[l]_-{\text{oplax}}
        }
\end{array} \quad
\begin{array}{c}
\xymatrix @R=0.3in @C=0.5in{
        \one_\CA \odot \one_\CL \ar[r]_-{\text{unitor}} \ar[dr]_{\text{oplax}} & \one_\CL \ar@{=}[d] \\
         & \one_\CL
        }
\end{array} \label{diag:1A-1L}
\end{gather} \normalsize
        where the the arrow $[(a_1 \otimes a_2) \otimes (b_1 \otimes b_2)] \odot (x \otimes y) \to [(a_1 \otimes b_1) \otimes (a_2 \otimes b_2)] \odot (x \otimes y)$ is defined by the anti-braiding $c^{-1}_{b_1,a_2}: a_2 \otimes b_1 \to b_1 \otimes a_2$ (recall Convention \ref{conv:monoidal_stru_tensor}));
\eit
such that $\CL$, together with the left $\CA$-action functor $\odot$, the associator and the unitor defined above, is a left $\CA$-\oplax module.

If, in addition, the left $\CA$-action functor $\odot$ is a monoidal functor, and the associator and the unitors are natural isomorphisms, then $\CL$ is a monoidal left $\CA$-module. 
\end{rem}

\begin{rem} \label{rem:oplax_monoidal_module_defn_diagram}
If, in addition, $\CL$ is a strongly unital left $\CA$-\oplax module, then the commutativity of the second diagram in \eqref{diag:1A-1L} follows from those of diagram \eqref{diag:oplax_monoidal_module_2} and the diagram \eqref{diag:oplax_monoidal_module_4b}.
Also, the commutativity of the first diagram in \eqref{diag:1A-1L} is a consequence of that of the second diagram in \eqref{diag:1A-1L} and the fact that $\CL$ is an $\CA$-\oplax module.
\end{rem}

\begin{rem} \label{rem:oplax_monoidal_module}
By Definition \ref{def:oplax_monoidal_module}, a monoidal structure on a left $\CA$-\oplax module $\CL$ consists of a monoidal structure on $\CL$ and a natural transformation \eqref{eq:abxy} satisfying proper axioms. All the rest data can be included in the defining data of a left $\CA$-\oplax module structure on $\CL$. 
\end{rem}

In the following, we use $\forget$ to denote the forgetful functor $\FZ_1(\CL) \to \CL$ for any monoidal category $\CL$. If $\varphi: \CA \to \FZ_1(\CL)$ is a braided (oplax) monoidal functor, then $\cf{\varphi} \coloneqq \forget \circ \varphi$ is an \textit{(oplax) central functor} \cite{Bez04}.

\begin{expl} \label{expl:oplax_monoidal_module_functor}
Let $\CA$ be a braided monoidal category and $\CL$ a monoidal category. If $\varphi: \CA \to \FZ_1(\CL)$ is a braided oplax-monoidal functor, then $\CL$ is a monoidal left $\CA$-\oplax module with the module action $\cf{\varphi}(-) \otimes - : \CA \times \CL \to \CL$. The oplax-monoidal structure $\varphi(a \otimes b) \otimes (x \otimes y) \to (\varphi(a) \otimes x) \otimes (\varphi(b) \otimes y)$ is given by
\small
    \begin{multline*}
        \cf{\varphi}(a \otimes b) \otimes (x \otimes y) \to (\cf{\varphi}(a) \otimes \cf{\varphi}(b)) \otimes (x \otimes y) \to (\cf{\varphi}(a) \otimes (\cf{\varphi}(b) \otimes x)) \otimes y \\
            \to (\cf{\varphi}(a) \otimes (x \otimes \cf{\varphi}(b))) \otimes y \to (\cf{\varphi}(a) \otimes x) \otimes (\cf{\varphi}(b) \otimes y),
    \end{multline*} \normalsize
where the third arrow is defined by the half-braiding of $\varphi(b)$.
\end{expl}

\begin{rem} \label{rem:left_unital_action_in_Mod}
There is another definition of a monoidal module \cite{Lur17,AFT16,KZ18,HPT16}. Given a monoidal left $\CA$-module $\CL$ as defined in Definition \ref{def:oplax_monoidal_module}, $\varphi \coloneqq (- \odot \one_\CL)$ defines a braided monoidal functor $\varphi : \CA \to \FZ_1(\CL)$ (recall Example \ref{expl:centers_in_cat}). Two monoidal left $\CA$-modules $(\CL, \odot)$ and $(\CL, \cf{\varphi}(-) \otimes -)$ are isomorphic. Hence, a monoidal left $\CA$-module $\CL$ can be equivalently defined by a monoidal category $\CL$ equipped with a braided monoidal functor $\CA \to \FZ_1(\CL)$. 

When $\CL$ is a monoidal left $\CA$-\oplax module, the morphism \eqref{eq:abxy} is not necessarily invertible, thus does not induce a half-braiding. 
\end{rem}

\begin{prop} \label{prop:monoidal_en_monoidal_oplax_module}
Given a pair $(\CA,\CL) \in \lmod$, let $\bc[\CA]{\CL}$ be the enriched category obtained from the pair $(\CA,\CL)$ via the canonical construction. 
There is a one-to-one correspondence between the monoidal structures on $\bc[\CA]{\CL}$ and the monoidal structures on the  left $\overline{\CA}$-\oplax module $\CL$ as shown by the following two mutually inverse constructions: 
\begin{itemize}

\item[$\Rightarrow$] Given a monoidal structure on $\bc[\CA]{\CL}$, i.e., a sextuple $(\bc[\CA]{\CL},\ec{\otimes},\ec{\one_\CL},\ec{\alpha},\ec{\lambda},\ec{\rho})$, then the sextuple $(\CL,\otimes,\one_\CL,\alpha,\lambda,\rho)$ automatically defines a monoidal structure on $\CL$ by definition. A monoidal structure on the left $\overline{\CA}$-\oplax module $\CL$ is determined (recall Remark \ref{rem:oplax_monoidal_module}) if we further define the morphisms \eqref{eq:abxy} to be 
the one induced from the composed morphism:
\[
(a \otimes  b) \xrightarrow{\coev_x \otimes \coev_y} [x,a \odot x] \otimes [y,b \odot y] \xrightarrow{\ec{\otimes}} [x \otimes y,(a \odot x) \otimes (b \odot y)]. 
\]

\item[$\Leftarrow$] Given a monoidal structure on the left $\overline{\CA}$-\oplax module $\CL$. A monoidal structure on $\bc[\CA]{\CL}$ is determined by defining the morphism $[x_1, x_2] \otimes [y_1, y_2] \to [x_1 \otimes y_1, x_2 \otimes y_2]$ to be the one induced by the following composed morphism
\small \[
([x_1, x_2] \otimes [y_1, y_2]) \odot (x_1 \otimes y_1) \to ([x_1, x_2] \odot x_1) \otimes ([y_1, y_2] \odot y_1) \xrightarrow{\ev_{x_1} \otimes \ev_{y_1}} x_2 \otimes y_2 ,
\] \normalsize
where the first arrow is given by the oplax-monoidal structure of $\odot: \CA \times \CL \to \CL$.
\end{itemize}
\end{prop}

\pf
It is routine to check that the diagrams \eqref{diag:oplax_monoidal_module_4} and \eqref{diag:oplax_monoidal_module_4b} for a monoidal left $\overline{\CA}$-\oplax module $\CL$ commute if and only if $\otimes_\CL$ is a $\otimes_\CA$-lax functor with the $\otimes_\CA$-lax structure given by the oplax-monoidal structure of $\odot$, and the diagrams in \eqref{diag:oplax_monoidal_module_1}, \eqref{diag:oplax_monoidal_module_2} 
commute if and only if $(\hat \alpha,\alpha),(\hat \lambda,\lambda),(\hat \rho,\rho)$ is a 2-morphism in $\lmod$, respectively. Then the assertion follows from Theorem \ref{thm:2-functor_from_A_modules_to_cat_A} (see also Proposition \ref{prop:1_morphism_des}).
\end{proof}

\begin{expl}[\cite{MP17}]\label{expl:canonical_construction_enriched_monoidal_cat}
    Let $\varphi: \overline{\CA} \to \FZ_1(\CL)$ be a strongly unital braided oplax-monoidal functor. Then $\CL$ equipped with the module action $\cf{\varphi}(-)\otimes -: \overline{\CA} \times \CL \to \CL$ is a strongly unital monoidal left $\overline{\CA}$-\oplax module (see Example \ref{expl:oplax_monoidal_module_functor}). If $\CL$ is also enriched in $\CA$, then $\bc[\CA]{\CL}$ is an $\CA$-enriched monoidal category.  
\end{expl}

\begin{expl}
Let $\CL$ be a left $\CA$-module. We use $\Fun_{\CA}(\CL, \CL)$ to denote the category of left $\CA$-module functors and left $\CA$-module natural transformations. There is an obvious braided monoidal functor $\overline{\FZ_1(\CA)} \to \FZ_1(\Fun_{\CA}(\CL, \CL))$ defined by $a \mapsto \forget(a) \odot -$. The left $\CA$-module functor $\forget(a) \odot -$ is equipped with a half-braiding
\[
\gamma_{F,\forget(a)} : F(\forget(a) \odot -) \Rightarrow \forget(a) \odot F(-), \quad\quad \forall F \in \fun_{\CA}(\CL,\CL),  
\]
defined by the left $\CA$-module functor structure on $F$ \cite{Ost03}. It follows that $\Fun_{\CA}(\CL, \CL)$ is a monoidal left $\overline{\FZ_1(\CA)}$-module.
\end{expl}

\begin{defn}\label{def:monoidal_lax_module_functor}
Let $\CA,\CB$ be two braided monoidal categories. Let $\CL$ be a monoidal left $\CA$-\oplax module and $\CM$ a monoidal left $\CB$-\oplax module. Given a braided monoidal functor $\hat{F}: \CA \to \CB$, a \textit{monoidal $\hat{F}$-lax functor} $F: \CL \to \CM$ is a monoidal functor and also an $\hat F$-lax functor such that the $\hat{F}$-lax structure $\hat{F}(a) \odot F(x) \to F(a \odot x)$ is an oplax-monoidal natural transformation (i.e., the following diagrams commute).
\begin{align}
\begin{array}{c}
\xymatrix @R=0.15in @C=0.8in{
            \hat{F}(a \otimes b) \odot F(x \otimes y) \ar[d] \ar[r] & F((a \otimes b) \odot (x\otimes y)) \ar[d] \\
            (\hat{F}(a) \otimes \hat{F}(b)) \odot (F(x) \otimes F(y)) \ar[d] & F((a \odot x) \otimes (b \odot y)) \ar[d]\\
         (\hat{F}(a) \odot F(x)) \otimes (\hat{F}(b) \odot F(y)) \ar[r] & F(a \odot x) \otimes F(b \odot y) 
}
\end{array} \label{diag:monoidal_lax_module_functor_1} \\
\begin{array}{c}
\xymatrix @C=0.7in{
\hat F(\one) \odot F(\one) \ar[rr] \ar[d] & & F(\one \odot \one) \ar[d] \\
\one \odot \one \ar[r] & \one & F(\one) \ar[l]
}
\end{array} \quad \quad \label{diag:monoidal_lax_module_functor_2}
\end{align} 
\end{defn}


\begin{rem}
The diagram \eqref{diag:monoidal_lax_module_functor_2} is always commutative because $F$ is an $\hat F$-lax functor.
\end{rem}

\begin{rem} \label{rem:mon_lax_functor_des}
Let $\varphi_i: \CA_i \to \FZ_1(\CL_i)$ be a braided oplax-monoidal functor, $i=1,2$. Then $\CL_i$ is a monoidal left $\CA_i$-\oplax module. Let $\hat{F}: \CA_1 \to \CA_2$ be a braided monoidal functor and $F: \CL_1 \to \CL_2$ be a monoidal functor. Suppose $\theta_a: \cf{\varphi}_2(\hat{F}(a)) \to F(\cf{\varphi}_1(a))$ is an oplax-monoidal natural transformation rendering the following diagram commutative,
\be \label{diag:mon_lax_functor_des_1}
\begin{array}{c}
\xymatrix @R=0.2in @C=0.4in{
            \cf{\varphi}_2(\hat{F}(a)) \otimes F(x) \ar[r]^-{\theta_a \otimes 1} \ar[d] & F(\cf{\varphi}_1(a)) \otimes F(x) \ar[r]^{\sim} & F(\cf{\varphi}_1(a) \otimes x) \ar[d] \\
            F(x) \otimes \cf{\varphi}_2(\hat{F}(a)) \ar[r]^-{1 \otimes \theta_a} & F(x) \otimes F(\cf{\varphi}_1(a))  \ar[r]^{\sim} & F(x\otimes \cf{\varphi}_1(a))
        }
\end{array}
\ee
where two vertical arrows are the half-braidings of $\varphi_2(\hat{F}(a))$ and $\varphi_1(a)$, respectively. Then $F$ equipped with the natural transformation 
\[
\beta_{a, x} \coloneqq \bigl( \cf{\varphi}_2(\hat{F}(a)) \otimes F(x) \xrightarrow{\theta_a \otimes 1} F(\cf{\varphi}_1(a)) \otimes F(x) \xrightarrow{\sim} F(\cf{\varphi}_1(a) \otimes x) \bigr) 
\]
is a monoidal $\hat{F}$-lax functor. 

    We claim that every monoidal $\hat{F}$-lax structure on $F$ arises in this way. Indeed, suppose $\beta_{a, x}: \cf{\varphi}_2(\hat{F}(a)) \otimes F(x) \to F(\cf{\varphi}_1(a) \otimes x)$ is a monoidal $\hat{F}$-lax structure on $F$. Then  
\[
\theta_a \coloneqq \bigl(\cf{\varphi}_2(\hat{F}(a)) \xrightarrow{\sim} \cf{\varphi}_2(\hat{F}(a)) \otimes F(\one_{\CL_1}) \xrightarrow{\beta_{a, \one_{\CL_1}}} F(\cf{\varphi}_1(a) \otimes \one_{\CL_1}) \xrightarrow{\sim} F(\cf{\varphi}_1(a)) \bigr)
\]
is an oplax-monoidal natural transformation. Taking  $b = \one_{\CA_1}$ and $x= \one_{\CL_2}$ in the diagram \eqref{diag:monoidal_lax_module_functor_1}, then we have
\[
\beta_{a, y} = \bigl( \cf{\varphi}_2(\hat{F}(a)) \otimes F(y) \xrightarrow{\theta_a \otimes 1} F(\cf{\varphi}_1(a)) \otimes F(y) \xrightarrow{\sim} F(\cf{\varphi}_1(a) \otimes y) \bigr) .
\]
By taking $a = \one_{\CA_1}$ and $y = \one_{\CL_2}$ in the diagram \eqref{diag:monoidal_lax_module_functor_1}, we can show that the oplax-monoidal natural transformation $\theta_a$ renders the diagram \eqref{diag:mon_lax_functor_des_1} commutative. Hence, when $\CA_1=\CA_2$ and $\hat{F}=1_{\CA_1}$, Definition \ref{def:monoidal_lax_module_functor} is precisely \cite[Definition\ 2.6.6]{KZ18}.
\end{rem}

\begin{prop} \label{prop:enriched_monoidal_functor_monoidal_lax_functor}
Let $\CA,\CB$ be braided monoidal categories. Suppose $\CL$ is a strongly unital monoidal left $\overline{\CA}$-\oplax module that is enriched in $\CA$, and $\CM$ is a strongly unital monoidal left $\overline{\CB}$-\oplax module that is enriched in $\CB$. Then the canonical construction gives enriched monoidal categories $\bc[\CA]{\CL}$ and $\bc[\CB]{\CM}$. Suppose $\hat F : \CA \to \CB$ is a braided monoidal functor and $F : \CL \to \CM$ is both a monoidal functor and an $\hat F$-lax functor. Then $F$ is a monoidal $\hat{F}$-lax functor if and only if $\hat F$ and $F$ define an enriched monoidal functor $\ec{F}: \bc[\CA]{\CL} \to \bc[\CB]{\CM}$. 
\end{prop}

\begin{proof}
    By Theorem \ref{thm:2-functor_from_A_modules_to_cat_A}, the $\hat F$-lax functor $F$ defines an enriched functor $\ec{F} : \bc[\CA]{\CL} \to \bc[\CB]{\CM}$. Then we only need to show that $F$ is a monoidal $\hat F$-lax functor if and only if $(\hat F^2,F^2)$ and $(\hat F^0,F^0)$ are 2-morphisms in $\lmod$. It is routine to check that $(\hat{F}^2, F^2)$ is a 2-morphism in $\lmod$ if and only if the diagram \eqref{diag:monoidal_lax_module_functor_1} commutes, and $(\hat F^0,F^0)$ is a 2-morphism in $\lmod$ if and only if the diagram \eqref{diag:monoidal_lax_module_functor_2} commutes.
\end{proof}

\begin{defn} \label{defn:monoidal_lax_natural_transformation}
Let $\CA,\CB$ be two braided monoidal categories. Let $\CL$ be a monoidal left $\CA$-\oplax module and $\CM$ be a monoidal left $\CB$-\oplax module. Suppose $\hat F_i : \CA \to \CB$ is a braided monoidal functor and $F_i: \CL \to \CM$ is a monoidal $\hat{F}_i$-lax functor, $i=1,2$. Given a monoidal natural transformation $\hat \xi : \hat F_1 \Rightarrow \hat F_2$, a \textit{monoidal $\hat \xi$-lax natural transformation} $\xi : F_1 \Rightarrow F_2$ is a $\hat \xi$-lax natural transformation and also a monoidal natural transformation.
\end{defn}

\begin{rem}
    For $i = 1,2$, let $\CL_i$ be a monoidal left $\CA_i$-\oplax module induced by a braided oplax-monoidal functor $\varphi_i: \CA_i \to \FZ_1(\CL_i)$. Assume that $\hat{F}_i: \CA_1 \to \CA_2$ is a braided monoidal functor and $F_i: \CL_1 \to \CL_2$ is a monoidal $\hat{F}_i$-lax functor. Let $\hat{\xi}: \hat{F}_1 \to \hat{F}_2$ and $\xi: F_1 \to F_2$ be two monoidal natural transformations. Then $\xi$ is a monoidal $\hat \xi$-lax natural transformation if and only if the following diagram commutes,
\[
        \xymatrix @R=0.2in @C=0.5in{
            \cf{\varphi}_2(\hat{F}_1(a)) \ar[r]^-{\theta_1} \ar[d]^-{\hat{\xi}} & F_1(\cf{\varphi}_1(a)) \ar[d]^-{\xi}\\
            \cf{\varphi}_2(\hat{F}_2(a)) \ar[r]^-{\theta_2} & F_2(\cf{\varphi}_1(a))
        }
\]
where $\theta_i: \cf{\varphi}_2 \circ \hat{F}_i \Rightarrow F_i \circ \cf{\varphi}_1$ is the oplax-monoidal natural transformation that induces the monoidal $\hat{F}_i$-lax module structure on $F_i$ (see Remark \ref{rem:mon_lax_functor_des}). 
\end{rem}
 

\begin{prop} \label{prop:enriched_monoidal_natural_transf_can}
Let $\CA,\CB$ be braided monoidal categories. Suppose $\CL$ is a strongly unital monoidal left $\overline{\CA}$-\oplax module that is enriched in $\CA$, and $\CM$ is a strongly unital monoidal left $\overline{\CB}$-\oplax module that is enriched in $\CB$. For $i=1,2$, let $\hat F_i : \CA \to \CB$ be a braided monoidal functor and $F_i : \CL \to \CM$ be a monoidal $\hat F_i$-lax functor. Then $(\hat{F}_i,F_i)$ defines an enriched monoidal functor $\ec{F_i}: \bc[\CA]{\CL} \to \bc[\CB]{\CM}$. Suppose $\hat \xi : \hat F_1 \Rightarrow \hat F_2$ is a monoidal natural transformation and $\xi : F_1 \Rightarrow F_2$ is a $\hat \xi$-lax natural transformation. Then $\xi$ is a monoidal $\hat \xi$-lax natural transformation if and only if $(\hat \xi ,\xi)$ defines an enriched monoidal natural transformation.
\end{prop}

\begin{proof}
By Proposition \ref{prop:2_morphism_des}, $\xi$ is a monoidal $\hat \xi$-lax natural transformation if and only if $\xi$ is a monoidal natural transformation, and $(\hat \xi,\xi)$ defines an enriched monoidal natural transformation. Thus $\xi$ is a monoidal $\hat \xi$-lax natural transformation is amount to the fact that $(\hat \xi ,\xi)$ defines an enriched monoidal natural transformation.
\end{proof}

We use $\ecat^{\otimes}$ to denote the 2-category of enriched monoidal categories, enriched monoidal functors and enriched monoidal natural transformations. It is symmetric monoidal with the tensor product given by the Cartesian product $\times$ and the tensor unit given by $\ast$. 

We define the 2-category $\lmod^{\otimes}$ as follows: 
\bit
    \item The objects are pairs $(\CA, \CL)$, where $\CA$ is a braided monoidal category and $\CL$ is a strongly unital monoidal left $\CA$-\oplax module that is enriched in $\CA$.

    \item A 1-morphism $(\CA, \CL) \to (\CB, \CM)$ is a pair $(\hat{F}, F)$, where $\hat{F}: \CA \to \CB$ is a braided monoidal functor and $F:\CL \to \CM$ is a monoidal $\hat{F}$-lax functor. 

    \item A 2-morphism $(\hat{F}, F) \Rightarrow (\hat{G}, G)$ is a pair $(\hat{\xi},\xi)$, where
    $\hat{\xi}: \hat{F} \Rightarrow \hat{G}$ is a monoidal natural transformation and $\xi:F\Rightarrow G$ is a monoidal $\hat{\xi}$-lax natural transformation. 
\eit
The horizontal/vertical composition is induced by the horizontal/vertical composition of functors and natural transformations. The 2-category $\lmod^{\otimes}$ is symmetric monoidal with the tensor product defined by the Cartesian product and the tensor unit given by $(\ast, \ast)$.

\medskip
By Proposition \ref{prop:monoidal_en_monoidal_oplax_module}, Proposition \ref{prop:enriched_monoidal_functor_monoidal_lax_functor} and Proposition \ref{prop:enriched_monoidal_natural_transf_can} we obtain the following result.


\begin{thm}\label{thm:sym_mono_2_functor_monoidal}
The canonical construction defines a symmetric monoidal 2-functor from $\lmod^{\otimes}$ to $\ecat^{\otimes}$. Moreover, this 2-functor is locally isomorphic.
\end{thm}

\begin{proof}
By Theorem \ref{thm:2-functor_from_A_modules_to_cat_A}, the canonical construction defines a symmetric monoidal 2-functor from $\mathbf{LMod}$ to $\ec[\Lax]{\ecat}$. Let $\CL$ be a strongly unital monoidal left $\overline{\CA}$-\oplax module. By Proposition \ref{prop:monoidal_en_monoidal_oplax_module}, $\bc[\CA]{\CL}$ is an enriched monoidal category. Given 1-morphisms $(\hat{F}, F): (\overline{\CA}, \CL) \to (\overline{\CB}, \CM)$ and $(\hat{G}, G): (\overline{\CB}, \CM) \to (\overline{\CC}, \CN)$ in $\lmod^{\otimes}$, the monoidal structures of $\ec{G} \circ \ec{F}$ and $\ec{GF}$ are both defined by the monoidal structures of $\hat{G}\hat{F}$ and $GF$, where $\ec{GF}$ is the image of $(\hat{G}\hat{F},GF)$. Therefore $\ec{G} \circ \ec{F} = \ec{GF}$ as monoidal functors. Then Proposition \ref{prop:enriched_monoidal_functor_monoidal_lax_functor} and Proposition \ref{prop:enriched_monoidal_natural_transf_can} imply that canonical construction defines a locally isomorphic 2-functor from $\lmod^{\otimes}$ to $\ecat^{\otimes}$.  

Note that $\ecat^{\otimes}$ inherits the symmetric monoidal structure of $\ec[\Lax]{\ecat}$. Recall that the background changing functor and the underlying functor of the enriched isomorphism $\ec{\chi}_{\bc[\CA]{\CL}, \bc[\CB]{\CM}}: \bc[\CA]{\CL} \times \bc[\CB]{\CM} \to \bc[(\CA \times \CB)]{(\CL \times \CM)}$ defined in the proof of Theorem \ref{thm:2-functor_from_A_modules_to_cat_A} are both identity functors. It is routine to check that $\ec{\chi}_{\bc[\CA]{\CL}, \bc[\CB]{\CM}}$ is a monoidal functor with trivial monoidal structure. Then the canonical construction defines a symmetric monoidal 2-functor from $\lmod^{\otimes}$ to $\ecat^{\otimes}$ by Theorem \ref{thm:2-functor_from_A_modules_to_cat_A}. 
\end{proof}


\subsection{The category of enriched endo-functors} 
In this subsection, we give a construction of an enriched category from enriched endo-functors. We show in Section\,\ref{sec:E0-center} that this construction realizes the $E_0$-centers of certain enriched categories.

\medskip
Let $\ec[\CB]{\CM}$ be a $\CB$-enriched category. We use $\Fun(\ec[\CB]{\CM}, \ec[\CB]{\CM})$ to denote the category of $\CB$-functors from $\ec[\CB]{\CM}$ to itself and $\CB$-natural transformations between them (recall Definition \ref{def:A-functor}). For every $\ec{F}, \ec{G}\in \Fun(\ec[\CB]{\CM}, \ec[\CB]{\CM})$, we define a category $\CP_{(\ec{F}, \ec{G})}$ as follows:
\begin{itemize}
    \item The objects are pairs $(a, \{ a_x\}_{x \in \CM})$, where $a \in \FZ_1(\CB)$ and $a_x: \forget(a) \to \ec[\CB]{\CM}(F(x), G(x))$ are morphisms in $\CB$ such that the following diagram commutes, 
\small \be \label{diag:E_0_center_I_left_center_def}
\begin{array}{c}
\xymatrix@R=0.15in @C=-0.05in{
            \forget(a) \otimes \ec[\CB]{\CM}(x,y) \ar[r]^-{a_y \otimes \ec{F}_{x,y}} & \ec[\CB]{\CM}(F(y),G(y)) \otimes \ec[\CB]{\CM}(F(x),F(y))\ar[dd]^{\circ}\\
            \ec[\CB]{\CM}(x,y) \otimes \forget(a) \ar[d]^-{\ec{G}_{x,y} \otimes a_x} \ar[u] &\\
            \ec[\CB]{\CM}(G(x),G(y)) \otimes \ec[\CB]{\CM}(F(x),G(x)) \ar[r]^-{\circ} & \ec[\CB]{\CM}(F(x),G(y))
        }
\end{array}        
\ee \normalsize
where the unlabeled arrow is given by the half-braiding of $a$. 

\item A morphism $f: (a, \{ a_x\}) \to (b, \{ b_x\})$ is a morphism $f : a \to b$ in $\FZ_1(\CB)$ such that the following diagram commutes for every $x \in \CM$.
\be \label{diag:E_0_center_I_unversal_unique_morphism}
\begin{array}{c}
\xymatrix @R=0.2in{
            \forget(a) \ar[rd]_-{ a_x} \ar[rr]^{\forget(f)} && \forget(b) \ar[ld]^-{b_x} \\
            &\ec[\CB]{\CM}(F(x),G(x)) &
        }
\end{array}
\ee
\end{itemize}

\begin{defn} \label{def:cond_star}
We say $\ec[\CB]{\CM}$ satisfies the \emph{$E_0$-center-existence condition} \eqref{cond:star} if: \small
\begin{equation*} \label{cond:star}
\mbox{for every $\ec{F}, \ec{G}\in \Fun(\ec[\CB]{\CM}, \ec[\CB]{\CM})$, the category $\CP_{(\ec{F}, \ec{G})}$ has a terminal object.} \tag{$\mathrm{CE0}$}
\end{equation*} \normalsize
We denote the terminal object of $\CP_{(\ec{F},\ec{G})}$ by $([\ec{F}, \ec{G}], \{[\ec{F}, \ec{G}]_x\}_{x \in \CM})$. 
\end{defn}

\begin{rem}\label{rem:ast_con_for_one}
    A family of morphisms $\{\xi_x: \one_\CB \to \ec[\CB]{\CM}(F(x), G(x))\}$ renders the diagram \eqref{diag:E_0_center_I_left_center_def} commutative (with $I(a), a_x$ replaced by $\one_\CB,\xi_x$, respectively) if and only if the family of morphisms $\{\xi_x\}$ defines a $\CB$-natural transformation $\ec{\xi} : \ec{F} \Rightarrow \ec{G}$. In particular, if $\ec[\CB]{\CM}$ satisfies the condition \eqref{cond:star}, then the hom set $\FZ_1(\CB)(\one_\CB, [\ec{F}, \ec{G}])$ is isomorphic to the hom set $\Fun(\ec[\CB]{\CM}, \ec[\CB]{\CM})(\ec{F}, \ec{G})$. 
\end{rem}

\begin{rem}
    Let $(A, m, \iota)$ be an algebra in a monoidal category $\CA$. By Example \ref{expl:enrich_cat_with_one_obj}, $\Fun(\ast^{A}, \ast^{A})$ is the category defined as follows:
    \begin{itemize}
        \item The objects are algebra homomorphisms $f: A \to A$.
        \item A morphism $f \to g$ is a morphism $\xi \in \CA(\one, A)$ such that the following diagram commutes.
\[
                \xymatrix @R=0.2in {
                    A \ar[r]^-{\sim} \ar[d]^-{\sim} & \one \otimes A \ar[r]^-{\xi \otimes f} & A \otimes A \ar[d]^{m}\\ 
                    A \otimes \one \ar[r]^-{g \otimes \xi} & A \otimes A \ar[r]^-{m} & A
                }
\]
    \end{itemize}
It is also clear that the object $[1_{\ast^A}, 1_{\ast^A}]$ (if exists) is the full center of $A$ (see \cite{Dav10} for the definition of the full center of an algebra).
\end{rem}

\begin{lem} \label{lem:composition}
There are two well-defined functors. 
\bnu[(1)]
\item $\CP_{(\ec{G},\ec{H})} \times \CP_{(\ec{F},\ec{G})} \to \CP_{(\ec{F},\ec{H})}$ is defined by $((b,\{ b_x\}), (a, \{ a_x \})) \mapsto (b\otimes a, \{ b_x\}\circ \{ a_x\})$, where $\{ b_x\}\circ \{ a_x\}$ is defined by the composed morphism: 
\[
\forget(b\otimes a) = \forget(b) \otimes \forget(a) \xrightarrow{b_x \otimes a_x} \ec[\CB]{\CM}(G(x),H(x)) \otimes \ec[\CB]{\CM}(F(x),G(x)) \xrightarrow{\circ} \ec[\CB]{\CM}(F(x),H(x)). 
\]
Since $[\ec{F}, \ec{H}]$ is terminal in $\CP_{(\ec{F},\ec{H})}$, we obtained a canonical morphism $b\otimes a \to [\ec{F}, \ec{H}]$.  

\item $\CP_{(\ec{F},\ec{G})} \times \CP_{(\ec{F'},\ec{G'})} \to \CP_{(\ec{FF'},\ec{GG'})}$ defined by
$(a, \{ a_x \}) \otimes (a', \{ a'_x \}) \mapsto (a\otimes a', \{ a_x \}\bullet\{ a'_x \})$, where $\{ a_x \}\bullet\{ a'_x \}$ is defined by the composed morphism: 
\[
a\otimes a' \xrightarrow{a_{G'(x)} \otimes \ec{F}_{F'(x),G'(x)}(a'_{x})} \ec[\CB]{\CM}(FG'(x),GG'(x)) \otimes \ec[\CB]{\CM}(FF'(x),FG'(x)) \xrightarrow{\circ} 
\ec[\CB]{\CM}(FF'(x),GG'(x)). 
\]
Since $[\ec{F}\ec{F'}, \ec{G}\ec{G'}]$ is terminal in $\CP_{(\ec{F},\ec{H})}$, we obtain a canonical morphism $a\otimes a' \to [\ec{F}\ec{F'}, \ec{G}\ec{G'}]$. 
\enu
\end{lem}

\begin{prop}\label{lem:E_0_center_I}
    If $\ec[\CB]{\CM}$ satisfies the condition \eqref{cond:star}, then $\Fun(\ec[\CB]{\CM}, \ec[\CB]{\CM})$ can be promoted to a strict $\FZ_1(\CB)$-enriched monoidal category $\ec[\FZ_1(\CB)]{\Fun(\ec[\CB]{\CM}, \ec[\CB]{\CM}})$. More explicitly, 
\begin{itemize}
    \item the hom objects are $\ec[\FZ_1(\CB)]{\Fun(\ec[\CB]{\CM}, \ec[\CB]{\CM})}(\ec{F}, \ec{G}) \coloneqq [\ec{F}, \ec{G}]$;
    \item the identity morphism $\one_\CB \to [\ec{F}, \ec{F}]$ is induced by the identity enriched natural transformation of $\ec{F}$ via the universal property of $[\ec{F}, \ec{F}]$;
    \item the composition morphism $[\ec{G}, \ec{H}] \otimes [\ec{F}, \ec{G}] \to [\ec{F}, \ec{H}]$ is defined by Lemma \ref{lem:composition} (1). 
            
    \item the tensor product morphism $[\ec{F}, \ec{G}] \otimes [\ec{F'}, \ec{G'}] \to [\ec{F}\ec{F}', \ec{G}\ec{G}']$ is defined by Lemma \ref{lem:composition} (2). 
            
\end{itemize}
\end{prop}

\begin{proof}
It is routine to check that $\ec[\FZ_1(\CB)]{\Fun(\ec[\CB]{\CM}, \ec[\CB]{\CM})}$ is an $\FZ_1(\CB)$-enriched category. Note that the underlying category of $\ec[\FZ_1(\CB)]{\Fun(\ec[\CB]{\CM}, \ec[\CB]{\CM})}$ is precisely $\Fun(\ec[\CB]{\CM}, \ec[\CB]{\CM})$ by Remark \ref{rem:ast_con_for_one}. This explains the notation.

It remains to show that $\ec[\FZ_1(\CB)]{\Fun(\ec[\CB]{\CM}, \ec[\CB]{\CM})}$ is monoidal. We claim the the following equation holds
\begin{gather*}
\begin{multlined}
\biggl( [\ec{F}_2, \ec{F}_3]\otimes [\ec{G}_2, \ec{G}_3] \otimes [\ec{F}_1, \ec{F}_2] \otimes [\ec{G}_1, \ec{G}_2] \to [\ec{F}_2, \ec{F}_3]\otimes [\ec{F}_1, \ec{F}_2] \otimes [\ec{G}_2, \ec{G}_3] \otimes [\ec{G}_1, \ec{G}_2] \\
\xrightarrow{\circ \otimes \circ} [\ec{F}_1, \ec{F}_3] \otimes [\ec{G}_1, \ec{G}_3] \to [\ec{F}_1\ec{G}_1, \ec{F}_3 \ec{G}_3] \xrightarrow{[\ec{F}_1\ec{G}_1, \ec{F}_3\ec{G}_3]_x} \ec[\CB]{\CM}(F_1G_1(x), F_3G_3(x)) \biggr)
\end{multlined} \\
\begin{multlined}
= \biggl( [\ec{F}_2, \ec{F}_3] \otimes [\ec{G}_2, \ec{G}_3] \otimes [\ec{F}_1, \ec{F}_2] \otimes [\ec{G}_1, \ec{G}_2] \to [\ec{F}_2\ec{G}_2, \ec{F}_3 \ec{G}_3] \otimes [\ec{F}_1\ec{G}_1, \ec{F}_2\ec{G}_2] \\
\xrightarrow{\circ} [\ec{F}_1\ec{G}_1, \ec{F}_3\ec{G}_3] \xrightarrow{[\ec{F}_1\ec{G}_1, \ec{F}_3\ec{G}_3]_x} \ec[\CB]{\CM}(F_1G_1(x), F_3G_3(x)) \biggr)
\end{multlined}
\end{gather*}
where the first unlabeled arrow is induced by the half braiding of $[\ec{F_1}, \ec{F_2}]$. Indeed, by the definition of $[\ec{G}, \ec{H}] \otimes [\ec{F}, \ec{G}] \to [\ec{F}, \ec{H}]$ and $[\ec{F}, \ec{G}] \otimes [\ec{F'}, \ec{G'}] \to [\ec{F}\ec{F'}, \ec{G}\ec{G'}]$, it not hard to check that both sides of the above equation are equal to
\tiny
    \begin{multline*}
        [\ec{F}_2, \ec{F}_3]\otimes [\ec{G}_2, \ec{G}_3] \otimes [\ec{F}_1, \ec{F}_2] \otimes [\ec{G}_1, \ec{G}_2] \xrightarrow{1 \otimes [\ec{G}_2,\ec{G}_3]_x \otimes 1 \otimes [\ec{G}_1, \ec{G}_2]_x} \\
        [\ec{F}_2, \ec{F}_3] \otimes \ec[\CB]{\CM}(G_2(x), G_3(x)) \otimes [\ec{F}_1, \ec{F}_2] \otimes \ec[\CB]{\CM}(G_1(x), G_2(x)) \xrightarrow{[\ec{F}_2, \ec{F}_3]_{G_3(x)} \otimes (\ec{F}_2)_{G_2(x), G_3(x)} \otimes [\ec{F}_1, \ec{F}_2]_{G_2(x)}\otimes (\ec{F}_1)_{G_1(x), G_2(x)}} \\
        \ec[\CB]{\CM}(F_2G_3(x), F_3G_3(x)) \otimes \ec[\CB]{\CM}(F_2G_2(x), F_2G_3(x)) \otimes \ec[\CB]{\CM}(F_1G_2(x), F_2G_2(x)) \otimes \ec[\CB]{\CM}(F_1G_1(x), F_1G_2(x)) \to \ec[\CB]{\CM}(F_1G_1(x), F_3G_3(x)).
    \end{multline*} \normalsize
Similarly, we can check that  
\[
1_{FG(x)} = \bigl( \one \simeq \one \otimes \one \xrightarrow{1 \otimes 1} [\ec{F}, \ec{F}] \otimes [\ec{G}, \ec{G}] \to [\ec{F}\ec{G}, \ec{F}\ec{G}] \xrightarrow{[\ec{F}\ec{G}, \ec{F}\ec{G}]_x} \ec[\CB]{\CM}(FG(x), FG(x)) \bigr) . 
\]
Therefore the morphisms $[\ec{F},\ec{G}] \otimes [\ec{F'}, \ec{G'}] \to [\ec{F}\ec{F'}, \ec{G}\ec{G'}]$ define an enriched functor. It is clear that the underlying monoidal category of $\ec[\FZ_1(\CB)]{\Fun(\ec[\CB]{\CM}, \ec[\CB]{\CM}})$ is $\Fun(\ec[\CB]{\CM}, \ec[\CB]{\CM})$. 
\end{proof}

\begin{rem}
The Drinfeld center $\FZ_1(\Set)$ of the category $\Set$ of sets is equivalent to $\Set$ as braided monoidal categories. Indeed, suppose $X \in \Set$ and $\gamma_{-,X} : - \times X \to X \times -$ is a half-braiding, then we have the following commutative diagram for every set $Y$ and map $f : \ast \to Y$:
\[
\xymatrix@R=1ex{
 & \ast \times X \ar[r]^{f \times 1} \ar[dd]^{\gamma_{\ast,X}} & Y \times X \ar[dd]^{\gamma_{Y,X}} \\
X \ar[ur]^{\simeq} \ar[dr]_{\simeq} \\
 & X \times \ast \ar[r]^{1 \times f} & X \times Y
}
\]
Therefore, we have $\gamma_{Y,X}(y,x) = (x,y)$ for every set $Y$ and element $y \in Y$. Thus the inclusion $\Set \to \FZ_1(\Set)$ induced by the braiding structure of $\Set$ is an equivalence of braided monoidal categories.

An ordinary category $\CM$ viewed as a $\Set$-enriched category $\ec[\Set]{\CM}$ always satisfies the condition \eqref{cond:star}. Indeed, in this case, a $\Set$-functor $\ec{F} : \ec[\Set]{\CM} \to \ec[\Set]{\CM}$ is an ordinary functor $F : \CM \to \CM$ and the terminal object $[\ec{F},\ec{G}] \in \FZ_1(\Set) \simeq \Set$ is given by the set of natural transformations $\mathrm{Nat}(F,G)$. As a consequence, $\ec[\FZ_1(\Set)]{\fun(\ec[\Set]{\CM},\ec[\Set]{\CM})}$ is precisely the functor category $\fun(\CM,\CM)$.
\end{rem}

In the rest of this subsection, we discuss the construction of $\ec[\FZ_1(\CB)]{\Fun(\ec[\CB]{\CM}, \ec[\CB]{\CM}})$ when $\ec[\CB]{\CM}=\bc[\CB]{\CM}$ for $(\CB, \CM) \in \lmod$ and $\CM$ is a left $\CB$-module. In particular, in this case, we give an easy-to-check condition that is equivalent to the condition \eqref{cond:star} (see Lemma\ \ref{lem:equivalent-to-ast}).

By Theorem \ref{thm:2-functor_from_A_modules_to_cat_A}, we can identify $\Fun(\bc[\CB]{\CM}, \bc[\CB]{\CM})$ with $\Fun^{\Lax}_{\CB}(\CM, \CM)$, i.e., the category of lax $\CB$-module functors and $\CB$-module natural transformations. There is an obvious monoidal functor from $\overline{\FZ_1(\CB)}$ into $\Fun_{\CB}^{\Lax}(\CM, \CM)$ which maps $a$ to the $\CB$-module functor $\forget(a) \odot -$. It is not hard to check that $\Fun^{\Lax}_{\CB}(\CM, \CM)$ is a strongly unital monoidal left $\overline{\FZ_1(\CB)}$-module with the module action induced by the monoidal functor $\overline{\FZ_1(\CB)} \to \Fun^{\Lax}_{\CB}(\CM, \CM)$ and the natural transformation
    \begin{align}
        \forget(a \otimes b) \odot FG(-) \xrightarrow{\sim} \forget(a) \odot [\forget(b) \odot FG(-)] \to \forget(a) \odot F(\forget(b) \odot G(-)), 
    \end{align}
where the first arrow is induced by the monoidal structure of $\forget$ and the left $\CB$-module structure of $\CM$ and the second arrow is induced by the lax $\CB$-module structure of $F$.

\begin{lem}  \label{lem:equivalent-to-ast}
When $(\CB, \CM) \in \lmod$ and $\CM$ is a left $\CB$-module, $\bc[\CB]{\CM}$ satisfies the condition \eqref{cond:star} if and only if $\Fun^{\Lax}_{\CB}(\CM, \CM)$ is enriched in $\overline{\FZ_1(\CB)}$. In this case, the object $[\ec{F},\ec{G}]$ is given by the internal hom $[F,G] \in \FZ_1(\CB)$, and the morphism $[\ec{F},\ec{G}]_x : [\ec{F},\ec{G}] \to \bc[\CB]{\CM}(F(x),G(x))$ is induced by $\forget([F,G]) \odot F(x) = ([F,G] \odot F)(x) \xrightarrow{(\ev_F)_x} G(x)$. 
\end{lem}

\begin{proof}    
Since both $([\ec{F}, \ec{G}], \{[\ec{F}, \ec{G}]_x\}_{x \in \CM})$ and the internal hom $([F,G],\ev_F=\{ (\ev_F)_x \}_{x\in\CM})$ are terminal objects, it suffices to show that they satisfy the same universal property. Let $F, G \in \Fun^{\Lax}_{\CB}(\CM, \CM)$ and $a \in \overline{\FZ_1(\CB)}$. Given a family of morphisms $\{\alpha_x: \forget(a) \to [F(x), G(x)]\}_{x \in \CM})$, define
\[
            \tilde{\alpha}_x \coloneqq \bigl( \forget(a) \odot F(x) \xrightarrow{\alpha_x \odot 1} [F(x), G(x)] \odot F(x) \xrightarrow{\ev_{F(x)}} G(x) \bigr) .
\]
Conversely,
\[
\alpha_x = \bigl( \forget(a) \xrightarrow{\coev_{F(x)}} [F(x),\forget(a) \odot F(x)] \xrightarrow{[1,\tilde \alpha_x]} [F(x),G(x)] \bigr) .
\] 
It is routine to check that $\{\alpha_x\}$ renders the diagram \eqref{diag:E_0_center_I_left_center_def} commutative if and only if the following diagram commutes,
\be \label{diag:left_module_1}
\begin{array}{c}
\xymatrix{
(b \otimes \forget(a)) \odot F(x) \ar[r]^{\gamma_{b,a} \odot 1} \ar[d]^{\simeq} & (\forget(a) \otimes b) \odot F(x) \ar[r]^{\simeq} & \forget(a) \odot (b \odot F(x)) \ar[d] \\
b \odot (\forget(a) \odot F(x)) \ar[d]^{1 \otimes \tilde \alpha_x} & & \forget(a) \odot F(b \odot x) \ar[d]^{\tilde \alpha_{b \odot x}} \\
b \odot G(x) \ar[rr] & & G(b \odot x)
}
\end{array}
\ee
i.e., $\tilde \alpha \colon a \odot F \Rightarrow G$ is a $\CB$-module natural transformation. 
\end{proof}

    

If $\bc[\CB]{\CM}$ satisfies the condition \eqref{cond:star}, by Proposition \ref{lem:E_0_center_I}, $\Fun_{\CB}^{\Lax}(\CM, \CM)$ can be promoted to a $\FZ_1(\CB)$-enriched monoidal category $\ec[\FZ_1(\CB)]{\Fun_{\CB}^{\Lax}(\CM, \CM)}$. 
We obtain the following result. 
\begin{prop}
When $(\CB, \CM) \in \lmod$ and $\CM$ is a left $\CB$-module, if $\Fun^{\Lax}_{\CB}(\CM, \CM)$ is enriched in $\overline{\FZ_1(\CB)}$, then we obtain a monoidal equivalence:
\[
\ec[\FZ_1(\CB)]{\Fun(\bc[\CB]{\CM}, \bc[\CB]{\CM}}) = \ec[\FZ_1(\CB)]{\Fun_{\CB}^{\Lax}(\CM, \CM)}\simeq \bc[\FZ_1(\CB)]{\Fun_{\CB}^{\Lax}(\CM, \CM)} .
\]  
\end{prop}

\begin{proof}
By Lemma\,\ref{lem:equivalent-to-ast}, we can set $[\ec{F}, \ec{G}]=[F,G]$. It suffices to verify (1) and (2).  
\bnu
\item[$(1)$] The composition $[\ec{G}, \ec{H}] \otimes [\ec{F}, \ec{G}]  \to [\ec{F}, \ec{H}]$ defined in Lemma\,\ref{lem:composition} coincides with the canonical morphism $[G,H]\otimes [F,G] \to [F,H]$. It is because both morphisms are induced from the same universal property according to Lemma\,\ref{lem:equivalent-to-ast}. 

\item[$(2)$] The tensor product morphism $[\ec{F}, \ec{G}] \otimes [\ec{F'}, \ec{G'}] \to [\ec{F}\ec{F}', \ec{G}\ec{G}']$ defined in Lemma\,\ref{lem:composition} is induced by
\[
[F,G] \otimes [F',G'] \odot FF' \to [F,G] \odot F \otimes [F',G'] \odot F' \xrightarrow{\ev_F \otimes \ev_{F'}} G G' ,
\]
where the first arrow is induced by the lax $\CB$-module functor structure of $F$. 

In other words, we need to show the outer subdiagram of the following diagram
\tiny
\[
\xymatrix@C=1em{
[F,G] \otimes [F',G'] \odot FF'(x) \ar[rr]^{1 \otimes [F',G']_x \odot 1} \ar[d] & & [F,G] \otimes [F'(x),G'(x)] \odot FF'(x) \ar[d]^{1 \otimes F \odot 1} \ar[dl] \\
[F,G] \odot F([F',G'] \odot F'(x)) \ar[r] \ar[d]^{1 \odot F(\ev_{F'})} & [F,G] \odot F([F'(x),G'(x)] \odot F'(x)) \ar[dl]|{1 \odot F(\ev_{F'(x)})} & [F,G] \otimes [FF'(x),FG'(x)] \odot FF'(x) \ar[d]^{[F,G]_{G'(x)} \odot \ev_{FF'(x)}} \ar[dll]^{1 \odot \ev_{FF'(x)}} \\
[F,G] \odot FG'(x) \ar[r]_{(\ev_F)_{G'(x)}} & GG'(x) & [FG'(x),GG'(x)] \odot FG'(x) \ar[l]^-{\ev_{FG'(x)}}
}
\] \normalsize
commutes for each $x \in \CM$. The upper left quadrangle commutes due to the naturality of the lax $\CB$-module functor structure of $F$, and the other subdiagrams commute by the adjunctions associated to internal homs.
\enu
\end{proof}

When $\CB$ is rigid, we have $\Fun^{\Lax}_{\CB}(\CM,\CM)=\Fun_{\CB}(\CM, \CM)$ \cite[Lemma\,2.10]{DSPS19}, i.e., the category of $\CB$-module functors, because the morphism $a\odot F(x) \to F(a\odot x)$ now has an inverse given by $F(a\odot x) \to (a\otimes a^L) \odot F(a\odot x) \to a \otimes F(a^L \odot (a \odot x)) \to a\odot F(x)$ for $a\in \CB,x\in\CM$. 

\begin{cor} \label{cor:B=rigid-construction}
When $(\CB, \CM) \in \lmod$ and $\CM$ is a left $\CB$-module, if $\CB$ is rigid and $\Fun_{\CB}(\CM, \CM)$ is enriched in $\overline{\FZ_1(\CB)}$, we have
\[
\ec[\FZ_1(\CB)]{\Fun(\bc[\CB]{\CM}, \bc[\CB]{\CM}}) \simeq \bc[\FZ_1(\CB)]{\Fun_{\CB}(\CM, \CM)}. 
\]
\end{cor}



\subsection{The \texorpdfstring{$E_0$}{E0}-centers of enriched categroies in \texorpdfstring{$\ecat$}{ECat}}
\label{sec:E0-center}

In this subsection, we prove that $\ec[\FZ_1(\CB)]{\Fun(\ec[\CB]{\CM}, \ec[\CB]{\CM}})$ is the $E_0$-center of $\ec[\CB]{\CM}$ in $\ecat$ when $\ec[\CB]{\CM}$ satisfies the condition \eqref{cond:star}.

\medskip
Recall that an $E_0$-algebra in $\ecat$ is a pair $(\ec[\CA]{\CL}, \ec{U})$, where $\ec{U}: \ast \to \ec[\CA]{\CL}$ is an enriched functor in $\ecat$. A left unital $(\ec[\CA]{\CL}, \ec{U})$-action on $\ec[\CB]{\CM}$ in $\ecat$ consists of an enriched functor $\ec{\odot}: \ec[\CA]{\CL} \times \ec[\CB]{\CM} \to \ec[\CB]{\CM}$ in $\ecat$ and an enriched natural isomorphism $\ec{\xi}$ as depicted in the following diagram (recall Definition \ref{def:left_unital_action}).
\begin{align}
\xymatrix @R=0.2in{
 & \ec[\CA]{\CL} \times \ec[\CB]{\CM} \ar[dr]^{\ec{\odot}} \\
{*} \times \ec[\CB]{\CM} \ar[ur]^{\ec{U} \times 1} \ar[rr]_{1_{\ec[\CB]{\CM}}} \rrtwocell<\omit>{<-2>\;\ec{\xi}} & & \ec[\CB]{\CM}
}
\end{align}
We set $\one_\CL \coloneqq U(\ast) \in \CL$. There exists a canonical enriched natural isomorphism $\ec{\eta}: \ec{\one_\CL} \Rightarrow \ec{U}$ (recall Example \ref{conv:morphism_to_x}) defined by $\one_\CA \to \hat{U}(\ast)$ (from the monoidal structure of $\hat{U}$) and the identity underlying natural transformation. It defines a left unital $(\ec[\CA]{\CL}, \ec{\one_\CL})$-action as depicted in the following diagram. 
\[
\begin{array}{c}
\xymatrix @C=0.5in @R=0.25in{
    & \ec[\CA]{\CL} \times \ec[\CB]{\CM} \ar@/^3ex/[ddr]^{\ec{\odot}} \\
            \rtwocell<\omit>{\quad \ec{\eta} \times 1}& \ec[\CA]{\CL} \times \ec[\CB]{\CM} \ar[dr]^{\ec{\odot}} \ar[u]|{1_{\ec[\CA]{\CL}} \times 1_{\ec[\CB]{\CM}}}  & \\
            \ast \times \ec[\CB]{\CM} \ar[rr]_{1_{\ec[\CB]{\CM}}} \ar@/^3ex/[uur]^{\ec{\one_{\CL}}\times 1_{\ec[\CB]{\CM}}} \ar[ur]^{\ec{U} \times 1_{\ec[\CB]{\CM}}} \rrtwocell<\omit>{<-2.5> \quad \ec{\xi}}  & & \ec[\CB]{\CM} 
}
\end{array}   
\]
In particular, the left unital $(\ec[\CA]{\CL}, \ec{\one_\CL})$-action is isomorphic to that of $(\ec[\CA]{\CL}, \ec{U})$-action. 
Therefore, in order to show that $\ec[\FZ_1(\CB)]{\Fun(\ec[\CB]{\CM}, \ec[\CB]{\CM}})$ is the $E_0$-center of $\ec[\CB]{\CM}$ in $\ecat$, it is enough to only consider the left unital $(\ec[\CA]{\CL}, \ec{\one_\CL})$-action on $\ec[\CB]{\CM}$ in $\ecat$ as depicted in the following diagram.
\begin{equation} \label{diag:left_unital_action_in_ecat}
\begin{array}{c}
\xymatrix @R=0.2in{
 & \ec[\CA]{\CL} \times \ec[\CB]{\CM} \ar[dr]^{\ec{\odot}} \\
{*} \times \ec[\CB]{\CM} \ar[ur]^{\ec{\one_\CL} \times 1} \ar[rr]_{1_{\ec[\CB]{\CM}}} \rrtwocell<\omit>{<-2>\;\ec{\xi}} & & \ec[\CB]{\CM}
}
\end{array}
\end{equation}

\begin{expl} \label{expl:E0_center_canonical_left_unital_action}
If $\ec[\CB]{\CM}$ satisfies the condition \eqref{cond:star}, then there exists a canonical left unital action  
\begin{align}\label{diag:unital_action_of_fun}
\xymatrix @C=-0.1in @R=0.3in{
    & \ec[\FZ_1(\CB)]{\Fun(\ec[\CB]{\CM}, \ec[\CB]{\CM}}) \times \ec[\CB]{\CM} \ar[dr]^{\ec{\ev}} \\
    \ast \times \ec[\CB]{\CM} \ar[ur]^{\ec{1_{\ec[\CB]{\CM}}} \times 1_{\ec[\CB]{\CM}}} \ar[rr] \rrtwocell<\omit>{<-3> \quad \ec{\alpha}} & & \ec[\CB]{\CM}
}
\end{align}
of the $E_0$-algebra $(\ec[\FZ_1(\CB)]{\Fun(\ec[\CB]{\CM}, \ec[\CB]{\CM}}), \ec{1_{\ec[\CB]{\CM}}})$ on $\ec[\CB]{\CM}$ in $\ecat$ defined as follows:
\begin{itemize}
    \item The enriched functor $\ec{\ev}$ is defined by the monoidal functor $\hat{\ev} \coloneqq \forget(-) \otimes - : \FZ_1(\CB) \times \CB \to \CB$, the map of objects $(\ec{F}, x) \mapsto F(x)$ and the family of morphisms
\begin{multline*}
\ec{\ev}_{(\ec{F}, x), (\ec{G}, y)} \coloneqq \bigl( \forget([\ec{F}, \ec{G}]) \otimes \ec[\CB]{\CM}(x,y) \xrightarrow{[\ec{F}, \ec{G}]_{y} \otimes \ec{F}_{x,y}} \\
\ec[\CB]{\CM}(F(y), G(y)) \otimes \ec[\CB]{\CM}(F(x),F(y)) \xrightarrow{\circ} \ec[\CB]{\CM}(F(x),G(y)) \bigr) ;
\end{multline*}
\item The background changing natural transformation $\hat \alpha$ of $\ec{\alpha}$ is given by the left unitor of $\CB$ and the underlying natural transformation $\alpha$ of $\ec{\alpha}$ is given by the identity natural transformation. 
\end{itemize}
\end{expl}

It is routine to check the following fact. The proof is omitted.

\begin{lem}\label{lem:left_module_induce_functor}
    Let $(\ec{\odot}, \ec{\xi})$ be a left unital $(\ec[\CA]{\CL}, \ec{\one_\CL})$-action of the $E_0$-algebra on $\ec[\CB]{\CM}$ as depicted in the diagram \eqref{diag:left_unital_action_in_ecat}. Then there exists a functor $\Phi: \CL \to \Fun(\ec[\CB]{\CM}, \ec[\CB]{\CM})$ defined as follows: 
    \begin{itemize}
        \item For each $a \in \CL$, $\Phi(a)$ is the $\CB$-functor defined by the map $x \mapsto a \odot x$ and the family of morphisms 
            \[
                \ec[\CB]{\CM}(x,y) \xrightarrow{\hat{\xi}^{-1}} \one_\CA \hodot \ec[\CB]{\CM}(x,y) \xrightarrow{1_{a} \hat{\odot} 1} \ec[\CA]{\CL}(a,a) \hodot \ec[\CB]{\CM}(x,y) \xrightarrow{\ec{\odot}_{(a,x), (a,y)}} \ec[\CB]{\CM}(a \odot x, a \odot y).
            \]
        \item For each morphism $f: \one_\CA \to \ec[\CA]{\CL}(a,b)$ in $\CL$, $\Phi(f)$ is the $\CB$-natural transformation defined by the family of morphisms
            \begin{align*}
                \one_\CB \xrightarrow{\hat{\xi}^{-1}} \one_\CA \hodot \one_\CB \xrightarrow{f \hat{\odot} 1} \ec[\CA]{\CL}(a, b) \hodot \ec[\CB]{\CM}(x, x) \xrightarrow{\ec{\odot}_{(a,x), (b,x)}} \ec[\CB]{\CM}(a\odot x, b\odot x).       
            \end{align*}
    \end{itemize}
\end{lem}

\begin{thm}\label{thm:E_0_center_in_ecat}
    If $\ec[\CB]{\CM}$ satisfies the condition \eqref{cond:star}, then $\ec[\FZ_1(\CB)]{\Fun(\ec[\CB]{\CM}, \ec[\CB]{\CM}})$ is the $E_0$-center of $\ec[\CB]{\CM}$ in $\ecat$, i.e., $\FZ_0(\ec[\CB]{\CM})\simeq \ec[\FZ_1(\CB)]{\Fun(\ec[\CB]{\CM}, \ec[\CB]{\CM}})$. 
\end{thm}

\begin{proof}
    Let $(\ec{\odot}, \ec{\xi})$ be a left unital action of an $E_0$-algebra $(\ec[\CA]{\CL}, \ec{\one_\CL})$ on $\ec[\CB]{\CM}$ as depicted in the diagram \eqref{diag:left_unital_action_in_ecat}. We first show that there exist an enriched functor $\ec{\Phi}: \ec[\CA]{\CL} \to \ec[\FZ_1(\CB)]{\Fun(\ec[\CB]{\CM}, \ec[\CB]{\CM}})$ in $\ecat$ and two enriched natural isomorphisms $\ec{\sigma}$ and $\ec{\rho}$ such that the following pasting diagrams 
\scriptsize
\be \label{diag:E_0_center_in_ecat_1}
        \begin{array}{c}
\xymatrix @C=0.05in @R=0.25in{
    & \ec[\FZ_1(\CB)]{\Fun(\ec[\CB]{\CM}, \ec[\CB]{\CM}}) \times \ec[\CB]{\CM} \ar@/^3ex/[ddr]^{\ec{\ev}} \\
            \rtwocell<\omit>{\quad \ec{\sigma} \times 1}& \ec[\CA]{\CL} \times \ec[\CB]{\CM} \ar[dr]^{\ec{\odot}} \ar[u]|{\ec{\Phi} \times 1_{\ec[\CB]{\CM}}} \rtwocell<\omit>{\ec{\rho}} & \\
            \ast \times \ec[\CB]{\CM} \ar[rr] \ar@/^3ex/[uur]^{\ec{1_{\ec[\CB]{\CM}}}\times 1_{\ec[\CB]{\CM}}} \ar[ur]^{\ec{\one_\CL} \times 1_{\ec[\CB]{\CM}}} \rrtwocell<\omit>{<-2.5> \ec{\xi}}  & & \ec[\CB]{\CM}
}
\end{array}
=
\begin{array}{c}
\xymatrix @C=-0.1in{
    & \ec[\FZ_1(\CB)]{\Fun(\ec[\CB]{\CM}, \ec[\CB]{\CM}}) \times \ec[\CB]{\CM} \ar[dr]^{\ec{\ev}} \\
    \ast \times \ec[\CB]{\CM} \ar[ur]^{\ec{1_{\ec[\CB]{\CM}}} \times 1_{\ec[\CB]{\CM}}} \ar[rr] \rrtwocell<\omit>{<-3> \ec{\alpha}} & & \ec[\CB]{\CM}
}
\end{array}
\ee \normalsize
are equal in $\ecat$, where the right hand side of the equation is the canonical left unital action of $(\ec[\FZ_1(\CB)]{\Fun(\ec[\CB]{\CM}, \ec[\CB]{\CM}}), \ec{1_{\ec[\CB]{\CM}}})$ on $\ec[\CB]{\CM}$ defined in Example \ref{expl:E0_center_canonical_left_unital_action}.
\bit
\item Let $\Phi: \CL \to \Fun(\ec[\CB]{\CM}, \ec[\CB]{\CM})$ be the functor described in Lemma \ref{lem:left_module_induce_functor}. Then $\Phi$ can be promoted to an enriched functor $\ec{\Phi} : \ec[\CA]{\CL} \to \ec[\FZ_1(\CB)]{\fun(\ec[\CB]{\CM},\ec[\CB]{\CM})}$ in $\ecat$ with the background changing functor $\hat{\Phi}: \CA \to \FZ_1(\CB)$ given by $- \hodot \one_\CB$ (see Example \ref{expl:centers_in_cat}) and the family of morphisms $\ec{\Phi}_{a,b}: \ec[\CA]{\CL}(a,b) \hodot \one_\CB \to [\Phi(a), \Phi(b)]$ given by the unique morphisms rendering the following diagrams commutative.
    \begin{align*}
        \xymatrix @C=0.8in @R=0.2in{
            \forget(\ec[\CA]{\CL}(a,b) \hodot \one_\CB) \ar[d]^-{1 \hat{\odot} 1_x} \ar[r]^-{I(\ec{\Phi}_{a,b})} & \forget([\Phi(a), \Phi(b)]) \ar[d]^-{[\Phi(a), \Phi(b)]_x}\\
            \ec[\CA]{\CL}(a,b) \hodot \ec[\CB]{\CM}(x,x) \ar[r]^-{\ec{\odot}_{(a,x), (b, x)}} & \ec[\CB]{\CM}(a \odot x, b \odot x)
        }
    \end{align*}

\item The enriched natural isomorphism $\ec{\sigma}$ is defined by the background changing natural transformation $\hat{\sigma}_* = \hat \xi_{\one_\CB}^{-1} : \one_\CB \to \one_\CA \hodot \one_\CB$ and the underlying natural transformation given by the unique morphism $\sigma_* : \one_\CB \to [1_{\ec[\CB]{\CM}}, \Phi(\one_\CL)]$ rendering the following diagram commutative.
    \begin{align*}
        \xymatrix @R=0.2in @C=0.6in{
            \forget(\one_\CB) \ar[r]^-{\forget(\sigma_*)} \ar[dr]_-{\xi_x^{-1}} & \forget([1_{\ec[\CB]{\CM}}, \Phi(\one_\CL)]) \ar[d]^-{[1_{\ec[\CB]{\CM}}, \Phi(\one_\CL)]_x}\\
            & \ec[\CB]{\CM}(x, \one_\CL \odot x)
        }
    \end{align*}

\item The enriched natural isomorphism $\ec{\rho}$ is defined by the background changing natural transformation
\[
\forget(\hat{\Phi}(-)) \otimes - = (- \hodot \one_\CB) \otimes - \xRightarrow{1 \otimes \hat{\xi}^{-1}} (- \hodot \one_\CB) \otimes (\one_\CA \hodot -) \xRightarrow{\sim} (- \otimes \one_\CA) \hodot {(\one_\CB \otimes -)} \xRightarrow{\sim} - \hodot - 
\]
and the identity underlying natural transformation.
\eit
Then it is not hard to check that the equation \eqref{diag:E_0_center_in_ecat_1} holds.

\smallskip
Let $(\ec{\Phi_i}, \ec{\sigma_i}, \ec{\rho_i})$, $i=1,2$, be two triples such that the similar equations as depicted by \eqref{diag:E_0_center_in_ecat_1} hold. We only need to show that there exists a unique enriched natural isomorphism $\ec{\beta}: \ec{\Phi_1} \Rightarrow \ec{\Phi_2}$ such that
\small
\begin{gather}
           \begin{array}{c}
        \xymatrix @R=0.3in @C=0.3in{
            \rtwocell<\omit>{<5> \;\ec{\sigma_1}}& \ec[\FZ_1(\CB)]{\Fun(\ec[\CB]{\CM}, \ec[\CB]{\CM}})\\
               \ast \ar@/^2ex/[ur]^-{\ec{1_{\ec[\CB]{\CM}}}} \ar[r]_-{\ec{\one_\CL}} & \ec[\CA]{\CL}  \utwocell^{\ec{\Phi_1}}_{\ec{\Phi_2}}{\ec{\beta}}
        }
   \end{array} =
    \begin{array}{c}
    \xymatrix @R=0.3in @C=0.3in{
        \rtwocell<\omit>{<5> \;\ec{\sigma_2}}& \ec[\FZ_1(\CB)]{\Fun(\ec[\CB]{\CM}, \ec[\CB]{\CM}})\\
        \ast \ar@/^2ex/[ur]^-{\ec{1_{\ec[\CB]{\CM}}}} \ar[r]_-{\ec{\one_\CL}} & \ec[\CA]{\CL} \ar[u]_-{\ec{\Phi_2}}
    }
    \end{array}, \label{diag:E_0_center_in_ecat_2} \\
   \begin{array}{c}
       \xymatrix @R=0.3in @C=0.3in{
           \ec[\FZ_1(\CB)]{\Fun(\ec[\CB]{\CM}, \ec[\CB]{\CM}}) \times \ec[\CB]{\CM} \ar@/^2ex/[dr]^-{\ec{\ev}} \rtwocell<\omit>{<5> \;\ec{\rho_2}} & \\
        \ec[\CA]{\CL} \times \ec[\CB]{\CM} \utwocell<4.5>^{\ec{\Phi_1} \times 1 \quad}_{\quad \ec{\Phi_2} \times 1}{\ec{\beta} \times 1} \ar[r]_-{\ec{\odot}}& \ec[\CB]{\CM}
    }
    \end{array} = \hspace{-3em}
    \begin{array}{c}
       \xymatrix @R=0.3in @C=0.3in{
        \ec[\FZ_1(\CB)]{\Fun(\ec[\CB]{\CM}, \ec[\CB]{\CM}}) \times \ec[\CB]{\CM} \ar@/^2ex/[dr]^-{\ec{\ev}} \rtwocell<\omit>{<5> \;\ec{\rho_1}} & \\
        \ec[\CA]{\CL} \times \ec[\CB]{\CM} \ar[u]^-{\ec{\Phi_1} \times 1} \ar[r]_-{\ec{\odot}}& \ec[\CB]{\CM}
    }
    \end{array}. \label{diag:E_0_center_in_ecat_3}
\end{gather} \normalsize

Since $\FZ_1(\CB)$ is the $E_0$-center of $\CB$ in $\Alg_{E_1}(\cat)$ (see Example \ref{expl:centers_in_cat}), there exists a unique monoidal natural isomorphism $\hat{\beta}: \hat{\Phi}_1 \Rightarrow \hat{\Phi}_2$ such that  
\begin{gather*}
            \begin{array}{c}
        \xymatrix @R=0.3in @C=0.5in{
            \rtwocell<\omit>{<4.5> \;\hat{\sigma}_1}& \FZ_1(\CB)\\
                \ast \ar@/^2ex/[ur]^-{\one_\CB} \ar[r]_-{\one_\CA} & \CA  \utwocell^{\hat{\Phi}_1}_{\hat{\Phi}_2}{\hat{\beta}}
        }
   \end{array}=
    \begin{array}{c}
    \xymatrix @R=0.3in @C=0.4in{
        \rtwocell<\omit>{<4.5> \;\hat{\sigma}_2}& \FZ_1(\CB)\\
        \ast \ar@/^2ex/[ur]^-{\one_\CB} \ar[r]_-{\one_\CA} & \CA \ar[u]_-{\hat{\Phi}_2}
    }
    \end{array}, \\
\begin{array}{c}
       \xymatrix @R=0.3in @C=0.5in{
           \FZ_1(\CB) \times \CB \ar@/^2ex/[dr]^-{\hat{\ev}} \rtwocell<\omit>{<4.5> \;\hat{\rho}_2} & \\
        \CA \times \CB \utwocell<4.5>^{\hat{\Phi}_1 \times 1 \quad}_{\quad \hat{\Phi}_2 \times 1}{\hat{\beta} \times 1} \ar[r]_-{\hat{\odot}}& \CB
    }
    \end{array} =
    \begin{array}{c}
       \xymatrix @R=0.3in @C=0.4in{
        \FZ_1(\CB) \times \CB \ar@/^2ex/[dr]^-{\hat{\ev}} \rtwocell<\omit>{<4.5> \;\hat{\rho}_1} & \\
        \CA \times \CB \ar[u]^-{\hat{\Phi}_1 \times 1} \ar[r]_-{\hat{\odot}}& \CB
    }
    \end{array}.
\end{gather*}
Define $\beta_a: \one_\CB \to [\Phi_1(a), \Phi_2(a)]$ to be the unique morphism rendering the following diagram 
\begin{align*}
    \begin{array}{c}
    \xymatrix @R=0.2in{
        \one_\CB= \forget(\one_\CB)  \ar[r]^-{\forget(\beta_a)} \ar[rd]_(0.4){(\rho_2)_{(a,x)}^{-1} \circ (\rho_1)_{(a,x)}} &  \forget([\Phi_1(a), \Phi_2(a)]) \ar[d]^-{[\Phi_1(a), \Phi_2(a)]_x}\\
        & \ec[\CB]{\CM}(\Phi_1(a)(x), \Phi_2(a)(x)) 
    }
    \end{array}
\end{align*}
commutative. Then the enriched natural isomorphism $\ec{\beta}$ define by $(\hat{\beta}, \{\beta_a\})$ is the unique enriched natural isomorphism such that the equations \eqref{diag:E_0_center_in_ecat_2} and \eqref{diag:E_0_center_in_ecat_3} hold.
\end{proof}

\begin{rem}
It is straightforward to check that the usual monoidal structure on $\ec[\FZ_1(\CB)]{\Fun(\ec[\CB]{\CM}, \ec[\CB]{\CM}})$ induced by the composition of $\CB$-functors coincides with the $E_1$-algebra structure of the $E_0$-center of $\ec[\CB]{\CM}$ in $\ecat$ (see Proposition \ref{prop:E_0_center_is_E_1_algebra}).
\end{rem}

The following result follows immediately from Corollary  \ref{cor:B=rigid-construction} and Theorem \ref{thm:E_0_center_in_ecat}. 

\begin{cor} \label{cor:B=rigid-E0-center}
When $(\CB, \CM) \in \lmod$ and $\CM$ is a left $\CB$-module, if $\CB$ is rigid and $\Fun_{\CB}(\CM, \CM)$ is enriched in $\overline{\FZ_1(\CB)}$, we have
\[
\FZ_0(\bc[\CB]{\CM}) \simeq \bc[\FZ_1(\CB)]{\Fun_{\CB}(\CM, \CM)}. 
\]
\end{cor}

\begin{rem}
Corollary \ref{cor:B=rigid-E0-center} has important applications in physics. In particular, it provides a mathematical foundation to Definition 3.18 and Remark 3.19 in \cite{KZ21}. More precisely, the macroscopic observables of a 0+1D boundary phase of a 1+1D quantum liquid $\SX$ form an enriched category, the $E_0$-center of which is the precisely the enriched monoidal category summarizing macroscopic observables of $\SX$. This fact is precisely a manifestation of the so-called boundary-bulk relation \cite{KWZ15,KWZ17}. Examples of this fact can be found in 1+1D CFT's \cite{KZ18b,KZ20,KZ21} and 1+1D lattice models \cite{KWZ22,XZ22}. 

We give an illustrating example from \cite{KWZ22}. The symmetric phase realized in the 1+1D Ising model can be described by the enriched monoidal category $\bc[\FZ_1(\mathrm{Rep}(\Zb_2))]{\mathrm{Rep}(\Zb_2)}$, where $\rep(\Zb_2)$ is the category of finite dimensional $\Zb_2$-representations. Its has two types of boundaries. The symmetric boundary can be described the enriched category $\bc[\mathrm{Rep}(\Zb_2)]{\mathrm{Rep}(\Zb_2)}$; the symmetry-breaking boundary can be described by $\bc[\mathrm{Vec}_{\Zb_2}]{\mathrm{Vec}}$, where $\mathrm{Vec}_{\Zb_2}$ is the category of $\Zb_2$-graded finite dimensional vector spaces. In this case, we have $\bc[\FZ_1(\mathrm{Rep}(\Zb_2))]{\mathrm{Rep}(\Zb_2)}\simeq \FZ_0(\bc[\mathrm{Rep}(\Zb_2)]{\mathrm{Rep}(\Zb_2)}) \simeq \FZ_0(\bc[\mathrm{Vec}_{\Zb_2}]{\mathrm{Vec}})$ and $\mathrm{Rep}(\Zb_2) \simeq \Fun_{\mathrm{Vec}_{\Zb_2}}(\mathrm{Vec}, \mathrm{Vec})$. 
\end{rem}






\section{Enriched braided monoidal categories} \label{sec:enriched_braided_monoidal_categories}

In this section, we study enriched braided monoidal categories, canonical construction and the $E_1$-centers of enriched monoidal categories.

\subsection{Definitions and examples}

For any two enriched categories $\ec[\CA]{\CL}$ and $\ec[\CB]{\CM}$, there is an obvious switching enriched functor $\ec{\Sigma}: \ec[\CA]{\CL} \times \ec[\CB]{\CM} \to \ec[\CB]{\CM} \times \ec[\CA]{\CL}$ (i.e., the braiding in $\ecat$). It is clear that $\ec{\Sigma}$ has an obvious monoidal structure if $\ec[\CA]{\CL}$ and $\ec[\CB]{\CM}$ are enriched monoidal categories.  

\begin{defn}
An \textit{enriched braided monoidal category} consists of the following data:
\bit
\item an enriched monoidal category $\ec[\CA]{\CL}$;
\item an enriched natural isomorphism $\ec{c}: \ec{\otimes} \to \ec{\otimes} \circ \ec{\Sigma}$;
\eit
such that    
   \begin{enumerate}
       \item The background category $\CA$ is symmetric.
       \item The background changing natural transformation $\hat c$ is equal to the braiding of $\CA$.
       \item The underlying monoidal category $\CL$ equipped with the underlying natural transformation $c$ of $\ec{c}$ is a braided monoidal category (called the underlying braided monoidal category of $\ec[\CA]{\CL}$).
   \end{enumerate}
\end{defn}


\begin{defn}
    An enriched monoidal functor $\ec{F}: \ec[\CA]{\CL} \to \ec[\CB]{\CM}$ between two enriched braided monoidal categories is an \textit{enriched braided monoidal functor} if the underlying functor $F: \CL \to \CM$ is a braided monoidal functor.
\end{defn}

\begin{expl}
Let $\ec[\CA]{\CL}$ be an enriched braided monoidal category. There is an $\CA$-enriched braided monoidal category $\ec[\CA]{\overline{\CL}}$ whose underlying braided monoidal category equals to $\overline{\CL}$, defined as follows. As an enriched monoidal category $\ec[\CA]{\overline{\CL}} = \ec[\CA]{\CL}$, but the braiding of $\ec[\CA]{\overline{\CL}}$ is given by the anti-braiding $\bar c_{x,y} \coloneqq c_{y,x}^{-1}$ of $\CL$ and the braiding of $\CA$. Since $\CA$ is symmetric, $(\hat c,c)$ is an enriched natural transformation if and only if $(\hat c,\bar c)$ is.
\end{expl}

\begin{expl}
Let $A$ be a commutative algebra in a symmetric monoidal category $\CA$. Then the $\CA$-enriched monoidal category $\ast^A$ introduced in Example \ref{expl:comm_alg_B_en_monoidal_cat} is an enriched braided monoidal category with the braiding given by $c_{*,*} = 1_*$. 
\end{expl}

\subsection{The canonical construction}

Let $\CA$ be a symmetric monoidal category viewed as an algebra in $\Alg_{E_2}^{\mathrm{oplax}}(\cat)$.

\begin{defn}\label{def:oplax_braided_monoidal_module}
A \textit{braided monoidal left $\CA$-\oplax module} is a left $\CA$-\oplax module in $\Alg_{E_2}^{\mathrm{oplax}}(\cat)$.
\end{defn}

\begin{rem}
More explicitly, a braided monoidal left $\CA$-\oplax module $\CL$ is both a braided monoidal category and a monoidal left $\CA$-\oplax module such that the action $\odot \colon \CA \times \CL \to \CL$ is a braided oplax-monoidal functor (i.e., the following diagram commutes, 
\be \label{diag:oplax_braided_module}
\begin{array}{c}
\xymatrix{
(a \otimes b) \odot (x \otimes y) \ar[r] \ar[d]^{\hat c_{a,b} \odot c_{x,y}} & (a \odot x) \otimes (b \odot y) \ar[d]^{c_{a \odot x,b \odot y}} \\
(b \otimes a) \odot (y \otimes x) \ar[r] & (b \odot y) \otimes (a \odot x)
}
\end{array}
\ee
where $c$ is the braiding of $\CL$ and $\hat c$ is the braiding of $\CA$).

If, in addition, $\CL$ is a monoidal left $\CA$-module, then $\CL$ is called a \textit{braided monoidal left $\CA$-module}.
\end{rem}

\begin{expl} \label{expl:oplax_braided_monoidal_module_functor}
 Let $\CA$ be a symmetric monoidal category and $\CL$ be a braided monoidal category. If $\varphi: \CA \to \FZ_2(\CL)$ is a symmetric oplax-monoidal functor, then $\CL$ is a braided monoidal left $\CA$-\oplax module with the module action $\varphi(-) \otimes - : \CA \times \CL \to \CL$. The monoidal left $\CA$-\oplax module structure of $\CL$ is induced by the composite functor $\CA \to \FZ_2(\CL) \to \FZ_1(\CL)$.
\end{expl}

\begin{rem}
If $(\CL,\odot:\CA \times \CL \to\CL)$ is a braided monoidal left $\CA$-module, the functor $\varphi \coloneqq (- \odot \one_\CL)$ is a symmetric monoidal functor $\varphi : \CA \to \FZ_2(\CL)$ (see Example \ref{expl:centers_in_cat}). 
A braided monoidal left $\CA$-module $\CL$ can be equivalently defined by a braided monoidal category $\CL$ equipped with a symmetric monoidal functor $\CA \to \FZ_2(\CL)$.
\end{rem}

\begin{prop} \label{prop:braided_en_braided_oplax_module}
Let $\CA$ be a symmetric monoidal category and $\CL$ a braided monoidal category. Let $\hat{c}$ and $c$ be the braidings of $\CA$ and $\CL$, respectively. If $\CL$ is also a strongly unital monoidal left $\CA$-\oplax module that is enriched in $\CA$, then the pair $(\hat c,c)$ defines a braiding structure on $\bc[\CA]{\CL}$ if and only if $\CL$ is a braided monoidal left $\CA$-\oplax module.  
\end{prop}

\begin{proof}
The pair $(\hat c,c)$ is a 2-morphism in $\lmod$ if and only if the diagram \eqref{diag:oplax_braided_module} commutes. 
\end{proof}

\begin{expl}
    Let $\varphi: \CA \to \FZ_2(\CL)$ be a braided oplax-monoidal functor, where $\CA$ is a symmetric monoidal category and $\CL$ is a braided monoidal category. Then $(\CL, \varphi)$ is a monoidal left $\CA$-\oplax module (see Example \ref{expl:oplax_braided_monoidal_module_functor}). If $(\CL, \varphi)$ is strongly unital and $\CL$ is enriched in $\CA$, then the enriched monoidal category $\bc[\CA]{\CL}$ constructed in Example \ref{expl:canonical_construction_enriched_monoidal_cat} (note that $\overline{\CA} = \CA$), together with the braiding in $\CA$ and $\CL$, is an enriched braided monoidal category.
\end{expl}

\begin{defn} \label{def:braided_monoidal_lax_module_functor}
    Let $\CL$ be a braided monoidal left $\CA$-\oplax module and $\CM$ be a braided monoidal left $\CB$-\oplax module. Given a symmetric monoidal functor $\hat{F}: \CA \to \CB$, a monoidal $\hat{F}$-lax functor $F: \CL \to \CM$ is called a \textit{braided monoidal $\hat{F}$-lax functor} if $F$ is braided monoidal.
\end{defn}

The proof of the following proposition is clear.

\begin{prop} \label{prop:enriched_braided_functor_braided_lax_functor}
Let $\CA,\CB$ be symmetric monoidal categories. Suppose $\CL$ is a strongly unital braided monoidal left $\CA$-\oplax module that is enriched in $\CA$, and $\CM$ is a strongly unital oplax braided monoidal left $\CB$-module that is enriched in $\CB$. Then the canonical construction gives enriched braided monoidal categories $\bc[\CA]{\CL}$ and $\bc[\CB]{\CM}$. Suppose $\hat F : \CA \to \CB$ is a symmetric monoidal functor and $F : \CL \to \CM$ is a monoidal $\hat F$-lax functor. Then $F$ is a braided monoidal $\hat{F}$-lax functor if and only if $\hat F$, together with $F$, defines an enriched braided monoidal functor $\ec{F}: \bc[\CA]{\CL} \to \bc[\CB]{\CM}$.
\end{prop}

We use $\ecat^{\mathrm{br}}$ to denote the 2-category of enriched braided monoidal categories, enriched braided monoidal functors and enriched monoidal natural transformations. The 2-category $\ecat^{\mathrm{br}}$ is symmetric monoidal with the tensor product given by the Cartesian product and the tensor unit given by $\ast$. 

We define the 2-category $\lmod^{\mathrm{br}}$ as follows: 
\bit
    \item The objects are pairs $(\CA, \CL)$, where $\CA$ is a symmetric monoidal category and $\CL$ is a strongly unital braided monoidal left $\CA$-\oplax module that is enriched in $\CA$.

    \item A 1-morphism $(\CA, \CL) \to (\CB, \CM)$ is a pair $(\hat{F}, F)$, where $\hat{F}: \CA \to \CB$ is a symmetric monoidal functor and $F:\CL \to \CM$ is a braided monoidal $\hat{F}$-lax functor. 

    \item A 2-morphism $(\hat{F}, F) \Rightarrow (\hat{G}, G)$ is a pair $(\hat{\xi},\xi)$, where
    $\hat{\xi}: \hat{F} \Rightarrow \hat{G}$ is a monoidal natural transformation and $\xi:F\Rightarrow G$ is a monoidal $\hat{\xi}$-lax natural transformation. 
\eit
The horizontal/vertical composition is induced by the horizontal/vertical composition of functors and natural transformations. The 2-category $\lmod^{\mathrm{br}}$ is symmetric monoidal with the tensor product defined by the Cartesian product and the tensor unit given by $(\ast, \ast)$. 


\medskip



\begin{thm} \label{thm:canonical_construction_2-functor_braided}
The canonical construction induces a symmetric monoidal 2-functor $\lmod^{\mathrm{br}} \to \ecat^{\mathrm{br}}$. Moreover, this 2-functor is locally isomorphic. 
\end{thm}

\begin{proof}
    Let $(\CA, \CL) \in \lmod^{\mathrm{br}}$. By Proposition \ref{prop:braided_en_braided_oplax_module}, $\bc[\CA]{\CL}$ is an enriched braided monoidal category. Since the canonical construction induces a locally isomorphic 2-functor from $\lmod^{\otimes}$ to $\ecat^{\otimes}$, the canonical construction also defines a locally isomorphic 2-functor from $\lmod^{\mathrm{br}} \to \ecat^{\mathrm{br}}$ (see Proposition \ref{prop:enriched_braided_functor_braided_lax_functor} and Proposition \ref{prop:enriched_monoidal_natural_transf_can}). Since the background changing functor and the underlying functor of the enriched isomorphism $\ec{\chi}_{\bc[\CA]{\CL}, \bc[\CB]{\CM}}: \bc[\CA]{\CL} \times \bc[\CB]{\CM} \to \bc[(\CA \times \CB)]{(\CL \times \CM)}$ defined in the proof of Theorem \ref{thm:2-functor_from_A_modules_to_cat_A} are both identity functors, $\ec{\chi}_{\bc[\CA]{\CL}, \bc[\CB]{\CM}}$ is an enriched braided monoidal functor. Then the canonical construction 2-functor is a  symmetric monoidal 2-functor by Theorem \ref{thm:sym_mono_2_functor_monoidal}.
\end{proof}


\subsection{The Drinfeld center of an enriched monoidal category} \label{sec:E1_center_enriched}
In this subsection, we provide an explicit construction of enriched braided monoidal category from an enriched monoidal category imitating the usual construction of Drinfeld center.

\begin{defn}
Let $\ec[\CA]{\CL}$ be an enriched monoidal category. A half-braiding for an object $x \in \ec[\CA]{\CL}$ is an $\CA$-natural isomorphism $\ec{\beta}_{-,x} : - \otimes x \to x \otimes -$ such that the underlying natural isomorphism $\beta_{-,x}$ is a usual half-braiding.
\end{defn}

In other words, the half-braiding $\ec{\beta}_{-,x}$ is a half-braiding $\beta_{-,x}$ on the underlying monoidal category such that the following diagram commutes. 
\[
\xymatrix{
\ec[\CA]{\CL}(y,z) \ar[r]^-{1 \otimes 1_x} \ar[d]^{1_x \otimes 1} & \ec[\CA]{\CL}(y,z) \otimes \ec[\CA]{\CL}(x,x) \ar[r]^-{\ec{\otimes}} & \ec[\CA]{\CL}(y \otimes x,z \otimes x) \ar[dd]^{\ec[\CA]{\CL}(1,\beta_{z,x})} \\
\ec[\CA]{\CL}(x,x) \otimes \ec[\CA]{\CL}(y,z) \ar[d]^{\ec{\otimes}} \\
\ec[\CA]{\CL}(x \otimes y,x \otimes z) \ar[rr]^{\ec[\CA]{\CL}(\beta_{y,x},1)} & & \ec[\CA]{\CL}(y \otimes x,x \otimes z)
}
\]

\begin{defn}
Let $\ec[\CA]{\CL}$ be an enriched monoidal category. Suppose $x,y \in \ec[\CA]{\CL}$ and $\ec{\beta_{-,x}},\ec{\beta_{-,y}}$ are half-braidings. Define an object $[(x,\ec{\beta_{-,x}}),(y,\ec{\beta_{-,y}})] \in \FZ_2(\CA)$ to be the terminal one (if exists) among all pairs $(a,\zeta)$, where $a \in \FZ_2(\CA)$ and $\zeta : a \to \ec[\CA]{\CL}(x,y)$ is a morphism in $\CA$ such that the following diagram commutes,
\be \label{diag:enriched_half_braiding}
\begin{array}{c}
\xymatrix@C=2em{
a \ar[r]^-{1_z \otimes \zeta} \ar[d]_{\zeta \otimes 1_z} & \ec[\CA]{\CL}(z,z) \otimes \ec[\CA]{\CL}(x,y) \ar[r]^{\ec{\otimes}} & \ec[\CA]{\CL}(z \otimes x,z \otimes y) \ar[d]_{\ec[\CA]{\CL}(1,\beta_{z,y})} \\
\ec[\CA]{\CL}(x,y) \otimes \ec[\CA]{\CL}(z,z) \ar[r]_{\ec{\otimes}} & \ec[\CA]{\CL}(x \otimes z,y \otimes z) \ar[r]_{\ec[\CA]{\CL}(\beta_{z,x},1)} & \ec[\CA]{\CL}(z \otimes x,y \otimes z)
}
\end{array}
\ee
(i.e., $a \xrightarrow{\zeta} \ec[\CA]{\CL}(x,y)$ equalizes two morphisms $\ec[\CA]{\CL}(x,y) \rightrightarrows \ec[\CA]{\CL}(z \otimes x,y \otimes z)$, for every $z \in \ec[\CA]{\CL}$).
\end{defn}


\begin{rem}
By the universal property, the morphism $[(x,\ec{\beta_{-,x}}),(y,\ec{\beta_{-,y}})] \to \ec[\CA]{\CL}(x,y)$ is monic. Thus $[(x,\ec{\beta_{-,x}}),(y,\ec{\beta_{-,y}})]$ is the maximal subobject $a$ of $\ec[\CA]{\CL}(x,y)$ in $\FZ_2(\CA)$ such that the diagram \eqref{diag:enriched_half_braiding} commutes.
\end{rem}

\begin{defn} \label{defn:cond_star_2}
We say that an enriched monoidal category $\ec[\CA]{\CL}$ satisfies the \emph{$E_1$-center-existence} condition \eqref{cond:star_2} if:
\small
\begin{equation*} \label{cond:star_2}
\mbox{for all objects $x,y \in \ec[\CA]{\CL}$ and half-braidings $\ec{\beta_{-,x}},\ec{\beta_{-,y}}$, the object $[(x,\ec{\beta_{-,x}}),(y,\ec{\beta_{-,y}})]$ exists.} \tag{$\mathrm{CE1}$}
\end{equation*} \normalsize
\end{defn}

\begin{defn}
Let $\ec[\CA]{\CL}$ be an enriched monoidal category that satisfies the condition \eqref{cond:star_2}. Then we define an $\FZ_2(\CA)$-enriched braided monoidal category $\ec[\FZ_2(\CA)]{\hbcat_1(\ec[\CA]{\CL})}$ as follows:
\bit
\item The objects are pairs $(x,\ec{\beta_{-,x}})$, where $x \in \ec[\CA]{\CL}$ and $\ec{\beta_{-,x}}$ is a half-braiding.
\item The hom objects are $\ec[\FZ_2(\CA)]{\hbcat_1(\ec[\CA]{\CL})}((x,\ec{\beta_{-,x}}),(y,\ec{\beta_{-,y}})) \coloneqq [(x,\ec{\beta_{-,x}}),(y,\ec{\beta_{-,y}})] \in \FZ_2(\CA)$.
\item The composition and identity morphisms are induced by those of $\ec[\CA]{\CL}$.
\item The monoidal structure is induced by that of $\ec[\CA]{\CL}$.
\item The braiding is given by $(x,\ec{\beta_{-,x}}) \otimes (y,\ec{\beta_{-,y}}) \xrightarrow{\beta_{x,y}} (y,\ec{\beta_{-,y}}) \otimes (x,\ec{\beta_{-,x}})$.
\eit
This enriched braided monoidal category $\ec[\FZ_2(\CA)]{\hbcat_1(\ec[\CA]{\CL})}$ is called the Drinfeld center of $\ec[\CA]{\CL}$. Its underlying braided monoidal category is denoted by $\hbcat_1(\ec[\CA]{\CL})$.  
\end{defn}

\begin{rem}
An ordinary monoidal category $\CL$ viewed as a $\Set$-enriched monoidal category $\ec[\Set]{\CL}$ always satisfies the condition \eqref{cond:star_2}. Indeed, in this case, a half-braiding $\ec{\beta_{-,x}} : - \otimes x \to x \otimes -$ is a usual half-braiding $\beta_{-,x}$ and the terminal object
\[
[(x,\ec{\beta_{-,x}}),(y,\ec{\beta_{-,y}})] \in \FZ_2(\Set) \simeq \Set
\]
is precisely the hom set $\FZ_1(\CL)((x,\beta_{-,x}),(y,\beta_{-,y}))$. As a consequence, $\ec[\FZ_2(\Set)]{\hbcat_1(\ec[\Set]{\CL})}$ is precisely the usual Drinfeld center \cite{Maj91,JS91} $\FZ_1(\CL)$.
\end{rem}

\begin{rem}
This construction is different from the Drinfeld center defined in \cite[Definition\ 4.2]{KZ18a}. The hom spaces of the Drinfeld center defined in \cite{KZ18a} are maximal subobjects in $\CA$, not $\FZ_2(\CA)$. When $\CA$ is a non-degenerate braided fusion category and $\CL$ is an indecomposable multi-fusion left $\CA$-module, these two definitions coincide (see Remark \ref{rem:E1_center_physical}). The physical meaning of the Drinfeld center defined in \cite{KZ18a} is given in \cite[Section 4.1]{KZ21}.
\end{rem}

\begin{lem}
Under the condition \eqref{cond:star_2}, the Drinfeld center $\ec[\FZ_2(\CA)]{\hbcat_1(\ec[\CA]{\CL})}$ is a well-defined enriched braided monoidal category.
\end{lem}

\begin{proof}
For any $(x,\ec{\beta_{-,x}}) \in \ec[\FZ_2(\CA)]{\hbcat_1(\ec[\CA]{\CL})}$ the diagram
\[
\xymatrix{
\one \ar[r]^-{1_z \otimes 1_x} \ar[d]_{1_x \otimes 1_z} & \ec[\CA]{\CL}(z,z) \otimes \ec[\CA]{\CL}(x,x) \ar[r]^{\ec{\otimes}} & \ec[\CA]{\CL}(z \otimes x,z \otimes x) \ar[d]^{\ec[\CA]{\CL}(1,\beta_{z,x})} \\
\ec[\CA]{\CL}(x,x) \otimes \ec[\CA]{\CL}(z,z) \ar[r]_{\ec{\otimes}} & \ec[\CA]{\CL}(x \otimes z,x \otimes z) \ar[r]_{\ec[\CA]{\CL}(\beta_{z,x},1)} & \ec[\CA]{\CL}(z \otimes x,x \otimes z)
}
\]
commutes because both two composite morphisms are equal to $\beta_{z,x}$. Thus the identity morphism $1_x : \one \to \ec[\CA]{\CL}(x,x)$ factors through $[(x,\ec{\beta_{-,x}}),(x,\ec{\beta_{-,x}})]$.

For any $(x,\ec{\beta_{-,x}}) , (y,\ec{\beta_{-,y}}) , (z,\ec{\beta_{-,z}}) \in \ec[\FZ_2(\CA)]{\hbcat_1(\ec[\CA]{\CL})}$, the diagram
\scriptsize\[
\xymatrix{
[(y,\ec{\beta_{-,y}}),(z,\ec{\beta_{-,z}})] [(x,\ec{\beta_{-,x}}),(y,\ec{\beta_{-,y}})] \ar[r] \ar[d] & \ec[\CA]{\CL}(y,z) \ec[\CA]{\CL}(x,y) \ar[d] \ar[dr]^{\circ} \\
\ec[\CA]{\CL}(y,z) \ec[\CA]{\CL}(x,y) \ar[d] \ar@/_15ex/[dd]_(0.7){\circ} & \ec[\CA]{\CL}(wy,wz) \ec[\CA]{\CL}(wx,wy) \ar[d]_{\ec[\CA]{\CL}(\beta_{w,y}^{-1},\beta_{w,z}) \ec[\CA]{\CL}(1,\beta_{w,y})} \ar[dr]^{\circ} & \ec[\CA]{\CL}(x,z) \ar[d] \\
\ec[\CA]{\CL}(yw,zw) \ec[\CA]{\CL}(xw,yw) \ar[r]_{1 \ec[\CA]{\CL}(\beta_{w,x},1)} \ar[dr]_{\circ} & \ec[\CA]{\CL}(yw,zw) \ec[\CA]{\CL}(wx,yw) \ar[dr]^{\circ} & \ec[\CA]{\CL}(wx,wz) \ar[d]^{\ec[\CA]{\CL}(1,\beta_{w,z})} \\
\ec[\CA]{\CL}(x,z) \ar[r] & \ec[\CA]{\CL}(xw,zw) \ar[r]_{\ec[\CA]{\CL}(\beta_{w,x},1)} & \ec[\CA]{\CL}(wx,zw)
}
\]\normalsize
commutes: the upper left rectangle commutes by the definition of $[(x,\ec{\beta_{-,x}}),(y,\ec{\beta_{-,y}})]$, and the other subdiagrams commute due to the functoriality of $\ec{\otimes}$ and the associativity of the composition. Thus the morphism $[(y,\ec{\beta_{-,y}}),(z,\ec{\beta_{-,z}})] \otimes [(x,\ec{\beta_{-,x}}),(y,\ec{\beta_{-,y}})] \to \ec[\CA]{\CL}(y,z) \otimes \ec[\CA]{\CL}(x,y) \xrightarrow{\circ} \ec[\CA]{\CL}(x,z)$ factors through $[(x,\ec{\beta_{-,x}}),(z,\ec{\beta_{-,z}})]$. This shows that the identity and composition morphisms in $\ec[\CA]{\CL}$ induce those of $\ec[\FZ_2(\CA)]{\hbcat_1(\ec[\CA]{\CL})}$.

For the monoidal structure, consider the following diagram.
\scriptsize\[
\xymatrix{
[(x,\ec{\beta_{-,x}}),(z,\ec{\beta_{-,z}})] [(y,\ec{\beta_{-,y}}),(w,\ec{\beta_{-,w}})] \ar[r] \ar[d] \ar[ddr] & \ec[\CA]{\CL}(x,z) \ec[\CA]{\CL}(y,w) \ar[r]^-{\ec{\otimes}} & \ec[\CA]{\CL}(xy,zw) \ar[d] \\
\ec[\CA]{\CL}(x,z) \ec[\CA]{\CL}(y,w) \ar[d]_{\ec{\otimes}} & & \ec[\CA]{\CL}(axy,azw) \ar[d]^{\ec[\CA]{\CL}(1,\beta_{a,z} 1)} \\
\ec[\CA]{\CL}(xy,zw) \ar[d] & \ec[\CA]{\CL}(xaz,yaw) \ar[r]^{\ec[\CA]{\CL}(\beta_{a,x} 1,1)} \ar[d]_{\ec[\CA]{\CL}(1,1 \beta_{a,w})} & \ec[\CA]{\CL}(axy,zaw) \ar[d]^{\ec[\CA]{\CL}(1,1 \beta_{a,w})} \\
\ec[\CA]{\CL}(xya,zwa) \ar[r]_{\ec[\CA]{\CL}(1 \beta_{a,y},1)} & \ec[\CA]{\CL}(xay,zwa) \ar[r]_{\ec[\CA]{\CL}(\beta_{a,x} 1,1)} & \ec[\CA]{\CL}(axy,zwa)
}
\]\normalsize
The upper quadrangle commutes by the definitions of $[(x,\ec{\beta_{-,x}}),(z,\ec{\beta_{-,z}})]$, the left quadrangle commutes by the definition of $[(y,\ec{\beta_{-,y}}),(w,\ec{\beta_{-,w}})]$, and the lower right rectangle is obviously commutative. It follows that the morphism
\small\[
[(x,\ec{\beta_{-,x}}),(z,\ec{\beta_{-,z}})] \otimes [(y,\ec{\beta_{-,y}}),(w,\ec{\beta_{-,w}})] \to \ec[\CA]{\CL}(x,z) \otimes \ec[\CA]{\CL}(y,w) \xrightarrow{\ec{\otimes}} \ec[\CA]{\CL}(x \otimes y,z \otimes w)
\]\normalsize
factors through $[(x \otimes y,\ec{\beta_{-,x \otimes y}}),(z \otimes w,\ec{\beta_{-,z \otimes w}})]$. It is not hard to see that this defines the monoidal structure of $\ec[\FZ_2(\CA)]{\hbcat_1(\ec[\CA]{\CL})}$.

For the braiding structure, note that the morphism $\one \xrightarrow{\beta_{x,y}} \ec[\CA]{\CL}(x \otimes y,y \otimes x)$ equalizes two morphisms $\ec[\CA]{\CL}(x \otimes y,y \otimes x) \rightrightarrows \ec[\CA]{\CL}(a \otimes x \otimes y,y \otimes x \otimes a)$ because $\beta_{-,y}$ is a half-braiding in the underlying category. Therefore, $\one \xrightarrow{\beta_{x,y}} \ec[\CA]{\CL}(x \otimes y,y \otimes x)$ factors through $[(x \otimes y,\ec{\beta_{-,x \otimes y}}),(y \otimes x,\ec{\beta_{-,y \otimes x}})]$ and thus gives the braiding morphism in $\ec[\FZ_2(\CA)]{\hbcat_1(\ec[\CA]{\CL})}$, denoted by
\[
(x \otimes y,\ec{\beta_{-,x \otimes y}}) \xrightarrow{\beta_{x,y}} (y \otimes x,\ec{\beta_{-,y \otimes x}}) .
\]
We need to verify that this braiding $\ec{\beta_{-,-}}$ is an enriched natural isomorphism, i.e., the diagram
\scriptsize\[
\xymatrix{
[(x,\ec{\beta_{-,x}}),(z,\ec{\beta_{-,z}})] \otimes [(y,\ec{\beta_{-,y}}),(w,\ec{\beta_{-,w}})] \ar[r]^{c} \ar[dd]_{\ec{\otimes}} & [(y,\ec{\beta_{-,y}}),(w,\ec{\beta_{-,w}})] \otimes [(x,\ec{\beta_{-,x}}),(z,\ec{\beta_{-,z}})] \ar[d]^{\ec{\otimes}} \\
 & [(y \otimes x,\ec{\beta_{-,y \otimes x}}),(w \otimes z,\ec{\beta_{-,w \otimes z}})] \ar[d]^{[\beta_{x,y},1]} \\
[(x \otimes y,\ec{\beta_{-,x \otimes y}}),(z \otimes w,\ec{\beta_{-,z \otimes w}})] \ar[r]^{[1,\beta_{z,w}]} & [(x \otimes y,\ec{\beta_{-,x \otimes y}}),(w \otimes z,\ec{\beta_{-,w \otimes z}})]
}
\]\normalsize
commutes, where $c$ is the braiding of $\CA$. Since the morphism
\[
[(x \otimes y,\ec{\beta_{-,x \otimes y}}),(w \otimes z,\ec{\beta_{-,w \otimes z}})] \to \ec[\CA]{\CL}(x \otimes y,w \otimes z)
\]
is monic, it suffices to show that the diagram
\scriptsize\[
\xymatrix@C=2em{
[(x,\ec{\beta_{-,x}}),(z,\ec{\beta_{-,z}})] [(y,\ec{\beta_{-,y}}),(w,\ec{\beta_{-,w}})] \ar[r] \ar[d] & \ec[\CA]{\CL}(x,z) \ec[\CA]{\CL}(y,w) \ar[r]^{c} \ar[d] & \ec[\CA]{\CL}(y,w) \ec[\CA]{\CL}(x,z) \ar[d]^{\ec{\otimes}} \\
\ec[\CA]{\CL}(x,z) \ec[\CA]{\CL}(y,w) \ar[d] \ar@/_12ex/[dd]_(0.7){\ec{\otimes}} & \ec[\CA]{\CL}(wx,wz) \ec[\CA]{\CL}(yx,wx) \ar[r]^-{\circ} \ar[d]^{\ec[\CA]{\CL}(1,\beta_{z,w}^{-1}) \ec[\CA]{\CL}(\beta_{x,y},1)} & \ec[\CA]{\CL}(yx,wz) \ar[dd]^{\ec[\CA]{\CL}(\beta_{x,y},1)} \\
\ec[\CA]{\CL}(xw,zw) \ec[\CA]{\CL}(xy,xw) \ar[r]^{\ec[\CA]{\CL}(\beta_{x,w}^{-1},1) \ec[\CA]{\CL}(1,\beta_{x,w})} \ar[d]^{\circ} & \ec[\CA]{\CL}(wx,zw) \ec[\CA]{\CL}(xy,wx) \ar[dl]^{\circ} \\
\ec[\CA]{\CL}(xy,zw) \ar[rr]^{\ec[\CA]{\CL}(1,\beta_{z,w})} & & \ec[\CA]{\CL}(xy,wz)
}
\]\normalsize
commutes. The upper left rectangle commutes by the definition of $[(y,\ec{\beta_{-,y}}),(w,\ec{\beta_{-,w}})]$, the lower right pentagon and the lower middle triangle commute due to the associativity of the composition $\circ$, and the other two subdiagrams commtes due to the functoriality of $\ec{\otimes}$. Hence, we conclude that $\ec[\FZ_2(\CA)]{\hbcat_1(\ec[\CA]{\CL})}$ is a well-defined enriched braided monoidal category.
\end{proof}

\begin{rem} \label{rem:enriched_forgetful_functor}
Let $\ec[\CA]{\CL}$ be an enriched monoidal category that satisfies the condition \eqref{cond:star_2}. The forgetful functor $\ec{\forget} : \ec[\FZ_2(\CA)]{\hbcat_1(\ec[\CA]{\CL})} \to \ec[\CA]{\CL}$ is an enriched monoidal functor whose background changing functor is the inclusion $\FZ_2(\CA) \hookrightarrow \CA$.
\end{rem}

\begin{prop} \label{prop:Drinfeld_center_underlying}
The underlying category $\hbcat_1(\ec[\CA]{\CL})$ is a full subcategory of $\FZ_1(\CL)$.
\end{prop}

\begin{proof}
Suppose $(x,\ec{\beta_{-,x}}) , (y,\ec{\beta_{-,y}}) \in \ec[\FZ_2(\CA)]{\FZ_1(\ec[\CA]{\CL})}$. By definition we have
\[
(x,\ec{\beta_{-,x}}),(y,\ec{\beta_{-,y}}) \in \FZ_1(\CL).
\]
It suffices to show that
\[
\FZ_2(\CA)(\one,[(x,\ec{\beta_{-,x}}),(y,\ec{\beta_{-,y}})]) \simeq \FZ_1(\CL)((x,\beta_{-,x}),(y,\beta_{-,y})) .
\]
By the universal property, the left hand side is isomorphic to the set of morphisms $f : \one \to \ec[\CA]{\CL}(x,y)$ rendering the following diagram commutative for every $z \in \ec[\CA]{\CL}$.
\[
\xymatrix{
\one \ar[r]^-{1_z \otimes f} \ar[d]_{f \otimes 1_z} & \ec[\CA]{\CL}(z,z) \otimes \ec[\CA]{\CL}(x,y) \ar[r]^{\ec{\otimes}} & \ec[\CA]{\CL}(z \otimes x,z \otimes y) \ar[d]^{\ec[\CA]{\CL}(1,\beta_{z,y})} \\
\ec[\CA]{\CL}(x,y) \otimes \ec[\CA]{\CL}(z,z) \ar[r]_{\ec{\otimes}} & \ec[\CA]{\CL}(x \otimes z,y \otimes z) \ar[r]_{\ec[\CA]{\CL}(\beta_{z,x},1)} & \ec[\CA]{\CL}(z \otimes x,y \otimes z)
}
\]
It means that the equation
\[
\beta_{z,y} \circ (1_z \otimes f) = (f \otimes 1_z) \circ \beta_{z,x}
\]
holds in the underlying category $\CL$. In other words, $f$ is exactly a morphism $f : (x,\beta_{-,x}) \to (y,\beta_{-,y})$ in the Drinfeld center $\FZ_1(\CL)$.
\end{proof}

\begin{rem}
An ordinary monoidal category $\CL$ viewed as a $\Set$-enriched monoidal category always satisfies the condition \eqref{cond:star_2}. Indeed, in this case an enriched half-braiding $(x,\ec{\beta_{-,x}})$ is simply an ordinary half-braiding $(x,\beta_{-,x}) \in \FZ_1(\CL)$, and the hom object $[(x,\ec{\beta_{-,x}}),(y,\ec{-,y})] \in \FZ_2(\Set) = \Set$ is the hom set $\FZ_1(\CL)((x,\beta_{-,x}),(y,\beta_{-,y}))$. As a conclusion, the Drinfeld center $\ec[\FZ_2(\Set)]{\hbcat_1(\ec[\Set]{\CL})}$ is precisely the ordinary Drinfeld center $\FZ_1(\CL)$.
\end{rem}

In the rest of this subsection, we compute the Drinfeld center of an enriched monoidal category $\bc[\CC]{\CM}$ obtained via the canonical construction from the pair $(\CC,\CM) \in \lmod^{\otimes}$ such that $\CM$ is a monoidal left $\overline{\CC}$-module defined by a braided monoidal functor (recall Remark \ref{rem:left_unital_action_in_Mod})
\[
\varphi : \overline{\CC} \to \FZ_1(\CM).
\]


It is straightforward to see that an enriched half-braiding $(x,\ec{\beta_{-,x}})$ is precisely an object $(x,\beta_{-,x})$ in $\FZ_1(\CM)$ transparent to $\varphi(c) \in \FZ_1(\CM)$ for all $c \in \CC$, i.e., an object in $\FZ_2(\varphi)$, which denotes the M\"{u}ger centralizer of $\overline{\CC}$ in $\FZ_1(\CM)$ via $\varphi$. This M\"{u}ger centralizer $\FZ_2(\varphi)$ is also called the relative Drinfeld center of $\CM$. Note that the composite functor $\FZ_2(\CC) \hookrightarrow \overline{\CC} \xrightarrow{\varphi} \FZ_1(\CM)$ factors through $\FZ_2(\CC) \xrightarrow{\phi}\FZ_2(\varphi) \hookrightarrow \FZ_1(\CM)$. Moreover, the image of $\phi$ is contained in the M\"{u}ger center $\FZ_2(\FZ_2(\varphi))$ of $\FZ_2(\varphi)$. Thus $\FZ_2(\varphi)$ is a braided monoidal left $\FZ_2(\CC)$-module with the braided module structure defined by $\phi : \FZ_2(\CC) \to \FZ_2(\FZ_2(\varphi))$.

\begin{lem}
Let $\CC$ be a braided monoidal category and $\CM$ be a monoidal left $\overline{\CC}$-module that is enriched in $\overline{\CC}$. Assume that the monoidal left $\overline{\CC}$-module structure of $\CM$ is defined by a braided monoidal functor $\varphi : \overline{\CC} \to \FZ_1(\CM)$. Then the enriched braided monoidal category $\bc[\CC]{\CM}$ satisfies the condition \eqref{cond:star_2} if and only if $\FZ_2(\varphi)$ is enriched in $\FZ_2(\CC)$. Moreover, in this case the object $[(x,\ec{\beta_{-,x}}),(y,\ec{\beta_{-,y}})]$ is precisely the internal hom in $\FZ_2(\CC)$. 
\end{lem}

\pf 
It is enough to show that $[(x,\ec{\beta_{-,x}}),(y,\ec{\beta_{-,y}})]$ satisfies the same universal property as the internal hom in $\FZ_2(\CC)$. Given an object $a \in \FZ_2(\CC)$ and a morphism $\zeta : a \to [x,y]_\CC$, we can define a morphism $\tilde \zeta : \cf{\varphi}(a) \otimes x \to y$ in $\CM$ by
\[
\tilde \zeta \coloneqq \bigl( \cf{\varphi}(a) \otimes x \xrightarrow{\cf{\varphi}(\zeta) \otimes 1} \cf{\varphi}([x,y]_\CC) \otimes x \xrightarrow{\ev_x} y \bigr) .
\]
Conversely, we have
\[
\zeta = \bigl( a \xrightarrow{\coev_x} [x,\cf{\varphi}(a) \otimes x]_\CC \xrightarrow{[1,\tilde \zeta]} [x,y]_\CC \bigr) .
\]
It is routine to check that $\zeta$ renders the diagram \eqref{diag:enriched_half_braiding} commutative if and only if the following diagram commutes,
\[
\xymatrix{
z \otimes \cf{\varphi}(a) \otimes x \ar[r]^{\gamma_{z,a} \otimes 1} \ar[d]^{1 \otimes \tilde{\zeta}} & \cf{\varphi}(a) \otimes z \otimes x \ar[r]^{1 \otimes \beta_{z,x}} & \cf{\varphi}(a) \otimes x \otimes z \ar[d]^{\tilde{\zeta} \otimes 1} \\
z \otimes y \ar[rr]^{\beta_{z,y}} & & y \otimes z
}
\]
(i.e., $\tilde \zeta$ is a morphism $\tilde \zeta : a \odot (x,\ec{\beta_{-,x}}) \to (y,\ec{\beta_{-,y}})$ in $\FZ_2(\varphi)$).
\end{proof}

\begin{cor} \label{cor:ZHCM}
Let $\CC$ be a braided monoidal category and $\CM$ be a monoidal left $\overline{\CC}$-module that is enriched in $\overline{\CC}$. Assume that the monoidal left $\overline{\CC}$-module structure of $\CM$ is defined by a braided monoidal functor $\varphi : \overline{\CC} \to \FZ_1(\CM)$. If $\FZ_2(\varphi)$ is enriched in $\FZ_2(\CC)$, we have $\ec[\FZ_2(\CC)]{\hbcat_1(\bc[\CC]{\CM})} = \bc[\FZ_2(\CC)]{\FZ_2(\varphi)}$.
\end{cor}

\subsection{The \texorpdfstring{$E_1$}{E1}-centers of enriched monoidal categories in \texorpdfstring{$\ecat$}{ECat}}   \label{sec:E1-center}

In this subsection we prove that the Drinfeld center of an enriched monoidal category is the $E_1$-center in $\ecat$.

\medskip

Let $\ec[\CC]{\CM}$ be an enriched monoidal category that satisfies the condition \eqref{cond:star_2}. Then the Drinfeld center $\ec[\FZ_2(\CC)]{\hbcat_1(\bc[\CC]{\CM})}$ is a well-defined enriched braided monoidal category.

Suppose $\ec[\CA]{\CL}$ is an enriched monoidal category. A left unital action of $\ec[\CA]{\CL}$ on $\ec[\CC]{\CM}$ is an enriched monoidal functor $\ec{\odot}$ and an enriched monoidal natural isomorphism $\ec{u}$ as depicted in the following diagram:
\be \label{diag:E1_left_unital_action}
\begin{array}{c}
\xymatrix{
 & \ec[\CA]{\CL} \times \ec[\CC]{\CM} \ar[dr]^{\ec{\odot}} \\
{*} \times \ec[\CC]{\CM} \ar[ur]^{\ec{\one} \times 1} \ar[rr] \rrtwocell<\omit>{<-3>\;\ec{u}} & & \ec[\CC]{\CM}
}
\end{array} .
\ee

\begin{expl}
There is an obvious left unital action of the Drinfeld center $\ec[\FZ_2(\CC)]{\hbcat_1(\ec[\CC]{\CM})}$ on $\ec[\CC]{\CM}$:
\[
\xymatrix{
 & \ec[\FZ_2(\CC)]{\hbcat_1(\ec[\CC]{\CM})} \times \ec[\CC]{\CM} \ar[dr]^{\ec{\circledast}} \\
{*} \times \ec[\CC]{\CM} \ar[ur]^{\ec{\one} \times 1} \ar[rr] \rrtwocell<\omit>{<-3>\;\ec{v}} & & \ec[\CC]{\CM}
}
\]
The enriched monoidal functor $\ec{\circledast} \coloneqq \ec{\otimes} \circ (\ec{\forget} \times 1)$, where $\ec{\forget} : \ec[\FZ_2(\CC)]{\hbcat_1(\ec[\CC]{\CM})} \to \ec[\CC]{\CM}$ is the forgetful functor (see Remark \ref{rem:enriched_forgetful_functor}). The monoidal structure of $\ec{\circledast}$ is induced by the braided monoidal structure of $\FZ_2(\CC) \times \CC \to \CC$ and the monoidal structure of $\FZ_1(\CM) \times \CM \to \CM$ (recall that $\hbcat_1(\ec[\CC]{\CM})$ is a full subcategory of $\FZ_1(\CM)$). The background changing natural transformation of $\ec{v}$ is given by the left unitor of $\CC$ and the underlying natural transformation of $\ec{v}$ is given by the left unitor of $\CM$.
\end{expl}

\begin{thm} \label{thm:Drinfeld_center=E1_center}
The Drinfeld center is the $E_1$-center in $\ecat$, i.e., we have an equivalence of enriched braided monoidal categories $\FZ_1(\ec[\CC]{\CM})\simeq \ec[\FZ_2(\CC)]{\hbcat_1(\ec[\CC]{\CM})}$. 
\end{thm}

\begin{proof}
Suppose $(\ec{\odot},\ec{u})$ is a left unital action of an enriched monoidal category $\ec[\CA]{\CL}$ on $\ec[\CC]{\CM}$ as depicted in the diagram \eqref{diag:E1_left_unital_action}. First we show that there is an enriched monoidal functor $\ec{P} : \ec[\CA]{\CL} \to \ec[\FZ_2(\CC)]{\hbcat_1(\ec[\CC]{\CM})}$ and an enriched monoidal natural isomorphism $\ec{\rho}$ such that the following pasting diagrams
{\small
\be\label{diag:E_1_center_in_ecat_1}
        \begin{array}{c}
\xymatrix @C=0.3in @R=0.25in{
    & \ec[\FZ_2(\CC)]{\hbcat_1(\ec[\CC]{\CM})} \times \ec[\CC]{\CM} \ar@/^3ex/[ddr]^{\ec{\circledast}} \\
            \rtwocell<\omit>{\quad \ec{P^0} \times 1}& \ec[\CA]{\CL} \times \ec[\CC]{\CM} \ar[dr]^{\ec{\odot}} \ar[u]|{\ec{P} \times 1_{\ec[\CC]{\CM}}} \rtwocell<\omit>{\ec{\rho}} & \\
            \ast \times \ec[\CC]{\CM} \ar[rr] \ar@/^3ex/[uur]^{\ec{\one}\times 1_{\ec[\CC]{\CM}}} \ar[ur]^{\ec{\one_\CL} \times 1_{\ec[\CC]{\CM}}} \rrtwocell<\omit>{<-2.5> \ec{u}}  & & \ec[\CC]{\CM}
}
\end{array}
=
\begin{array}{c}
\xymatrix @C=-0.2in{
    & \ec[\FZ_2(\CC)]{\hbcat_1(\ec[\CC]{\CM})} \times \ec[\CC]{\CM} \ar[dr]^{\ec{\circledast}} \\
    \ast \times \ec[\CC]{\CM} \ar[ur]^{\ec{\one} \times 1_{\ec[\CC]{\CM}}} \ar[rr] \rrtwocell<\omit>{<-3> \ec{v}} & & \ec[\CC]{\CM}
}
\end{array}.
\ee}
are equal.
\bit
\item Since the underlying functor $\odot$ of $\ec{\odot}$ is a monoidal functor, for any $x \in \ec[\CA]{\CL}$, the natural isomorphism
\begin{multline*}
\beta_{m,x} : m \otimes (x \odot \one_\CM) \xrightarrow{u_m^{-1} \otimes 1} (\one_\CL \odot m) \otimes (x \odot \one_\CM) \simeq x \odot m \\
\simeq (x \odot \one_\CM) \otimes (\one_\CL \odot m) \xrightarrow{1 \otimes u_m} (x \odot \one_\CM) \otimes m
\end{multline*}
is a half-braiding on $x \odot \one_\CM \in \FZ_1(\CM)$ (see Example \ref{expl:centers_in_cat}). It is clear that $\ec{\beta_{-,x}}$ is a $\CC$-natural isomorphism because it is the composition of 2-isomorphisms in $\ecat$ and the background changing natural transformation is identity. Hence $x \odot \one_\CM$ together with $\ec{\beta_{-,x}}$ is an object in $\ec[\FZ_2(\CC)]{\hbcat_1(\ec[\CC]{\CM})}$. Then we define an enriched monoidal functor $\ec{P} : \ec[\CA]{\CL} \to \ec[\FZ_2(\CC)]{\hbcat_1(\ec[\CC]{\CM})}$ as follows:
\bit
\item The background changing functor $\hat P$ is given by $- \hodot \one_\CC : \CA \to \FZ_2(\CC)$, which is a braided monoidal functor (see Example \ref{expl:centers_in_cat}).
\item The object $P(x) \in \ec[\FZ_2(\CC)]{\hbcat_1(\ec[\CC]{\CM})}$ is defined by the object $x \odot \one_\CM$ together with the half-braiding $\ec{\beta_{-,x}}$.
\item The morphism $\ec{P}_{x,y} : \ec[\CA]{\CL}(x,y) \hodot \one_\CC \to [P(x),P(y)]$ is induced by the morphism
\[
\ec[\CA]{\CL}(x,y) \hodot \one_\CC \xrightarrow{1 \hat{\odot} 1_{\one_\CM}} \ec[\CA]{\CL}(x,y) \hodot \ec[\CC]{\CM}(\one_\CM,\one_\CM) \xrightarrow{\ec{\odot}} \ec[\CC]{\CM}(x \odot \one_\CM,y \odot \one_\CM)
\]
and the universal property of $[P(x),P(y)]$. In other words, $\ec{P}_{x,y}$ is the unique morphism rendering the following diagram
\[
\xymatrix{
\ec[\CA]{\CL}(x,y) \hodot \one_\CC \ar[r]^{\ec{P}_{x,y}} \ar[d]^{1 \hat{\odot} 1_{\one_\CM}} & [P(x),P(y)] \ar[d] \\
\ec[\CA]{\CL}(x,y) \hodot \ec[\CC]{\CM}(\one_\CM,\one_\CM) \ar[r]^{\ec{\odot}} & \ec[\CC]{\CM}(x \odot \one_\CM,y \odot \one_\CM)
}
\]
commutative.
\item The monoidal structure is induced by that of $\ec{\odot}$.
\eit
\item The enriched monoidal natural isomorphism $\ec{\rho} : \ec{\circledast} \circ (\ec{P} \times 1) \Rightarrow \ec{\odot}$ is defined by the background changing natural transformation
\[
\hat \rho_{a,c} : (a \hodot \one_\CC) \hotimes c \xrightarrow{1 \otimes \hat u_c^{-1}} (a \hodot \one_\CC) \hotimes {(\one_\CA \hodot c)} \xrightarrow{\simeq} a \hodot c
\]
and the underlying natural transformation
\[
\rho_{x,m} : (x \odot \one_\CM) \otimes m \xrightarrow{1 \otimes u_m^{-1}} (x \odot \one_\CM) \otimes (\one_\CL \odot m) \xrightarrow{\simeq} x \odot m .
\]
\eit
It is not hard to check that the equation \eqref{diag:E_1_center_in_ecat_1} holds.

\smallskip

Let $(\ec{Q_i}, \ec{\rho_i})$, $i=1,2$, be two pairs such that the similar equations as depicted by \eqref{diag:E_1_center_in_ecat_1} hold. We only need to show that there exists a unique enriched monoidal natural isomorphism $\ec{\alpha}: \ec{Q_1} \Rightarrow \ec{Q_2}$ such that
    \begin{align}
   \begin{array}{c}
       \xymatrix @R=0.3in @C=0.3in{
           \ec[\FZ_2(\CC)]{\hbcat_1(\ec[\CC]{\CM})} \times \ec[\CC]{\CM} \ar@/^2ex/[dr]^-{\ec{\circledast}} \rtwocell<\omit>{<5> \;\ec{\rho_2}} & \\
        \ec[\CA]{\CL} \times \ec[\CC]{\CM} \utwocell<4.5>^{\ec{Q_1} \times 1 \quad}_{\quad \ec{Q_2} \times 1}{\ec{\alpha} \times 1} \ar[r]_-{\ec{\odot}}& \ec[\CC]{\CM}
    }
    \end{array} & =
    \begin{array}{c}
       \xymatrix @R=0.3in @C=0.3in{
        \ec[\FZ_2(\CC)]{\hbcat_1(\ec[\CC]{\CM})} \times \ec[\CC]{\CM} \ar@/^2ex/[dr]^-{\ec{\circledast}} \rtwocell<\omit>{<5> \;\ec{\rho_1}} & \\
        \ec[\CA]{\CL} \times \ec[\CC]{\CM} \ar[u]^-{\ec{Q_1} \times 1} \ar[r]_-{\ec{\odot}}& \ec[\CC]{\CM}
    }
    \end{array}. \label{diag:E_1_center_in_ecat_3}
\end{align}

Since $\FZ_2(\CC)$ is the $E_0$-center of $\CC$ in $\Algn[2](\cat)$ (see Example \ref{expl:centers_in_cat}), there exists a unique monoidal natural isomorphism $\hat{\alpha}: \hat{Q}_1 \Rightarrow \hat{Q}_2$ such that  
\begin{align*}
\begin{array}{c}
       \xymatrix @R=0.3in @C=0.5in{
           \FZ_2(\CC) \times \CC \ar@/^2ex/[dr]^-{\hat{\circledast}} \rtwocell<\omit>{<4.5> \;\hat{\rho}_2} & \\
        \CA \times \CC \utwocell<4.5>^{\hat{Q}_1 \times 1 \quad}_{\quad \hat{Q}_2 \times 1}{\hat{\alpha} \times 1} \ar[r]_-{\hat{\odot}}& \CC
    }
    \end{array} =
    \begin{array}{c}
       \xymatrix @R=0.3in @C=0.4in{
        \FZ_2(\CC) \times \CC \ar@/^2ex/[dr]^-{\hat{\circledast}} \rtwocell<\omit>{<4.5> \;\hat{\rho}_1} & \\
        \CA \times \CC \ar[u]^-{\hat{Q}_1 \times 1} \ar[r]_-{\hat{\odot}}& \CC
    }
    \end{array}.
\end{align*}
Define $\alpha_a: \one_\CC \to [Q_1(a), Q_2(a)]$ to be the unique morphism rendering the following diagram 
\begin{align*}
    \begin{array}{c}
    \xymatrix @R=0.2in{
        \one_\CC  \ar[r]^-{\alpha_a} \ar[d]_{(\rho_2)_{(a,\one)}^{-1} \circ (\rho_1)_{(a,\one)}} &  [Q_1(a), Q_2(a)] \ar[d]\\
        \ec[\CC]{\CM}(Q_1(a) \otimes \one_\CM,Q_2(a) \otimes \one_\CM) \ar[r]^-{\simeq} & \ec[\CC]{\CM}(Q_1(a), Q_2(a)) 
    }
    \end{array}
\end{align*}
commutative. Then the enriched natural isomorphism $\ec{\alpha}$ define by $(\hat{\alpha}, \{\alpha_a\})$ is the unique enriched natural isomorphism such that the equation \eqref{diag:E_1_center_in_ecat_3} holds.

It is routine to check that the braiding structure of $\ec[\FZ_2(\CC)]{\hbcat_1(\ec[\CC]{\CM})}$ coincides with the $E_2$-algebra structure of the $E_1$-center of $\ec[\CC]{\CM}$ in $\ecat$ induced from the universal property. 
\end{proof}

Combining Corollary\,\ref{cor:ZHCM} and Theorem\,\ref{thm:Drinfeld_center=E1_center}, we obtain a corollary. 
\begin{cor} \label{cor:E1_center=Z2Z2}
Let $\CC$ be a braided monoidal category and $\CM$ a strongly unital monoidal left $\overline{\CC}$-module that is enriched in $\overline{\CC}$.  If $\bc[\CC]{\CM}$ satisfies the condition \eqref{cond:star_2}, we have 
\[
\FZ_1(\bc[\CC]{\CM})=\bc[\FZ_2(\CC)]{\FZ_2(\varphi)}.
\]
\end{cor}

\begin{rem} \label{rem:E1_center_physical}
Corollary \ref{cor:E1_center=Z2Z2} has important applications in physics. When $\CC$ is a non-degenerate braided fusion category and $\CM$ is an indecomposable multi-fusion left $\CC$-module defined by a braided monoidal functor $\varphi: \overline{\CC} \to \FZ_1(\CM)$ \cite{KZ18}, the condition \eqref{cond:star_2} is satisfied automatically. In this case, we obtain $\FZ_1(\bc[\CC]{\CM}) \simeq \FZ_2(\varphi)$. This recovers Corollary 5.4 in \cite{KZ18a}. This result has important applications in the theory of gapless boundaries of 2+1D topological orders \cite{KZ18b,KZ20,KZ21}. Moreover, for an indecomposable multi-fusion category $\CA$ and a finite semisimple left $\CA$-module $\CL$, then we have 
\[
\FZ_1(\FZ_0(\bc[\CA]{\CL}))\simeq \bk, 
\] 
i.e., the center of a center is trivial. This is the mathematical manifestation of a physical fact: ``the bulk of a bulk is trivial''.
\end{rem}


\section{Enriched symmetric monoidal categories} \label{sec:enriched_symmetric_monoidal_categories}

In this section, we study enriched symmetric monoidal categories, canonical construction and the $E_2$-centers of enriched braided monoidal categories.

\subsection{Definitions and canonical construction}

\begin{defn}
An enriched braided monoidal category $\ec[\CA]{\CL}$ is called an \textit{enriched symmetric monoidal category} if the underlying braided monoidal category $\CL$ is symmetric. 
\end{defn}

\begin{defn}
An enriched braided monoidal functor $\ec{F} \colon \ec[\CA]{\CL} \to \ec[\CB]{\CM}$ between two enriched symmetric monoidal categories $\ec[\CA]{\CL}$ and $\ec[\CB]{\CM}$ is also called an \textit{enriched symmetric monoidal functor}.
\end{defn}

\begin{defn}
Let $\CA$ be a symmetric monoidal category and $\CL$ be a braided monoidal left $\CA$-\oplax module. If $\CL$ is symmetric, we say $\CL$ is a \textit{symmetric monoidal left $\CA$-\oplax module}.
\end{defn}

\begin{defn}
A braided monoidal $\hat F$-lax functor $F \colon \CL \to \CM$ between two symmetric monoidal categories $\CL$ and $\CM$ is also called a \textit{symmetric monoidal $\hat F$-lax functor}.
\end{defn}

Let $\ecat^{\mathrm{sym}}$ be a symmetric monoidal full sub-2-category of $\ecat^{\mathrm{br}}$ consisting of enriched symmetric monoidal categories. Let $\lmod^{\mathrm{sym}}$ be the symmetric monoidal full sub-2-category of $\lmod^{\mathrm{br}}$ consisting of pairs $(\CA,\CL)$ such that $\CL$ is strongly unital symmetric monoidal left $\CA$-\oplax module that is enriched in $\CA$.
We obtain a corollary of Theorem \ref{thm:canonical_construction_2-functor_braided}.


\begin{cor} 
The canonical construction induces a symmetric monoidal 2-functor $\lmod^{\mathrm{sym}} \to \ecat^{\mathrm{sym}}$. Moreover, this 2-functor is locally isomorphic. 
\end{cor}

\begin{proof}
    Let $(\CA, \CL) \in \lmod^{\mathrm{br}}$. Then $\bc[\CA]{\CL}$ is an enriched monoidal category if and only if $\CL$ is a symmetric monoidal left $\CA$-\oplax module. Then Theorem \ref{thm:canonical_construction_2-functor_braided} implies that the canonical construction induces a locally isomorphic symmetric monoidal 2-functor from $\lmod^{\mathrm{sym}}$ to $\ecat^{\mathrm{sym}}$.
\end{proof}

\begin{defn}
Let $\ec[\CA]{\CL}$ be an enriched braided monoidal category. The M\"{u}ger center $\ec[\CA]{\hbcat_2(\ec[\CA]{\CL})}$ of $\ec[\CA]{\CL}$ is defined by the subcategory of $\ec[\CA]{\CL}$ consisting of the transparent objects in $\CL$ and the same hom spaces as those in $\ec[\CA]{\CL}$ (i.e., $\ec[\CA]{\hbcat_2(\ec[\CA]{\CL})} \coloneqq \ec[\CA]{\FZ_2(\CL)}$).
\end{defn}

\begin{rem}
The M\"{u}ger center $\ec[\Set]{\hbcat_2(\ec[\Set]{\CL})}$ of a $\Set$-enriched braided monoidal category $\ec[\Set]{\CL}$ is precisely the usual M\"{u}ger center $\FZ_2(\CL)$ \cite{Mueg03a} of the braided monoidal category $\CL$.
\end{rem}

\begin{thm} \label{thm:mueger-center-canonical-construction}
Let $\CA$ be a symmetric monoidal category and $\CL$ be a braided monoidal left $\CA$-module that is enriched in $\CA$. Then $\FZ_1(\CL)$ is also enriched in $\CA$ and we have $\ec[\CA]{\hbcat_2(\bc[\CA]{\CL})} = \bc[\CA]{\FZ_2(\CL)}$. 
\end{thm}

\subsection{The \texorpdfstring{$E_2$}{E2}-centers of enriched braided monoidal categories in \texorpdfstring{$\ecat$}{ECat}}

In this subsection we prove that the M\"{u}ger center $\ec[\CC]{\hbcat_2(\ec[\CC]{\CM})}$ of an enriched braided monoidal category $\ec[\CC]{\CM}$ is its $E_2$-center in $\ecat$.

\medskip

Suppose $\ec[\CE]{\CL}$ is an enriched braided monoidal category. A left unital action of $\ec[\CA]{\CL}$ on $\ec[\CC]{\CM}$ is an enriched braided monoidal functor $\ec{\odot}$ and an enriched monoidal natural isomorphism $\ec{u}$ as depicted in the following diagram:
\be \label{diag:E2_left_unital_action}
\begin{array}{c}
\xymatrix{
 & \ec[\CE]{\CL} \times \ec[\CC]{\CM} \ar[dr]^{\ec{\odot}} \\
{*} \times \ec[\CC]{\CM} \ar[ur]^{\ec{\one} \times 1} \ar[rr] \rrtwocell<\omit>{<-3>\;\ec{u}} & & \ec[\CC]{\CM}
}
\end{array} .
\ee

\begin{expl}
There is an obvious left unital action of the M\"{u}ger center $\ec[\CC]{\hbcat_2(\ec[\CC]{\CM})}$ on $\ec[\CC]{\CM}$:
\[
\xymatrix{
 & \ec[\CC]{\hbcat_2(\ec[\CC]{\CM})} \times \ec[\CC]{\CM} \ar[dr]^{\ec{\otimes}} \\
{*} \times \ec[\CC]{\CM} \ar[ur]^{\ec{\one} \times 1} \ar[rr] \rrtwocell<\omit>{<-3>\;\ec{v}} & & \ec[\CC]{\CM}
}
\]
The enriched braided monoidal functor $\ec{\otimes}$ is given by the tensor product of $\ec[\CC]{\CM}$. The background changing natural isomorphism of $\ec{v}$ is given by the left unitor of $\CC$ and the underlying natural isomorphism of $\ec{v}$ is given by the left unitor of $\CM$.
\end{expl}

\begin{thm} \label{thm:Mueger_center=E2_center}
The M\"{u}ger center $\ec[\CC]{\hbcat_2(\ec[\CC]{\CM})}$ is the $E_2$-center of $\ec[\CC]{\CM}$ in $\ecat$.
\end{thm}

\pf
Suppose $(\ec{\odot},\ec{u})$ is a left unital action of an enriched braided monoidal category $\ec[\CE]{\CL}$ on $\ec[\CC]{\CM}$ as depicted in the diagram \eqref{diag:E2_left_unital_action}. First we show that there is an enriched braided monoidal functor $\ec{P} \colon \ec[\CE]{\CL} \to \ec[\CC]{\hbcat_2(\ec[\CC]{\CM})}$ and an enriched monoidal natural isomorphism $\ec{\rho}$ such that the following pasting diagrams
{\small
\be\label{diag:E_2_center_in_ecat_1}
        \begin{array}{c}
\xymatrix @C=0.4in @R=0.25in{
    & \ec[\CC]{\hbcat_2(\ec[\CC]{\CM})} \times \ec[\CC]{\CM} \ar@/^3ex/[ddr]^{\ec{\otimes}} \\
            \rtwocell<\omit>{\quad \ec{P^0} \times 1}& \ec[\CE]{\CL} \times \ec[\CC]{\CM} \ar[dr]^{\ec{\odot}} \ar[u]|{\ec{P} \times 1_{\ec[\CC]{\CM}}} \rtwocell<\omit>{\ec{\rho}} & \\
            \ast \times \ec[\CC]{\CM} \ar[rr] \ar@/^3ex/[uur]^{\ec{\one}\times 1_{\ec[\CC]{\CM}}} \ar[ur]^{\ec{\one_\CL} \times 1_{\ec[\CC]{\CM}}} \rrtwocell<\omit>{<-2.5> \ec{u}}  & & \ec[\CC]{\CM}
}
\end{array}
=
\begin{array}{c}
\xymatrix @C=-0.1in{
    & \ec[\CC]{\hbcat_2(\ec[\CC]{\CM})} \times \ec[\CC]{\CM} \ar[dr]^{\ec{\otimes}} \\
    \ast \times \ec[\CC]{\CM} \ar[ur]^{\ec{\one} \times 1_{\ec[\CC]{\CM}}} \ar[rr] \rrtwocell<\omit>{<-3> \ec{v}} & & \ec[\CC]{\CM}
}
\end{array}.
\ee}
are equal.
\bit
\item Since the underlying functor $\odot$ of $\ec{\odot}$ is a braided monoidal functor, $x \odot \one_\CM \in \FZ_2(\CM)$ is transparent for any $x \in \ec[\CE]{\CL}$ (see Example \ref{expl:centers_in_cat}). Then we define an enriched monoidal functor $\ec{P} \colon \ec[\CE]{\CL} \to \ec[\CC]{\hbcat_2(\ec[\CC]{\CM})}$ by $\ec{P} \coloneqq \ec{(- \odot \one_\CM)}$. More explicitly:
\bit
\item The background changing functor $\hat P$ is given by $- \hodot \one_\CC \colon \CE \to \CC$.
\item The object $P(x) \in \ec[\CC]{\hbcat_2(\ec[\CC]{\CM})}$ is defined by the object $x \odot \one_\CM \in \FZ_2(\CM)$.
\item The morphism $\ec{P}_{x,y} \colon \ec[\CE]{\CL}(x,y) \hodot \one_\CC \to \ec[\CC]{\CM}(P(x),P(y))$ is induced by the morphism
\[
\ec[\CE]{\CL}(x,y) \hodot \one_\CC \xrightarrow{1 \hat{\odot} 1_{\one_\CM}} \ec[\CE]{\CL}(x,y) \hodot \ec[\CC]{\CM}(\one_\CM,\one_\CM) \xrightarrow{\ec{\odot}} \ec[\CC]{\CM}(x \odot \one_\CM,y \odot \one_\CM) .
\]
\item The braided monoidal structure is induced by that of $\ec{\odot}$.
\eit
\item The enriched monoidal natural isomorphism $\ec{\rho} \colon \ec{\otimes} \circ (\ec{P} \times 1) \Rightarrow \ec{\odot}$ is defined by the background changing natural transformation
\[
\hat \rho_{a,c} \colon (a \hodot \one_\CC) \hotimes c \xrightarrow{1 \otimes \hat u_c^{-1}} (a \hodot \one_\CC) \hotimes {(\one_\CA \hodot c)} \xrightarrow{\simeq} a \hodot c
\]
and the underlying natural transformation
\[
\rho_{x,m} \colon (x \odot \one_\CM) \otimes m \xrightarrow{1 \otimes u_m^{-1}} (x \odot \one_\CM) \otimes (\one_\CL \odot m) \xrightarrow{\simeq} x \odot m .
\]
\eit
It is not hard to check that the equation \eqref{diag:E_2_center_in_ecat_1} holds.

\smallskip

Let $(\ec{Q_i}, \ec{\rho_i})$, $i=1,2$, be two pairs such that the similar equations as depicted by \eqref{diag:E_2_center_in_ecat_1} hold. We only need to show that there exists a unique enriched monoidal natural isomorphism $\ec{\alpha} \colon \ec{Q_1} \Rightarrow \ec{Q_2}$ such that
    \begin{align}
   \begin{array}{c}
       \xymatrix @R=0.3in @C=0.3in{
           \ec[\CC]{\hbcat_2(\ec[\CC]{\CM})} \times \ec[\CC]{\CM} \ar@/^2ex/[dr]^-{\ec{\otimes}} \rtwocell<\omit>{<5> \;\ec{\rho_2}} & \\
        \ec[\CE]{\CL} \times \ec[\CC]{\CM} \utwocell<4.5>^{\ec{Q_1} \times 1 \quad}_{\quad \ec{Q_2} \times 1}{\ec{\alpha} \times 1} \ar[r]_-{\ec{\odot}}& \ec[\CC]{\CM}
    }
    \end{array} & =
    \begin{array}{c}
       \xymatrix @R=0.3in @C=0.3in{
        \ec[\CC]{\hbcat_2(\ec[\CC]{\CM})} \times \ec[\CC]{\CM} \ar@/^2ex/[dr]^-{\ec{\otimes}} \rtwocell<\omit>{<5> \;\ec{\rho_1}} & \\
        \ec[\CE]{\CL} \times \ec[\CC]{\CM} \ar[u]^-{\ec{Q_1} \times 1} \ar[r]_-{\ec{\odot}}& \ec[\CC]{\CM}
    }
    \end{array}. \label{diag:E_2_center_in_ecat_3}
\end{align}

Since $\CC$ itself is the $E_0$-center of $\CC$ in $\Algn[3](\cat)$ (see Example \ref{expl:centers_in_cat}), there exists a unique monoidal natural isomorphism $\hat{\alpha} \colon \hat{Q}_1 \Rightarrow \hat{Q}_2$ such that  
\begin{align*}
\begin{array}{c}
       \xymatrix @R=0.3in @C=0.7in{
           \CC \times \CC \ar@/^2ex/[dr]^-{\hat{\otimes}} \rtwocell<\omit>{<4.5> \;\hat{\rho}_2} & \\
        \CE \times \CC \utwocell<4.5>^{\hat{Q}_1 \times 1 \quad}_{\quad \hat{Q}_2 \times 1}{\hat{\alpha} \times 1} \ar[r]_-{\hat{\odot}}& \CC
    }
    \end{array} =
    \begin{array}{c}
       \xymatrix @R=0.3in @C=0.5in{
        \CC \times \CC \ar@/^2ex/[dr]^-{\hat{\otimes}} \rtwocell<\omit>{<4.5> \;\hat{\rho}_1} & \\
        \CE \times \CC \ar[u]^-{\hat{Q}_1 \times 1} \ar[r]_-{\hat{\odot}}& \CC
    }
    \end{array}.
\end{align*}
Define
\[
\alpha_a \coloneqq \bigl( Q_1(a) \simeq Q_1(a) \otimes \one_\CM \xrightarrow{(\rho_1)_{a,\one}} a \odot \one \xrightarrow{(\rho_2)_{a,\one}^{-1}} Q_2(a) \otimes \one \simeq Q_2(a) \bigr) .
\]
Then the enriched natural isomorphism $\ec{\alpha}$ define by $(\hat{\alpha}, \{\alpha_a\})$ is the unique enriched natural isomorphism such that the equation \eqref{diag:E_2_center_in_ecat_3} holds. 
\end{proof}

%
%

By Theorem\ \ref{thm:mueger-center-canonical-construction}, we obtain the following corollary. 
\begin{cor} \label{cor:E2_center_of_canonical_construction}
Let $\CA$ be a symmetric monoidal category and $\CL$ be a braided monoidal left $\CA$-module that is enriched in $\CA$. Then we have $\FZ_2(\bc[\CA]{\CL})=\bc[\CA]{\FZ_2(\CL)}$. 
\end{cor}

\begin{expl}
Let $\CE$ be a symmetric fusion category, i.e., $\CE=\rep(G)$ or $\rep(G,z)$ for a finite group $G$. Then the canonical construction $\bc[\CE]{\CE}$ is an example of enriched symmetric monoidal category. We have $\FZ_2(\bc[\CE]{\CE})\simeq \bc[\CE]{\CE}$. 
\end{expl}

\begin{rem}
When $\CA$ is a non-degenerate braided fusion category and $\CL$ is an indecomposable multi-fusion left $\overline{\CA}$-module \cite{KZ18}, we have 
$\FZ_2(\FZ_1(\bc[\CA]{\CL}))\simeq \bk$.  
\end{rem}

\begin{rem} \label{rem:En_centers}
In this work, we compute the $E_n$-centers of $E_n$-monoidal enriched categories for $n=0,1,2$. These results should generalize to much more general context, for example, for enriched $E_n$-monoidal higher categories (see \cite[Section\,5.2]{KZ22b}). 
\end{rem}

\bibliography{Top}

\end{document}